\documentclass[final]{siamltex}

\usepackage{amsmath}
\usepackage{amssymb}
\usepackage{graphicx}
\usepackage{epstopdf}
\usepackage{algorithm}
\usepackage{algorithmic,color}
\usepackage{multirow}

\newcommand{\zb}[1]{\ensuremath{\boldsymbol{#1}}}

\newcommand{\ww}{28mm}
\newcommand{\www}{35mm}

\newtheorem{remark}[theorem]{Remark}

\title{Linkage between Piecewise Constant Mumford-Shah model and ROF model and its virtue in image segmentation}

\author{Xiaohao Cai\thanks{
 Department of Applied Mathematics and Theoretical Physics (DAMTP),
 University of Cambridge, Cambridge CB3 0WA, UK.
 Mullard Space Science Laboratory (MSSL),  University College London, Surrey RH5 6NT, UK.
  Email:  x.cai@ucl.ac.uk. 
  XC acknowledges support from the EPSRC grant EP/K032208/1, Issac Newton Trust (University of Cambridge),
  and the Leverhulme grant.}
  \and
  Raymond Chan\thanks{ 
  Department of Mathematics, College of Science, City University of Hong Kong. Email:  rchan.sci@cityu.edu.hk.
  RC acknowledges support from the HKRGC Grants No. CUHK14306316, CityU Grant 9380101, CRF Grant C1007-15G, AoE/M-05/12.}
  \and
  Carola-Bibiane Sch{\"o}nlieb\thanks{
  Department of Applied Mathematics and Theoretical Physics (DAMTP), University of Cambridge,
  Cambridge CB3 0WA, UK. Email: cbs31@cam.ac.uk. 
  CBS acknowledges support from the Leverhulme Trust project ``Breaking the non-convexity barrier", the EPSRC grant EP/K032208/1, the EPSRC grant EP/M00483X/1, the EPSRC Centre Nr. EP/N014588/1, the RISE projects ChiPS and NoMADS, the Cantab Capital Institute for the Mathematics of Information, and the Alan Turing Institute.}
  \and
  Gabriele Steidl\thanks{
  Department of Mathematics, University of Kaiserslautern, Kaiserslautern, Germany. Email: steidl@mathematik.uni-kl.de.}
  \and
  Tieyong Zeng\thanks{
   Department of Mathematics, The Chinese University of Hong Kong,
  Shatin, Hong Kong. Email: zeng@math.cuhk.edu.hk.
  TZ acknowledges support from the National Natural Science Foundation of China under
Grant 11671002, CUHK start-up, CUHK DAG 4053296, 4053342, and the RGC 14300219.
}
  }

\begin{document}

\maketitle

\begin{abstract}
The piecewise constant Mumford-Shah (PCMS) model and the Rudin-Osher-Fatemi (ROF) model are two important variational models
in image segmentation and image restoration, respectively.
In this paper, we explore a linkage between these models. 
We prove that for the two-phase segmentation problem a partial minimizer of the PCMS model 
can be obtained by thresholding the minimizer of the ROF model.
A similar linkage is still valid for multiphase segmentation under specific assumptions.
Thus it opens a new segmentation paradigm: image segmentation can be done via image restoration plus thresholding.
This new paradigm, which circumvents the innate non-convex property of the PCMS model, therefore improves
the segmentation performance in both efficiency (much faster than state-of-the-art methods based on PCMS model,
particularly when the phase number is high)
and effectiveness (producing segmentation results with better quality) due to the flexibility of the ROF model in tackling degraded images,
such as noisy images, blurry images or images with information loss. As a by-product
of the new paradigm, we derive a novel segmentation method,
called {\it thresholded}-ROF (T-ROF) method, to illustrate the virtue of managing image segmentation through image restoration techniques.
The convergence of the T-ROF method is proved, and elaborate experimental results and 
comparisons are presented.
\end{abstract}


%\begin{AMS}52A41, 65D15, 68W40, 90C25, 90C90\end{AMS}

%-------------------------------------------------------------------
\section{Introduction}\label{sec:introduction}
%-------------------------------------------------------------------
Image segmentation aims to group parts of a given image with similar characteristics together, while
image restoration intends to remove image degradations such as noise, blur or occlusions.
%In computer vision and image processing, image segmentation and image restoration
%are fundamental and challenging tasks and serve for example as a preliminary or postprocessing
%step for object recognition and interpretation, with ubiquitous applications in various types of imaging.
The piecewise constant Mumford-Shah (PCMS) model (nonconvex, a special case of the Mumford-Shah model \cite{MS89})
and the Rudin-Osher-Fatemi (ROF) model (convex, \cite{ROF92}) are two of the most famous variational models in the
research areas of image segmentation and restoration, respectively.
Following earlier works e.g. \cite{MS89,ROF92}, in this paper, we show a linkage between the PCMS and ROF models,
which gives rise to a new image segmentation paradigm:
manipulating image segmentation through image restoration plus thresholding.

Let us first recall the PCMS and ROF models. 
Throughout this paper, let $\Omega \subset \mathbb{R}^2$ be a bounded, open set with Lipschitz boundary,
and $f: \Omega \rightarrow [0,1]$ be a given (degraded) image.  
In 1989 Mumford and Shah  \cite{MS89} proposed solving segmentation problems by minimizing over
$\Gamma \subset \Omega$ and $u \in H^{1}(\Omega\backslash\Gamma)$ the energy functional
\begin{equation} \label{ms}
E_{\rm MS}(u, \Gamma; \Omega)={\cal H}^{1}(\Gamma)+ \lambda^\prime
\int_{\Omega\backslash\Gamma}|\nabla u|^2 dx+ \lambda \int_{\Omega}(u-f)^2 dx, \quad \lambda^\prime, \lambda > 0,
\end{equation}
where ${\cal H}^1$ denotes the one-dimensional Hausdorff measure in $\mathbb R^2$. 
The functional $E_{\rm MS}$ contains three terms:
the penalty term on the length of $\Gamma$, the $H^1$ semi-norm that enforces  the smoothness of $u$ in $\Omega\backslash\Gamma$,
and the data fidelity term controlling the distance of $u$ to the given image $f$.
Related approaches in a spatially discrete setting were proposed in \cite{BZ87,GG84}.
An early attempt to solve the challenging task of finding a minimizer of the non-convex and non-smooth Mumford-Shah 
functional \eqref{ms} was done by approximating it using a sequence of simpler elliptic problems, see \cite{AT90}
for the so-called Ambrosio-Tortorelli approximation.
Many approaches to simplify model \eqref{ms} were meanwhile proposed in the literature, for example, in \cite{PCBC09a}, a
convex relaxation of the model was suggested. 
Another important simplification is to restrict its solution to 
be piecewise constant, which leads to the so-called PCMS model. 

{\bf PCMS model.}
This model is based on the restriction $\nabla u = 0$ on $\Omega\backslash \Gamma$,
which results in 
\begin{equation}  \label{pcms-org}
E_{\rm PCMS}(u, \Gamma; \Omega)={\cal H}^{1}(\Gamma) + \lambda \int_{\Omega}(u-f)^2 dx.
\end{equation}
Assuming that $\Omega=\bigcup_{i=0}^{K-1}\Omega_i$ with pairwise disjoint sets $\Omega_i$ and constant functions $u(x) \equiv m_i$ 
on $\Omega_i$, $i = 0,\ldots K-1$, model \eqref{pcms-org} can be rewritten as
\begin{equation} \label{pcms}
E_{\rm PCMS}({\bf \Omega},{\zb m})=
\frac{1}{2}\sum_{i=0}^{K-1}{\rm Per}(\Omega_i; \Omega)
+ \lambda \sum_{i=0}^{K-1}\int_{\Omega_i}(m_i-f)^2dx,
\end{equation}
where ${\bf \Omega} := \{\Omega_i\}_{i=0}^{K-1}$, ${\zb m} := \{m_i\}_{i=0}^{K-1}$, and 
${\rm Per}(\Omega_i;\Omega)$ denotes the perimeter of $\Omega_i$ in $\Omega$.
If the number of phases is two, i.e. $K=2$, the PCMS model is the model of the active contours without edges
(Chan-Vese model) \cite{CV01-2}, 
\begin{equation}\label{chan-vese}
E_{\rm CV}(\Omega_1,m_0,m_1) = {\rm Per}(\Omega_1;\Omega) +
\lambda \Big(\int_{\Omega_1}
(m_1-f)^2 \, dx + \int_{\Omega\backslash\Omega_1} (m_0-f)^2 \, dx\Big).
\end{equation}
In \cite{CV01-2} the authors proposed to solve \eqref{chan-vese}, where
it can easily get stuck in local minima. To overcome this drawback, a convex relaxation
approach was proposed in \cite{CEN06}.
More precisely, it was shown that a global minimizer
of $E_{\rm CV}(\cdot, m_0,m_1)$ for fixed $m_0, m_1$ can be found by solving
\begin{equation}\label{chan-vese-convex}
\bar{u} = \underset{u\in BV(\Omega)}{\rm argmin} \Big\{TV(u) + \lambda \int_{\Omega} \big( (m_0-f)^2 - (m_1-f)^2 \big ) u\, dx \Big\},
\end{equation}
and setting $\Omega_1 := \{x \in \Omega: \bar{u}(x) > \rho\}$ for any choice of $\rho \in [0, 1)$,
see also \cite{BPV91,BEVTO07}. Note that the first term of \eqref{chan-vese-convex} is known as the total variation ($TV$) 
and the space $BV$ is the space of functions of bounded variation, see Section 2 for the definition.
In other words, \eqref{chan-vese-convex} is a tight relaxation of the Chan-Vese model with fixed $m_0$ and $m_1$.
For the convex formulation of the full model \eqref{chan-vese}, see \cite{BCB12}.

There are many other approaches for two-phase image segmentation based on the Chan-Vese model and its convex version,
see e.g. \cite{ZMSM08,BEVTO07,DCS10,BCPSS17}. 
In particular, a  hybrid level set method was proposed
in \cite{ZMSM08}, which replaces the first term of (\ref{chan-vese}) by a boundary feature map and the data fidelity terms
in (\ref{chan-vese}) by the difference between the given image $f$ and a fixed threshold chosen by a
user or a specialist. Method \cite{ZMSM08} was used in medical image segmentation.
However, since every time it needs the user to choose a proper threshold for its model, 
it is not automatic and thus its applications are restricted. 
In \cite{BEVTO07}, the $TV$ term
of (\ref{chan-vese-convex}) was replaced by a weighted $TV$ term which
helps the new model to capture much more important geometric properties.
In \cite{DCS10}, the $TV$ term of \eqref{chan-vese-convex} was replaced by
a wavelet frame decomposition operator which, similar to the model in
\cite{BEVTO07},  can also capture important geometric
properties. 
Nevertheless, for its solution $u$, no similar conclusions as the ones in \cite{CEN06} can be 
addressed; that is, there is no theory to support that its segmentation result $\Omega_1  = \{x: u(x) > \rho\}$ for $\rho \in [0, 1)$ 
is a solution as to some kind of objective functional. In \cite{BCPSS17}, the Chan-Vese model 
was extended for 3D biopores segmentation in tomographic images. 

In \cite{VC02}, Chan and Vese proposed a multiphase segmentation model based on the
PCMS model using level sets. However, this method can also  get stuck easily in local minima.
Convex (non-tight) relaxation approaches for the PCMS model were proposed, which are
basically focusing on solving
\begin{equation}
\min_{m_i, u_i\in [0,1]} \Big\{
\sum_{i=0}^{K-1}\int_{\Omega} |\nabla u_i|dx + \lambda \sum_{i=0}^{K-1}\int_{\Omega}(m_i-f)^2 u_idx \Big\},
\quad {\rm s.t.} \ \ \sum_{i=0}^{K-1} u_i =1.
\label{pcms-convex}
\end{equation}
For more details along this line, refer e.g. to \cite{BCCJKM11,C15,CFNSS15,LS12,LNZS10,PCCB09,YBTB10,ZGFN08} and the references therein.
We are interested in a relation between the PCMS model and the  ROF model for 
 image restoration which we introduce next.

{\bf ROF model.}
In 1992, Rudin, Osher and Fatemi \cite{ROF92} proposed the  variational model
\begin{equation}\label{rof}
\min_{u\in BV(\Omega)} \Big\{TV(u) + \frac{\mu}{2} \int_{\Omega} \big(u-f)^2dx \Big\}, \quad \mu >0.
\end{equation}
which has been studied extensively
in the literature, see e.g. \cite{Ch05,CNCP10,CEPY06} and references therein.

Actually, a subtle connection between image segmentation and image restoration has been raised in \cite{CCZ13}.
In detail,  a two-stage image segmentation method is proposed -- smoothing and thresholding (SaT) method --
which finds the solution of a convex variant of the Mumford-Shah model  in the first stage followed by a thresholding step in the second one.
The convex minimization functional in the first stage (the smoothing stage)
is the ROF functional \eqref{rof} plus an additional smoothing term $\int_\Omega | \nabla u|^2 \, dx$.
The SaT method is very efficient and flexible: it performs excellently for degraded images
(e.g. noisy and blurry images and images with information loss);
the minimizer from the first stage is unique; and one can change the number of phases $K$ without solving
the minimization functional again. The success of the SaT method indicates a new methodology 
for image segmentation: first smoothing and then thresholding. This approach was extended in \cite{CYZ13} 
for images corrupted by Poisson and Gamma noises, and in \cite{CCNZ15} to degraded color images.

{\bf Our contribution.}
In this paper we highlight a relationship between the PCMS model \eqref{pcms} and the ROF model \eqref{rof}.
We prove that thresholding the minimizer of the ROF model leads to a partial minimizer ({\it cf.} definition \eqref{partial_min}) 
of the PCMS model when $K=2$
(Chan-Vese model \eqref{chan-vese}), which remains true under specific assumptions when $K>2$.
This linkage between the PCMS model and the ROF model validates the effectiveness of our proposed SaT method
in \cite{CCZ13} for image segmentation. 
Due to the significance of the PCMS model and ROF model, respectively in image
segmentation and image restoration, this linkage  bridges  to some extent these two research areas
and  might serve as a motivation to improve and design better methods.
A direct  benefit is a new efficient segmentation method
-- {\it thresholded-ROF} (T-ROF) method -- proposed in this paper. 
The T-ROF method exactly follows the new paradigm to perform image segmentation through image restoration 
plus iterative thresholding, where these thresholds are selected automatically following certain rules.
This appears to be more sophisticated
than the SaT method \cite{CCZ13} which is based on $K$-means.
We emphasize that we just need to
solve the ROF model once, and our method gives optimal segmentation results akin to the PCMS model.
We prove the convergence of our T-ROF method regarding thresholds automatic selection.

On the one hand, the T-ROF method can be regarded as a special case of our proposed SaT method. However, it is directly obtained
from the linkage between the PCMS model and the ROF model discovered in this paper and thus is more theoretically justified.
Moreover, the strategy of choosing the thresholds automatically and optimally in the T-ROF method is not covered in
the SaT method in \cite{CCZ13}. The strategy makes our T-ROF method more effective particularly for degraded images whose phases
have close intensities. On the other hand, the T-ROF method inherits the advantages of the SaT method -- fast speed
and computational cost independent of the required number of phases $K$. In contrast, methods solving the 
PCMS model become computational demanding as the required number of phases increases.
Numerical experiments and detailed comparisons to the start-of-the-art methods
are presented to demonstrate the great performance of the proposed T-ROF method.
Partial results of this paper have been  presented in the conference paper \cite{CS13}. 

The paper is organized as follows. In Section \ref{sec:notation}, we introduce the basic notation.
The linkage between the PCMS and ROF models is presented in Section \ref{sec:link}. 
In Section \ref{sec:segmed},
we present our T-ROF model and provide an algorithm to solve it together with its
convergence analysis. 
In Section \ref{sec:numerics},
we demonstrate the performance of our T-ROF method on various synthetic and real-world images 
and compare it with 
the representative related segmentation methods \cite{LNZS10,PCCB09,YBTB10,HHMSS12,CCZ13}.
Conclusions are given in Section \ref{sec:conclusions}.

%-------------------------------------------------------------------
\section{Basic Notation}\label{sec:notation}
%-------------------------------------------------------------------
We briefly introduce the basic notation which will be used in the followings, see \cite{AFP00,ABM06} for more details.
By $BV(\Omega)$ we denote the space of functions of bounded variation defined on $\Omega$,  i.e.,
the Banach space of functions $u:\Omega \rightarrow [0,1]$
with finite norm $\|u\|_{BV} := \|u\|_{L^1(\Omega)} + TV(u)$, where
$$
TV(u) := \sup \big \{ \int_{\Omega} u(x)  {\rm div} \varphi \, dx: 
\varphi \in {\cal C}_c^1(\Omega,\mathbb R^2), \| \varphi \|_\infty \le 1 \big\}.
$$
The distributional first order derivative $Du$ of $u$
is a vector-valued Radon measure having total mass 
$|Du| (\Omega) = \int_{\Omega} |Du| dx = TV(u)$. In particular, we have for $u \in W^{1,1}(\Omega)$ that
$Du = \nabla u \in L^1(\Omega)$ so that in this case
$
TV(u) = \int_{\Omega} |\nabla u| \, dx.
$

 A ``set" is understood as a Lebesgue measurable
set in $\mathbb R^2$, where we mainly
consider equivalence classes of sets which are equal up to Lebesgue measure zero.
By $|A|$ we denote the Lebesgue measure of a set $A$.
For a Lebesgue measurable set $A \subset \Omega$, the {\it perimeter} of $A$ in $\Omega$ is
defined by
$$
{\rm Per}( A;\Omega) := TV(\chi_A),
$$
where $\chi_A$ is the characteristic function of $A$. Hence $A$ is of finite perimeter
if its characteristic function has finite bounded total variation. If $A$ has a ${\cal C}^1$ boundary,
then ${\rm Per}( A;\Omega)$ coincides with ${\cal H}^1(\partial A \cap \Omega)$.
For $A,B \subseteq \Omega$ the relation
\begin{gather} \label{per_prop_1}
{\rm Per}(A \cup B;\Omega) + {\rm Per}(A \cap B;\Omega) \le {\rm Per}(A;\Omega) + {\rm Per}(B;\Omega)
\end{gather}
holds true. 
We will also use the notation ${\rm Per}( A; \hat{\Omega})$ for non open set $\hat{\Omega}$, 
in the sense
$
{\rm Per}( A; \hat{\Omega}) = {\rm Per}( A; {\rm int}(\hat{\Omega})) 
$,
where ${\rm int}(\hat{\Omega})$ denotes the interior of $\hat{\Omega}$. 
We define the mean of $f$ on $A \subset \mathbb R^2$ by
\begin{gather*} \label{mean}
{\rm mean}_f (A) :=
\left\{
\begin{array}{ll}
\frac{1}{|A|} \int_A f \, dx, & \ {\rm if} \; |A| > 0,\\
0, &\ {\rm otherwise}.
\end{array}
\right.
\end{gather*}

%-------------------------------------------------------------------
\section{Linkages}\label{sec:link}
%-------------------------------------------------------------------
We first propose our  T-ROF model, and then use it to
derive a linkage between the PCMS and ROF models using the T-ROF model.

%-------------------------------------------------------------------
\subsection{Thresholded-ROF (T-ROF) Model}\label{sec:c_model}
%-------------------------------------------------------------------
To motivate our T-ROF model, we start by considering 
for fixed $\tau \in (0,1)$\footnote{Note that $E(\emptyset,\tau) = 0$
and $E(\Omega,\tau) = \mu \int_{\Omega} (\tau - f) \, dx$.
Since $f$ maps into $[0,1]$, the global minimizer of $E(\cdot, \tau)$
for fixed $\tau \le 0$ is $\Omega$ and for fixed $\tau \ge 1$ is $\emptyset$.
Therefore we restrict our attention to $\tau \in (0,1)$.}, 
the minimization of
\begin{equation}\label{E-single}
E(\Sigma, \tau) := {\rm Per}(\Sigma;\Omega) + \mu \int_{\Sigma} (\tau - f) \, dx.
\end{equation}
The following proposition gives a way to solve it.

%----------------------------------------------------------
\begin{proposition} \label{thm-glob-minimizer}
For any fixed $\tau \in (0,1)$, the minimizer $\Sigma_{\tau}$ of $E(\cdot,\tau)$ in \eqref{E-single}
can be found by solving the convex minimization problem
\begin{equation}\label{E-c}
\bar{u} = \underset{u\in BV(\Omega)}{\rm argmin} \Big\{ TV(u) + \mu \int_{\Omega} (\tau-f) \, u \, dx \Big\}
\end{equation}
and then setting
$\Sigma_{\tau} = \{ x \in \Omega: \bar{u}(x) > \rho \}$ for any $\rho\in [0,1)$.
\end{proposition}
%----------------------------------------------------------

For a proof of the proposition, we refer to Proposition 2.1 in the review paper \cite{CNCP10}, where
the proof uses the same ideas as in \cite{BPV91,NEC06}.
The functional \eqref{E-c} is convex and it is well known
that there exists a global minimizer. 
Hence the proposition ensures the existence
of a global minimizer of \eqref{E-single}. 
Moreover, based on 
the following Lemma from \cite[Lemma 4i)]{ACC05} and
a smoothness argument, an explanation that the minimizing
set $\Sigma_{\tau}$ is unique was given in \cite{CNCP10}.

%-----------------------------------------------------
\begin{lemma} \label{tau-sigma-ordering}
For fixed $0 < \tau_1 < \tau_2 < 1$,
let $\Sigma_i = \Sigma_{\tau_i}$ 
be minimizers of \eqref{E-single}, $i=1,2$.
Then
$|\Sigma_{2} \backslash \Sigma_{1}| = 0$
is fulfilled, i.e.,
$\Sigma_{1} \supseteq \Sigma_{2}$ up to a measure zero set.
\end{lemma}
%---------------------------------------------------

The following proposition gives another way of solving \eqref{E-single} via the ROF function.
For the proof see \cite[Proposition 2.6]{CNCP10}.

%----------------------------------------------------------
\begin{proposition} \label{pro-rof-trof}
The set $ \Sigma_{\tau} := \{x \in \Omega : u(x) > \tau\}$ solves $E(\cdot;\tau)$ in \eqref{E-single} for every $\tau \in (0,1)$
if and only if
the function $u \in BV(\Omega)$ solves the ROF model \eqref{rof}, i.e.,
\begin{equation}
\min_{u\in BV(\Omega)} \Big\{TV(u) + \frac{\mu}{2} \int_{\Omega} \big(u-f)^2dx \Big\}, \quad \mu >0.
\end{equation}
\end{proposition}
%----------------------------------------------------------

Using Lemma \ref{tau-sigma-ordering}, after minimizing \eqref{E-single}
for $0 < \tau_1 < \tau_2 < \cdots < \tau_{K-1}< 1$, we have for the corresponding sets
\begin{equation} \label{nest-sigma}
\Omega \supseteq \Sigma_{1} \supseteq \Sigma_{2}
\supseteq \cdots \supseteq \Sigma_{K-1} \supseteq \emptyset.
\end{equation}
Setting $\Sigma_{0} := \Omega$ and $ \Sigma_{K} := \emptyset$,
we see that the sets
\begin{equation} \label{nest-omega}
\Omega_i:= \Sigma_{i} \backslash \Sigma_{i+1}, \;  i=0,\ldots,K-1
\end{equation}
are pairwise disjoint and fulfill $\bigcup_{i=0}^{K-1} \Omega_i = \Omega$.

Let ${\bf \Sigma} := \{ \Sigma_i \}_{i =1}^{K-1}$, $\zb \tau := \{ \tau_i \}_{i =1}^{K-1}$ $(\tau_i < \tau_j, i < j)$, and
\begin{gather}\label{functional-E-general}
{\cal E}({\bf \Sigma}, \zb \tau) := \sum_{i=1}^{K-1} \big( {\rm Per}(\Sigma_i;\Omega)
+ \mu \int_{\Sigma_i} (\tau_i - f) \, dx \big), \quad \mu > 0.
\end{gather}
Our  T-ROF ({\it thresholded-ROF}) model aims to find a pair $({\bf \Sigma}^*, {\zb \tau}^*)$ fulfilling the two conditions
\begin{align}  \label{E-general}
{\cal E}({\bf \Sigma}^*, \zb \tau^*) \le {\cal E}({\bf \Sigma}, \zb \tau^*) \quad \mathrm{for\; all} \quad  \zb \Sigma \subset \Omega^{K-1}, 
\end{align}
and
\begin{align}\label{multicond}
 \tau_i^* = \frac{1}{2}(m_{i-1}^* + m_i^*), \quad i=1,\ldots,K-1,
 \end{align}
where
\begin{equation} \label{multicond_1}
 m_i^* := {\rm mean}_f (\Omega_i^*), \quad 
 \Omega_i^*:=\Sigma_i^* \backslash \Sigma_{i+1}^*, \;
\Sigma_0^* := \Omega, \; \Sigma_K^* := \emptyset.
\end{equation}
Since ${\cal E}(\cdot,\cdot)$ in \eqref{functional-E-general} is separable in each $i$ and, for each $i$, it is precisely of the form of  \eqref{E-single}, a minimizer of
${\cal E}(\cdot,\zb \tau^*)$ in \eqref{E-general} for some fixed $\zb \tau^*$ can be found easily by componentwise minimization for each $i$ 
using Proposition \ref{pro-rof-trof}, i.e., thresholding
the minimizer of the ROF model \eqref{rof} with $\zb \tau^*$.
However, finding a pair of $({\bf \Sigma}^*, \zb \tau^*)$ satisfying \eqref{E-general} as well as the condition \eqref{multicond}
is not straightforward. In Section \ref{sec:segmed}, we will propose an efficient way to address the solution of the T-ROF
model \eqref{E-general} that satisfies the condition \eqref{multicond}.
Once we have obtained $({\bf \Sigma}^*, {\zb \tau}^*)$, the desired segmentation for the given image $f$ is then given by $\Omega_i^*$, $i=0,\ldots,K-1$.
Finally, note that finding a pair $({\bf \Sigma}, \zb \tau^*)$ which solves the T-ROF model differs from solving
\begin{align}
 \min_{{\bf \Sigma}, {\zb \tau}} {\cal E}({\bf \Sigma}, \zb \tau), \quad \mbox{subject to} \quad \tau_i = \frac{1}{2}(m_{i-1} + m_i), \quad i=1,\ldots,K-1,
\end{align}
since we do not minimize over all feasible ${\zb \tau}$. For an example see Remark 1 in \cite{CS13}.

%---------------------------------------------------------------------------------------
\subsection{Linkage of PCMS and ROF Models} \label{subsec:related}
%---------------------------------------------------------------------------------------
Recall that 
$(\Sigma^*, m^*)$ is a {\it partial minimizer} of some objective function $E$ if
\begin{gather} \label{partial_min}
\begin{cases}
E(\Sigma^*, m^*) \le E(\Sigma^*, m), \quad \mbox{for all feasible} \; m, \\
E(\Sigma^*, m^*) \le E(\Sigma, m^*), \quad \mbox{for all feasible} \; \Sigma.
\end{cases}
\end{gather}
We see immediately that a partial minimizer 
${\bf \Omega}^*= \{\Omega_i^*\}_{i=0}^{K-1}$, 
${\zb m}^* = \{m_i^*\}_{i=0}^{K-1}$
of the PCMS model \eqref{pcms} has to fulfill
\begin{gather} \label{pcms_m}
m_i^* = {\rm mean}_f( \Omega_i^*), \quad i=0,\ldots,K-1.
\end{gather}
We note that if $E$ is differentiable on its domain, then every partial minimizer contained
in the interior of the domain is stationary, see e.g. \cite{GFK07}.
Note also that a partial minimizer must not be a local minimizer and conversely, see Fig. \ref{fig:partial_min}.
%---------------------------------------------
\begin{figure}[ht]
\centering
\includegraphics[width=5.5cm]{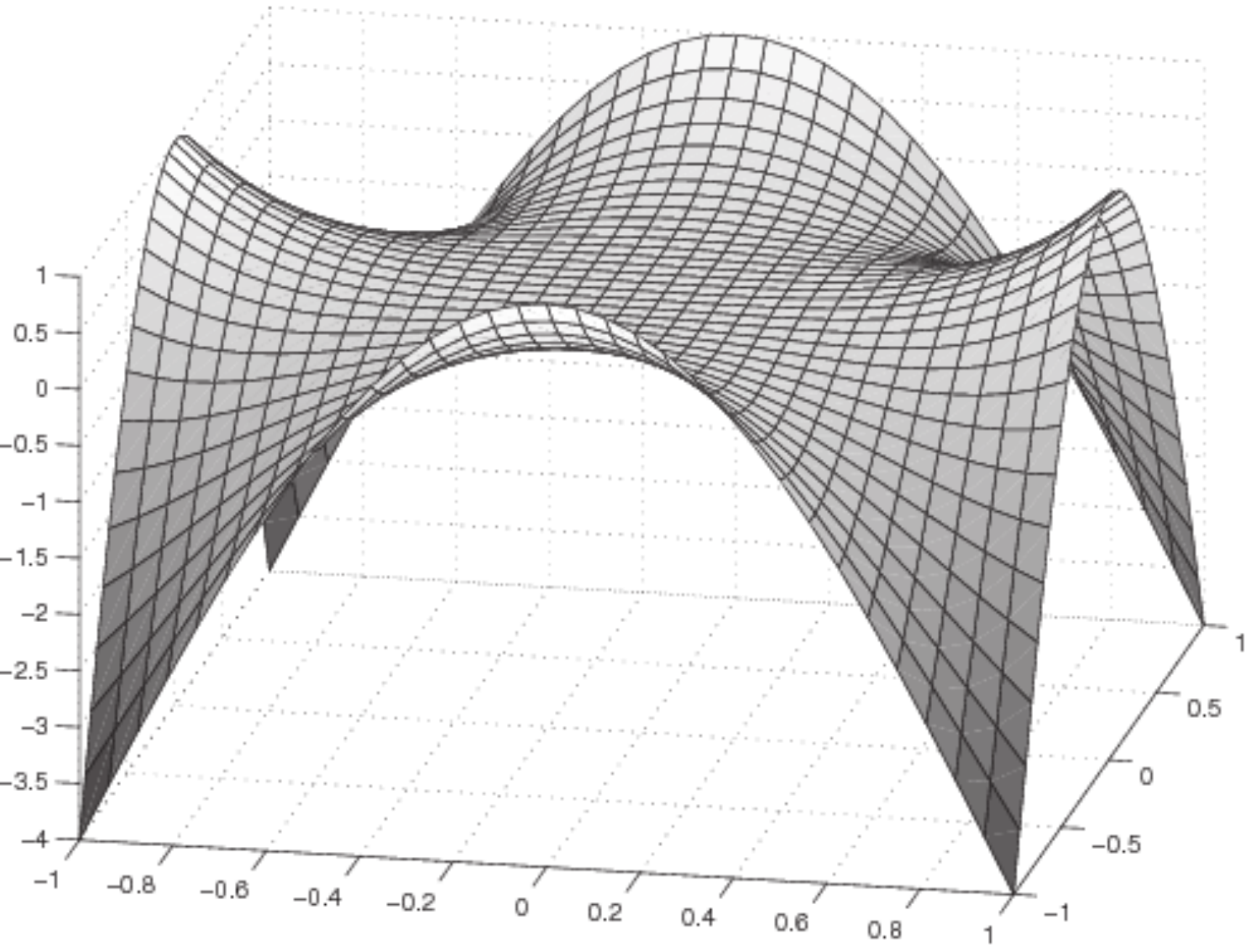}
\caption{\label{fig:partial_min} Function with a partial minimizer $(0, 0)$
which is not a local minimizer (function $(x,y) \mapsto {\rm Re} ( (x+iy)^4 ) = x^4 - 6(xy)^2 + y^4$). }
\end{figure}
%---------------------------------------------
%
\medskip

\textbf{Case $K=2$.}
We first give the relationship between our T-ROF model  and the
PCMS model for $K=2$, and then use it to derive
the relationship between the ROF and PCMS models.

%-------------------------------------------------------------
\begin{theorem} \label{lem-cv-equi}
{\rm (Relation between T-ROF and PCMS models for $K=2$)}
For $K=2$, let $(\Sigma_1^*,\tau_1^*)$ with  $0 < |\Sigma_1^*| < |\Omega|$ 
be a solution of the T-ROF model \eqref{E-general}--\eqref{multicond}. Then
$\big(\Sigma_1^*,m_0^*,m_1^* \big)$ is a partial minimizer of the PCMS
model \eqref{chan-vese} with the parameter $\lambda := \frac{\mu}{2(m_1^* - m_0^*)}$.
\end{theorem}
%-------------------------------------------------------------

\begin{proof}
Since ${\cal E}(\Sigma_1^*, \tau_1^*)\le {\cal E}(\emptyset, \tau_1^*) = 0$, we conclude
$\int_{\Sigma_1^*} (\tau_1^* - f)\, dx < 0$ which implies that 
$\tau_1^* < {\rm mean}_f(\Sigma_1^*) = m_1^*$.
Similarly, since ${\cal E}(\Sigma_1^*, \tau_1^*)\le {\cal E}(\Omega, \tau_1^*)$, we see that
\[
0 < {\rm Per}(\Sigma_1^*;\Omega)  \le \mu \int_{\Omega \backslash \Sigma_1^*} (\tau_1^* - f) \, dx
\]
and consequently $m_0^* = {\rm mean}_f(\Omega\backslash \Sigma_1^*)\le \tau_1^*$. Therefore $m_0^* < m_1^*$.

Clearly, the set $\Sigma_1^*$ is also a minimizer of ${\cal E}(\cdot, \tau_1^*) + C$ 
with the constant $C := \lambda \int_{\Omega} (m_0^* -f)^2 \, dx$.
Setting $\tau_1^* := \frac{m_1^*+m_0^*}{2}$, we obtain
\begin{align}
	{\cal E}(\Sigma, \tau_1^*) \! +\! C
	&=
	{\rm Per}(\Sigma;\Omega) +\mu \int_{\Sigma} (\tau_1^* - f) \, dx + C \nonumber\\
	&=
	 {\rm Per}(\Sigma;\Omega) + \frac{\mu}{2(m_1^* - m_0^*)}
\int_{\Sigma} \big[(m_1^* -f)^2 \! - \! (m_0^* - f)^2 \big]\, dx \! + C \nonumber \\
	&=
	{\rm Per}(\Sigma;\Omega) + \lambda
	\left( \int_{\Sigma} (m_1^* -f)^2 dx + \int_{\Omega\backslash\Sigma} (m_0^* -f)^2 dx \right).	 \label{eqt3.5}
\end{align}
By the definition of $m_i$, $i=0,1$ and \eqref{pcms_m}
we get the assertion.
\end{proof}

%--------------------
\begin{remark}
Since $f\in [0, 1]$, we have $0< m_1^* - m_0^* \le1 $. The parameter $\lambda = \frac{\mu}{2(m_1^* - m_0^*)}$ in the
Chan-Vese model \eqref{chan-vese} is larger than $\mu$ and increases dramatically if
$(m_1^* - m_0^*)$ becomes smaller. Hence, this $\lambda$ is adapted to
the difference between $m_1^*$ and $m_0^*$ and penalizes the data term more in the Chan-Vese model if this difference becomes smaller.
We now know that when $(m_1^* - m_0^*)$ is very small, $\lambda$ used in the Chan-Vese model should be large; however,
in practice, to solve the Chan-Vese model, $\lambda$ is given beforehand with no knowledge about this kind of information. 
It is therefore very hard if not impossible for the Chan-Vese model to be given a good value of $\lambda$ to obtain a high quality 
segmentation result.
In contrast, it is much easier for our  T-ROF model to get good results by just tuning the threshold $\tau_1^*$ 
(automatically, see Section \ref{sec:segmed}), 
no matter how large or small the difference between $m_1^*$ and $m_0^*$ is.
\end{remark}
%--------------------

Next we give the relationship between the ROF model  and the
PCMS model for $K=2$.

%-------------------------------------------------------------
\begin{theorem} \label{thm-cv-rof-equi} {\rm (Relation between ROF and PCMS models for $K=2$)}
Let $K=2$ and $u^* \in BV(\Omega)$ solve the ROF model \eqref{rof}.
For given $0<m_0 < m_1 \le 1$, let 
$\tilde{\Sigma} := \{x \in \Omega : u^*(x) > \frac{m_1+m_0}{2} \}$ fulfill 
$0 < |\tilde{\Sigma}| < |\Omega|$.
 Then $\tilde{\Sigma}$ is a  minimizer of the PCMS model \eqref{chan-vese} for
 $\lambda := \frac{\mu}{2(m_1 - m_0)}$ and fixed $m_0, m_1$.
 In particular, $(\tilde{\Sigma}, m_0, m_1)$ is a partial minimizer of \eqref{chan-vese}
 if $m_0 = {\rm mean}_f(\Omega\backslash\tilde{\Sigma})$ and $m_1 = {\rm mean}_f(\tilde{\Sigma})$.
\end{theorem}
%-------------------------------------------------------------

\begin{proof}
Following \eqref{eqt3.5}, we have for all $m_0 < m_1$, $\lambda = \frac{\mu}{2(m_1 - m_0)}$ and $\Sigma \subset \Omega$ that
\begin{align} \label{eqn:tocv}
{\cal E}(\Sigma, \frac{m_1+m_0}{2} ) + \lambda \int_{\Omega} (m_0 -f)^2 \, dx =
E_{\rm CV}(\Sigma,m_0,m_1).
\end{align}
By Proposition \ref{pro-rof-trof}, $\tilde{\Sigma}$ minimizes
${\cal E}(\cdot, \frac{m_1+m_0}{2} )$ and hence also $E_{\rm CV}(\cdot,m_0,m_1)$.
If $\tilde{m}_0 = {\rm mean}_f(\Omega\backslash\tilde{\Sigma})$ and $\tilde{m}_1 = {\rm mean}_f(\tilde{\Sigma})$, then
$0 < \tilde{m}_0 < \tilde{m}_1$ ({\it cf.} the first part of the proof of Theorem \ref{lem-cv-equi});
and by \eqref{pcms_m} they minimize $E_{\rm CV}(\tilde{\Sigma}, \cdot, \cdot)$.
\end{proof}
\medskip

%--------------------------------------------------------------------------------------------------------
\textbf{Case $K>2$.}
Now we consider $K>2$. For ${\bf \Sigma}$ in \eqref{nest-sigma} and $\{\Omega_i\}_{i=0}^{K-1}$
in \eqref{nest-omega}, we know $\Omega_i\cup\cdots\cup \Omega_{K-1} = \Sigma_{i}$ so that
\[
\sum_{i=0}^{K-1}{\rm Per}(\Omega_i\cup\cdots\cup \Omega_{K-1}; \Omega) =
\sum_{i=0}^{K-1}{\rm Per}(\Sigma_{i}; \Omega).
\]
If ${\rm int}(\Sigma_{i})\supset \bar{\Sigma}_{i+1}$, 
that is the closure of set $\Sigma_{i+1}$ is inside the interior of $\Sigma_{i}$ and therefore 
$\partial\Sigma_{i} \cap \partial\Sigma_{i+1} = \emptyset$, then
\begin{equation}
\sum_{i=0}^{K-1}{\rm Per}(\Omega_i\cup\cdots\cup \Omega_{K-1}; \Omega) =
\sum_{i=0}^{K-1}{\rm Per}(\Sigma_{i}; \Omega) =
\frac{1}{2}\sum_{i=0}^{K-1}{\rm Per}(\Omega_i; \Omega).
\end{equation}
We consider the following variant (extension) of the PCMS model \eqref{pcms}
\begin{equation} \label{E-multi-vari-ms}
\min_{\Omega_i, m_i} \Big\{
\sum_{i=0}^{K-1}{\rm Per}(\Omega_i\cup\cdots\cup \Omega_{K-1}; \Omega)
+ \sum_{i=0}^{K-1}\tilde{\mu}_i\int_{\Omega_i} (m_i-f)^2dx \Big\},
\end{equation}
where $\tilde{\mu}_i > 0$ are regularization parameters.
Let us call model \eqref{E-multi-vari-ms} the PCMS-V model.
The relationship between the T-ROF model \eqref{E-general}--\eqref{multicond} and the PCMS-V model \eqref{E-multi-vari-ms}
is given in the following Theorem \ref{thm-equi-MS}.

%-------------------------------------------------------------
\begin{theorem} \label{thm-equi-MS} {\rm (Relation between the T-ROF and PCMS-V models for $K>2$)}
For $K>2$, let $\left( \{ \Sigma^*_i \}_{i=1}^{K-1}, \{\tau_i^*\}_{i=1}^{K-1} \right)$ 
be a solution of the T-ROF model
\eqref{E-general}--\eqref{multicond}  with $m^*_i < m^*_{i+1}$, $i=0,\ldots,K-1$.
Then $\left(\{\Omega_i^*\}_{i=0}^{K-1}, \{m_i^*\}_{i=0}^{K-1} \right)$ is a partial minimizer of the
PCMS-V model \eqref{E-multi-vari-ms} with regularization parameters defined as
\begin{equation} \label{reg-par-redefine}
\tilde{\mu}_i =
\begin{cases}
\frac{\mu}{2(m_1^*-m_0^*)}, & i=0, \\
\frac{\mu}{2(m_i^*-m_{i-1}^*)}+\frac{\mu}{2(m_{i+1}^*-m_{i}^*)}, & i=1, \ldots, K-2,\\
\frac{\mu}{2(m_{K-1}^*-m_{K-2}^*)}, & i=K-1,
\end{cases}
\end{equation}
where $\{\Omega_i^*\}_{i=0}^{K-1}$
is obtained by \eqref{nest-omega} with $\{ \Sigma^*_i \}_{i=1}^{K-1}$.
\end{theorem}
%-------------------------------------------------------------
\begin{proof}
When minimizing ${\cal E}({\bf \Sigma}, \zb \tau^*)$ \eqref{functional-E-general} with respect to
${\bf \Sigma}$, i.e.,
\begin{equation}\label{E-separate-i}
\min_{\Sigma_i} \Big\{ {\rm Per}(\Sigma_i, \Omega)+ \mu
\int_{\Sigma_i} (\tau_i^*-f)dx \Big\},
\quad \quad i=1, \ldots, K-1,
\end{equation}
it is clear that finding $\Sigma_i^*$ is independent with finding $\Sigma_j^*$ for $j\neq i$.
Hence, $\Sigma_1^*, \ldots, \Sigma_{i-1}^*$, $\Sigma_{i+1}^*, \ldots, \Sigma_{K-1}^*$
can be regarded as fixed when finding $\Sigma_i^*$.
Note that, we also have
$\Sigma_1^*\supseteq\Sigma_2^* \supseteq \cdots \supseteq \Sigma_{K-1}^*$ using Lemma \ref{tau-sigma-ordering}.
From Theorem \ref{lem-cv-equi}, we know the minimizer of
\eqref{E-separate-i} for each $i$ is a partial minimizer of
\begin{equation}\label{E-separate-i-equi-rewrite}
\min_{\Sigma_i,m_{i-1}, m_i} \Big\{ {\rm Per}(\Sigma_{i}; \Omega) +
\frac{\mu}{2(m_{i}^*-m_{i-1}^*)} \Big(\int_{\Sigma_{i}}
(m_{i}-f)^2dx + \int_{\Omega\backslash \Sigma_i} (m_{i-1}-f)^2dx\Big) \Big \},
\end{equation}
which is equivalent to
\begin{align}
\min_{\Omega_i,m_{i-1}, m_i} & \Big\{ {\rm Per}(\Omega_i\cup\Omega_{i+1}^*\cup\cdots\cup \Omega_{K-1}^*; \Omega) \nonumber \\
& + \frac{\mu}{2(m_{i}^*-m_{i-1}^*)} \Big(\int_{\Omega_{i}} (m_{i}-f)^2dx
 + \int_{\Omega_{i-1}} (m_{i-1}-f)^2dx\Big) \Big\}. \label{E-separate-i-equi-rewrite-2}
\end{align}
The proof is completed by summing up the above objective functions for all $i = 1, \ldots, K-1$.
\end{proof}
%-------------------------------------------------------------

\begin{remark}
Note that for the standard PCMS model in \eqref{pcms}, the regularization parameter $\mu$ is fixed. In contrast,
its variation PCMS-V model in \eqref{E-multi-vari-ms} derived from our T-ROF model \eqref{E-general} has more flexible regularization
parameters. This kind of resetting for regularization parameters will avail the PCMS-V model (and our T-ROF method which is given in the
next section) for multiphase segmentation particularly for images containing phases with close intensities.
We demonstrate this fact in experimental results later.
\end{remark}

In summary, we conclude that, when $K=2$, the T-ROF model gives a segmentation result of the PCMS model
for fixed $\lambda$ and $m_i, i=0, 1$ using the way of defining $\lambda$ in Theorem \ref{thm-cv-rof-equi}, and
vice versa. When $K>2$, Theorem \ref{thm-equi-MS} tells us that, if $\partial\Sigma_{i} \cap \partial\Sigma_{i+1} = \emptyset, i=1, \ldots, K-1$,
then the T-ROF model gives a segmentation result of the PCMS model with $\mu$ redefined as in \eqref{reg-par-redefine};
otherwise, only an approximation to the PCMS model can be achieved by the T-ROF model. 
It is worth mentioning that the case of $\partial\Sigma_{i} \cap \partial\Sigma_{i+1} \neq \emptyset$ is due to jumps in the ROF solution,
which are a subset of the jumps of the original image $f$, see \cite[Theorem 5]{CNCP10} for more details. 

Even though the equivalence between the PCMS model and the T-ROF model cannot hold anymore 
when $K>2$ and $\partial\Sigma_{i} \cap \partial\Sigma_{i+1} \neq \emptyset$,  
the way of quantifying, reducing and/or finally overcoming the gap between them will be of great interest for future research. 
Note, importantly, that the lack of equivalence for $K>2$  does not mean that the T-ROF model performs poorer than the PCMS model. 
In the experimental results, we will show that 
in some cases the T-ROF method is actually much better than the state-of-the-art methods based on the PCMS model.

%---------------------------------------------------------------------------------------
\section{T-ROF Algorithm and Its Convergence} \label{sec:segmed}
%---------------------------------------------------------------------------------------
Proposition \ref{pro-rof-trof} implies that we can
obtain a minimizer $\zb \Sigma$ of ${\cal E}(\cdot,\zb \tau)$ in \eqref{E-general}
by minimizing the ROF functional and subsequently thresholding the minimizing function by $\tau_i$, $i=1,\ldots,K-1$.
This method is particularly efficient since the minimizer of the ROF functional remains the same and thus just need
to be solved once when we apply various thresholds $\zb \tau^ {(k)}$.
Here, at iteration $k$, when we have $({\bf \Sigma}^{(k)}, \zb \tau^{(k)})$, we use the following rule to
obtain $\zb \tau^{(k+1)}$:
\begin{equation} \label{trof-thd}
 m_i^{(k)} = {\rm mean}_f \left(\Sigma_i^{(k)} \backslash \Sigma_{i+1}^{(k)} \right), \
\tau_i^{(k+1)} = \frac12 \left(m_{i-1}^{(k)} + m_i^{(k)} \right),  \ i=1,\ldots,K-1,
\end{equation}
where we assume the ordering
$0 < \tau_1^{(k)} <  \cdots < \tau_{K-1}^{(k)} < 1$ and 
\begin{equation} \label{eqn:seg-meaningful-omega}
|\Omega_i^{(k)}| = |\Sigma_{i}^{(k)} \backslash \Sigma_{i+1}^{(k)}| > 0, \quad  i = 0, \ldots, K-1.
\end{equation}
We abbreviate the rule in \eqref{trof-thd} by 
\begin{equation} \label{trof-thd-abb}
\zb \tau^{(k+1)}= \Phi({\bf \Sigma}^{(k)}, \zb \tau^{(k)}).
\end{equation}

\begin{remark}
If the above criterion \eqref{eqn:seg-meaningful-omega}  is not satisfied at step $k$, 
for example  $|\Omega_i^{(k)}| = 0$, $i\in  \{0, \ldots, K-1\}$, then
we remove $\tau_i^{(k)}$ from the 
threshold sequence $\{\tau_i^{(k)}\}_{i =1}^{K-1}$, which will not affect much of the segmentation result practically when not considering 
measure zero set. A new threshold sequence $\{\tau_i^{(k)}\}_{i =1}^{\hat{K}-1}, \hat{K}<K$ can be formed 
which is considered in \eqref{eqn:seg-meaningful-omega} now.

For a set  $\hat{\Omega} \subset \Omega$ and $\hat{\Omega} \supset \Sigma \neq \emptyset$, where $\Sigma$ is the minimizer of ${E}_{ \Omega}(\cdot,\tau)$
in \eqref{E-single}, we want that segmenting  
$\hat{\Omega}$ into $\Sigma$ and $\hat{\Omega} \backslash \Sigma$ 
does not produce a worse segmentation result than the
naive segmentation of $\hat{\Omega}$ into the set itself and an empty set. In other words, we require 
\begin{equation} \label{eqn:assump}
	E_{\hat{\Omega}}(\Sigma,{\tau}) \le E_{\hat{\Omega}}(\hat{\Omega},{\tau}) =  \mu \big(\tau - {\rm mean}_f (\hat \Omega) \big) |\hat \Omega|,
	\quad
	E_{\hat{\Omega}}(\Sigma,{\tau}) \le E_{\hat{\Omega}}(\emptyset,{\tau}) = 0.
\end{equation}
If the requirement \eqref{eqn:assump} is not fulfilled, this might imply that the difference between $\Sigma$ and
$\hat{\Omega} \backslash \Sigma$ is too subtle to be distinguished and  we keep $\hat{\Omega}$ rather than segmenting it into 
$\Sigma$ and $\hat{\Omega} \backslash \Sigma$.
\end{remark}

\begin{lemma} \label{lemma:assump}
For $i = 1, \ldots, K-2$, let $\Sigma_{i} \neq \Omega$ and $\Sigma_{i+1} \neq \emptyset$ be the minimizers of ${E}(\cdot,{\tau_i})$ and
${E}(\cdot,{\tau_{i+1}})$ for  $0< \tau_i < \tau_{i+1} < 1$ appearing in the T-ROF model, where $|\Sigma_{i} \backslash \Sigma_{i+1}| > 0$.
If the inequalities in \eqref{eqn:assump} hold, then 
\begin{gather} \label{eqn:add-rule}
\tau_i \le {\rm mean}_f \left(\Sigma_{i} \backslash \Sigma_{i+1} \right) \le \tau_{i+1}.
\end{gather}
\end{lemma}

\begin{proof} By definition we have $\Sigma_{i} \supset \Sigma_{i+1}$.
Let $\hat{\Omega} = \Sigma_i$ and $\Sigma = \Sigma_{i+1}$. 
Then using the first inequality in \eqref{eqn:assump}, we have
\begin{align}
{\rm Per}(\Sigma_{i+1};\Sigma_i) + \mu \int_{ \Sigma_{i+1} } (\tau_{i+1} - f) \, dx 
&\le \mu \int_{ \Sigma_{i} } (\tau_{i+1} -f) \, dx,
\label{eqn:assump-rhs}
\end{align}
i.e.,
\begin{equation*}
	0 \le {\rm Per}(\Sigma_{i+1};\Sigma_{i})  \le  \mu \int_{ \Sigma_{i}\backslash \Sigma_{i+1} } (\tau_{i+1} - f) \, dx
\end{equation*}
which implies 
$
{\rm mean}_f \left( \Sigma_{i} \backslash \Sigma_{i+1} \right) \le \tau_{i+1}.
$

Let $\hat{\Omega} = \Omega \backslash \Sigma_{i+1}$ and 
$\Sigma = \Sigma_{i} \backslash \Sigma_{i+1}$.  
Using the second inequality in \eqref{eqn:assump}, we have
\begin{equation*} 
	{\rm Per}\big(\Sigma_{i} \backslash \Sigma_{i+1}; \Omega\backslash \Sigma_{i+1}\big) 
	+ \mu \int_{ \Sigma_{i} \backslash \Sigma_{i+1} } (\tau_i - f) \, dx \le 0,
\end{equation*}
which implies with $|\Sigma_{i} \backslash \Sigma_{i+1}| > 0$ that
\begin{equation*}
\mu \int_{ \Sigma_{i} \backslash \Sigma_{i+1} } (\tau_i - f) \, dx \le 0,
\end{equation*}
i.e.,
$
\tau_i \le {\rm mean}_f (\Sigma_{i} \backslash \Sigma_{i+1}).
$
This completes the proof.
\end{proof}

In practice, for $\Sigma_{i} \supset \Sigma_{i+1}$ ($0< \tau_i < \tau_{i+1} < 1$), if 
$\Sigma_{i} \backslash \Sigma_{{i+1}}$ contains different texture from $\Sigma_{{i+1}}$, 
then it is highly likely that segmenting $\Sigma_i$ into $\Sigma_{i} \backslash \Sigma_{{i+1}}$
and $\Sigma_{{i+1}}$ is much better than the naive segmentation of $\Sigma_{i}$.
This gives us a useful criterion to design our T-ROF algorithm: at some step, if the naive segmentation of $\Sigma_{i}$ is better
than segmenting it into $\Sigma_{i} \backslash \Sigma_{i+1}$
and $\Sigma_{i+1}$, then 
removing $\Sigma_{i+1}$ (${\tau_{i+1}}$) from the set (threshold) sequence is more meaningful and 
will not influence the segmentation result much.

In sum, after obtaining $({\bf \Sigma}^{(k)}, \zb \tau^{(k)})$ at iteration $k$, we apply 
the criterions in \eqref{eqn:seg-meaningful-omega} and \eqref{eqn:add-rule} to remove measure zero phases and ignore
unnecessary segmentation, respectively. To achieve this, we can particularly keep checking the following steps i)--iii) until they are all satisfied: 
i) if $(\Sigma_{i}^{(k)}, \tau_{i}^{(k)})$ and $(\Sigma_{i+1}^{(k)}, \tau_{i+1}^{(k)})$, 
$0\le i \le K-1$, do not fulfil
the criterion  \eqref{eqn:seg-meaningful-omega}, then remove $(\Sigma_{i+1}^{(k)}, \tau_{i+1}^{(k)})$ from 
$({\bf \Sigma}^{(k)}, \zb \tau^{(k)})$, and form a new $({\bf \Sigma}^{(k)}, \zb \tau^{(k)})$ by reordering the indices, and update $K$; 
ii) if $(\Sigma_{i}^{(k)}, \tau_{i}^{(k)})$, $1\le i \le K-1$, and one of its neighbors in $({\bf \Sigma}^{(k-1)}, \zb \tau^{(k-1)})$ 
say $(\Sigma_{j}^{(k-1)}, \tau_{j}^{(k-1)}), j \in \{i-1, i, i+1\}$, do not fulfil
the criterion  \eqref{eqn:add-rule}, then replace the values of $(\Sigma_{i}^{(k)}, \tau_{i}^{(k)})$ by $(\Sigma_{j}^{(k-1)}, \tau_{j}^{(k-1)})$;
and iii) if $(\Sigma_{i}^{(k)}, \tau_{i}^{(k)})$ and $(\Sigma_{i+1}^{(k)}, \tau_{i+1}^{(k)})$, $0\le i \le K-1$, do not fulfil
the criterion  \eqref{eqn:add-rule}, then remove $(\Sigma_{i+1}^{(k)}, \tau_{i+1}^{(k)})$ from 
$({\bf \Sigma}^{(k)}, \zb \tau^{(k)})$, and form a new $({\bf \Sigma}^{(k)}, \zb \tau^{(k)})$ by reordering the indices, and update $K$.
We abbreviate the above steps i)--iii) conducting the criterions \eqref{eqn:seg-meaningful-omega} and \eqref{eqn:add-rule} by 
\begin{equation} \label{eqn:clean}
({\bf \Sigma}^{(k)}, \zb \tau^{(k)})  := {\cal C} ({\bf \Sigma}^{(k)}, \zb \tau^{(k)}).
\end{equation}

By combining the criterion in \eqref{eqn:clean} and  the rule of updating $\zb \tau^{(k)}$ by \eqref{trof-thd-abb},
we obtain our T-ROF segmentation method in Algorithm \ref{alg:t-rof}, to find a solution of the T-ROF model \eqref{E-general}--\eqref{multicond}.

%-----------------
\begin{algorithm}[h]
 \caption{T-ROF Segmentation Algorithm}
 \label{alg:t-rof}

{\bf Initialization:} Phase number $K \ge 2$ and
initial thresholds
$\zb\tau^{(0)} = \big(\tau_i^{(0)} \big)_{i=1}^{K-1}$ with $0 \le \tau_1^{(0)}
< \cdots < \tau_{K-1}^{(0)} \le 1$. \\
 Compute the solution $u$ of the ROF model \eqref{rof}. \\
  {\bf For} $k=0, 1,\ldots,$ until stopping criterion reached
  \begin{itemize}
 \item[1.] Compute the minimizers $\Sigma_{i}^{(k)}$ of $E(\cdot,\tau_i^{(k)})$
 by setting \\ 
  $\Sigma_{i}^{(k)} = \{x\in \Omega: u(x)>\tau_i^{(k)}\}$.
 \item[2.] Apply criterions \eqref{eqn:seg-meaningful-omega} and \eqref{eqn:add-rule} to update $({\bf \Sigma}^{(k)}, \zb \tau^{(k)})$, e.g. see \eqref{eqn:clean}.
 \item[3.] Update $\zb \tau^{(k+1)}= \Phi({\bf \Sigma}^{(k)}, \zb \tau^{(k)})$ by \eqref{trof-thd-abb}.
\end{itemize}
{\bf Endfor}

\end{algorithm}
%-----------------

From  Algorithm \ref{alg:t-rof}, we see that  the T-ROF method exactly follows the new paradigm which is: performing image segmentation
through image restoration plus a thresholding. The T-ROF method can therefore be regarded as a special case of our proposed
SaT method in \cite{CCZ13}. Particularly, in addition to the theoretical bare bones in the T-ROF method and the fast speed it inherits from
the standard SaT method, the iterative way of selecting optimal thresholds -- the strategy in \eqref{trof-thd} -- in the T-ROF 
method makes it more effective in multiphase segmentation when compared with the standard SaT method which uses the K-means to 
select thresholds. This advantage is more notable when segmenting images with phases which have close intensities.

There are many efficient methods to solve the ROF model, for example
the primal-dual algorithm \cite{CP11}, alternating direction method with multipliers (ADMM) \cite{BPCPE10},
or the split-Bregman algorithm \cite{GO09}. In this paper,
we use the ADMM to solve the ROF model \eqref{rof}.
The convergence property of Algorithm \ref{alg:t-rof} is discussed in the next section.

%---------------------------------------------------------------------------------------
\subsection{Convergence Proof of the T-ROF Algorithm}  \label{subsec:proof}
%---------------------------------------------------------------------------------------
We now prove the convergence of our T-ROF algorithm -- Algorithm \ref{alg:t-rof}.

%----------------------------------------------------
\begin{lemma} \label{lem-ass-seq}
Let $A,B,C,D \subset \mathbb R^2$ be bounded measurable sets with $A \supseteq B \supseteq C \supseteq D$
and let $f: \mathbb R^2 \rightarrow \mathbb R$ be a Lebesgue-integrable function.
Then the following implications hold true:
\begin{itemize}
 \item[{\rm i)}] if ${\rm mean}_f(A \backslash B) \le {\rm mean}_f(B \backslash C)$,
then: \\
${\rm mean}_f(A \backslash C) \le {\rm mean}_f(B \backslash C)$
and
${\rm mean}_f(A \backslash B) \le {\rm mean}_f(A \backslash C)$;
 \item[{\rm ii)}] if ${\rm mean}_f(A \backslash B) \le {\rm mean}_f(B \backslash C) \le {\rm mean}_f(C \backslash D)$,
then: \\
 ${\rm mean}_f(A \backslash C) \le {\rm mean}_f(B \backslash D)$.
\end{itemize}
\end{lemma}

\begin{proof}
i) We first prove the first assertion in i).
If $|A \backslash C| = 0$, then this assertion is clearly true.
Let $|A \backslash C| > 0$.
Since
$A\backslash C = A\backslash B \cup B \backslash C$ and $A\backslash B \cap B \backslash C = \emptyset$,
we conclude $|A\backslash C| = |A\backslash B| + |B \backslash C|$ and in particular
$|A\backslash B| >0$ if $|B \backslash C| = 0$.
Assume that ${\rm mean}_f(A \backslash C) > {\rm mean}_f(B \backslash C)$. Then
$$
\frac{ \int_{A \backslash C} f \, dx }{ |A\backslash C|}
=
\frac{ \int_{A \backslash B} f \, dx + \int_{ B \backslash C} f \, dx }{ |A\backslash B|
+ |B \backslash C| }
> {\rm mean}_f(B \backslash C) =
\left\{
\begin{array}{ll}
0 &{\rm if} \; |B \backslash C| = 0,\\
\frac{ \int_{B \backslash C} f \, dx}{|B \backslash C|} &{\rm if} \; |B \backslash C| > 0.
\end{array}
\right.
$$
Both cases yield a contradiction to the assumption ${\rm mean}_f(A \backslash B) \le {\rm mean}_f(B \backslash C)$.

Concerning the second assertion in i) we are done if $|A \backslash B| = 0$.
If $|A \backslash B| > 0$, then by the above considerations $|A \backslash C| > 0$
and assuming
${\rm mean}_f(A \backslash B) > {\rm mean}_f(A \backslash C)$
we obtain
$$
\frac{ \int_{A \backslash C} f \, dx }{ |A\backslash C|}
=
\frac{ \int_{A \backslash B} f \, dx + \int_{ B \backslash C} f \, dx }{ |A\backslash B|
+ |B \backslash C| }
<
\frac{ \int_{A \backslash B} f \, dx }{ |A\backslash B|}.
$$
This yields again a contradiction to the assumption.

ii) Applying the first/second implication in i) with respect to the first/second inequality in the assumption of ii) we obtain
$$
{\rm mean}_f(A \backslash C) \le {\rm mean}_f(B \backslash C)
\quad {\rm and} \quad
{\rm mean}_f(B \backslash C) \le {\rm mean}_f(B \backslash D),
$$
which lead to
${\rm mean}_f(A \backslash C) \le {\rm mean}_f(B \backslash D)$
and we are done.
\end{proof}
%----------------------------------------------------

Recall that $\tau_i^{(k+1)} = \frac12(m_{i-1}^{(k)} + m_i^{(k)})$,
where
$m_i^{(k)} = {\rm mean}_f(\Sigma_i^{(k)} \backslash \Sigma_{i+1}^{(k)})$.
Using the criterion derived in \eqref{eqn:add-rule} and Lemma \ref{lem-ass-seq}, we can prove the following lemma.
%----------------------------------------------------
\begin{lemma} \label{lem-ass}
Our T-ROF Algorithm \ref{alg:t-rof} produces sequences
$( \zb \tau^{(k)} )_k$ and $(\zb m^{(k)})_k$
with the following properties:
\begin{itemize}
\item[{\rm i)}]
$
0 \le m_0^{(k)} \le \tau_1^{(k)} \le m_1^{(k)}
\le \cdots \le m_{K-2}^{(k)} \le \tau_{K-1}^{(k)} \le m_{K-1}^{(k)}
$.
\item[{\rm ii)}] Set
$\tau_0^{(k)} := 0$ and $\tau_K^{(k)} := 1$.
If $\tau_i^{(k)} \ge \tau_i^{(k-1)}$ and
$\tau_{i+1}^{(k)} \ge \tau_{i+1}^{(k-1)}$,
then
$m_i^{(k)} \ge m_i^{(k-1)}$, $i=0,\ldots,K-1$
and this also holds true if $\ge$ is replaced everywhere by $\le$.
\end{itemize}
\end{lemma}

\begin{proof} i) At step $k$, let $\tau_i = \tau_i^{(k)}$.
For fixed $0 \le \tau_1 < \tau_2 < \cdots < \tau_{K-1} < 1$, let $\Sigma_i$ be a minimizer of $E(\cdot,\tau_i)$.
With Lemma \ref{lemma:assump}, we only need to prove $m_0 \le \tau_1$ and $\tau_{K-1} \le m_{K-1}$.
Assume that
$\Omega_0 := \Omega \backslash \Sigma_1$ and
$\Omega_{K-1} := \Sigma_{{K-1}}$ have positive measure.

Since $|\Omega_0|, |\Omega_{K-1}|>0$, neither $\Omega$ nor $\emptyset$ is minimizer of $E(\cdot,\tau_i)$,
we verify that
\begin{align}
{\rm Per}(\Sigma_1;\Omega) + \mu \int_{\Sigma_1} (\tau_1 - f) \, dx &\le \mu \int_{\Omega} (\tau_1 -f) \, dx,
\label{help1}
\\
{\rm Per}(\Sigma_{K-1};\Omega) + \mu \int_{\Sigma_{K-1}} (\tau_{K-1} - f) \, dx &\le 0.
\label{help2}
\end{align}
From \eqref{help1}, we have
\begin{equation*}
{\rm Per}(\Sigma_1; \Omega) + \mu \int_{\Sigma_1} (\tau_1 - f) \, dx \le
\mu \int_{ \Omega_0} (\tau_1 - f) \, dx + \mu \int_{\Sigma_1} (\tau_1 - f) \, dx.
\end{equation*}
Note that $m_i := {\rm mean}_f (\Omega_i)$.
Thus, $0 \le \mu \int_{ \Omega_0} (\tau_1 - f) \, dx$ which gives $m_0 \le \tau_1$.
From \eqref{help2}, we have
$0 \ge \int_{ \Omega_{K-1}} (\tau_{K-1} - f) \, dx,$
which means $\tau_{K-1} \le m_{K-1}$.

ii) We only prove for sign $\ge$. The proof for sign $\le$ follows the same lines.

If $\tau_1^{(k)} \ge \tau_1^{(k-1)}$, let $\tau_1 := \tau_{1}^{(k-1)}, \tau_2 := \tau_{1}^{(k)}$, using the conclusion from
Lemma \ref{lem-ass} i), we get
$$
0 \le {\rm mean}_f (\Omega \backslash \Sigma_1^{(k-1)}) \le \tau_{1}^{(k-1)}
\le {\rm mean}_f (\Sigma_1^{(k-1)} \backslash \Sigma_1^{(k)})
\le \tau_{1}^{(k)}.
$$
Hence, from the second implication of Lemma \ref{lem-ass-seq} i), we have $m_0^{(k)} \ge m_0^{(k-1)}$.

If $\tau_{K-1}^{(k)} \ge \tau_{K-1}^{(k-1)}$, let $\tau_1 := \tau_{K-1}^{(k-1)}, \tau_2 := \tau_{K-1}^{(k)}$,
using the conclusion from Lemma \ref{lem-ass} i), we get
$$
0 \le \tau_{{K-1}}^{(k-1)} \le {\rm mean}_f (\Sigma_{K-1}^{(k-1)} \backslash \Sigma_{K-1}^{(k)})
\le \tau_{K-1}^{(k)} \le {\rm mean}_f (\Sigma_{K-1}^{(k)}).
$$
Hence, from the first implication of Lemma \ref{lem-ass-seq} i), we have $m_{K-1}^{(k)} \ge m_{K-1}^{(k-1)}$.

For $i \in \{1,\ldots,K-2 \}$ and if $\tau_i^{(k)} \ge \tau_i^{(k-1)}$
and $\tau_{i+1}^{(k)} \ge \tau_{i+1}^{(k-1)}$, from Lemma \ref{lem-ass} i), we have
$$
\tau_i^{(k-1)} \le \tau_{i+1}^{(k-1)}
\quad {\rm and} \quad
\tau_i^{(k)} \le \tau_{i+1}^{(k)}.
$$
Hence, we can only have one of the following orderings:
\begin{itemize}
\item[a)] $\tau_i^{(k-1)} \le \tau_{i+1}^{(k-1)} \le \tau_i^{(k)} \le \tau_{i+1}^{(k)}$;
\item[b)] $\tau_i^{(k-1)} \le \tau_i^{(k)} \le \tau_{i+1}^{(k-1)} \le \tau_{i+1}^{(k)}$.
\end{itemize}
In case a) we obtain by Lemma \ref{lem-ass} i) that
$$
\tau_i^{(k-1)} \le m_i ^{(k-1)} \le \tau_{i+1}^{(k-1)} \le \tau_i^{(k)} \le m_i^{(k)}.
$$
In case b) we conclude by Lemma \ref{lem-ass} i) with the settings
$\tau_1 := \tau_i^{(k-1)}, \tau_2 := \tau_i^{(k)}$
and
$\tau_1 := \tau_i^{(k)}, \tau_2 := \tau_{i+1}^{(k-1)}$
and
$\tau_1 := \tau_{i+1}^{(k-1)}, \tau_2 := \tau_{i+1}^{(k)}$, respectively,
\begin{align*}
\tau_i^{(k-1)} \le& {\rm mean}_f (\Sigma_i^{(k-1)} \backslash \Sigma_{i}^{(k)})
\le \tau_i^{(k)}
\le {\rm mean}_f (\Sigma_i^{(k)} \backslash \Sigma_{i+1}^{(k-1)})
\le \tau_{i+1}^{(k-1)} \le \\
\le& {\rm mean}_f (\Sigma_{i+1}^{(k-1)} \backslash \Sigma_{i+1}^{(k)})
\le \tau_{i+1}^{(k)}.
\end{align*}
By Lemma \ref{lem-ass-seq} ii) this implies $m_i^{(k-1)} \le m_i^{(k)}$, which completes the proof.
\end{proof}

It is straightforward to verify that the generated sequences satisfying \eqref{eqn:clean} will 
fulfil the above Lemma \ref{lem-ass}.

%----------------------------------------------------

To prove the convergence of the sequence $({\zb \tau}^{(k)})_{k}$, we
define a sign sequence
$\zb \zeta^{(k)} = (\zeta_i^{(k)})_{i=1}^{K-1}$ as follows:
If $\tau_i^{(k)} \neq \tau_i^{(k-1)}$,
\begin{gather}\label{tau-sign-general}
\zeta_i^{(k)}:=
\begin{cases}
+1, & \textrm{if}\ \tau_i^{(k)} > \tau_i^{(k-1)}, \\
-1, & \textrm{if}\ \tau_i^{(k)} < \tau_i^{(k-1)},
\end{cases}
\end{gather}
and otherwise
\begin{gather}\label{tau-sign}
\zeta_i^{(k)}:=
\begin{cases}
\zeta_j^{(k)}, & \textrm{if}\ i=1,\\
\zeta_{i-1}^{(k)}, & \textrm{if}\ i\neq 1,
\end{cases}
\end{gather}
where
$ j := \min\{ l \ | \ \tau_l^{(k)} \neq \tau_l^{(k-1)}\}$.
Additionally, since $\tau_0^{(k)}$ and $\tau_K^{(k)}$ do not change with $k$, we set 
\begin{equation} \label{eqn:z0zk}
\zeta_0^{(k)} := \zeta_1^{(k)}, \quad \zeta_K^{(k)} := \zeta_{K-1}^{(k)}.
\end{equation}
By $s_k$ we denote the number of sign changes in ${\zb \zeta}^{(k)}$.
Writing for simplicity $\pm$ instead of $\pm 1$, we obtain for example that
$$
\zb \zeta^{(k)} = \{\zeta_i^{(k)}\}_{i=1}^{K-1} = 
 (+  + | - |+ | -|   +  + + |-  -|  +| -| +|  - - |+ |-   |++| ---| +  +| - |+ |-)
$$
has $s_k = 17$.

%----------------------------------------------------
\begin{lemma} \label{lem-sign}
{\rm i)} The number of sign changes $s_k$ is monotone decreasing in $k$. \\
{\rm ii)} If $\zeta_1^{(k+1)} \neq \zeta_1^{(k)}$, then we have the strict decrease $s_{k+1} < s_{k}$.
\end{lemma}
%----------------------------------------------------

\begin{proof} i) Let $s_k = N -1$
and consider the sequence
\begin{equation} \label{whole_seq}
\underbrace{
{\zeta_0^{(k)}}, \cdots, \zeta_{l_1}^{(k)}\vert}_{v_1^{(k)}} \, \cdots \, ,
\underbrace{\vert \zeta_{i_j}^{(k)}, \cdots, \zeta_{l_{j}}^{(k)}\vert}_{v_j^{(k)}}
\, \cdots \, ,
\underbrace{\vert \zeta_{i_N}^{(k)}, \cdots, {\zeta_{K}^{(k)}}}_{v_N^{(k)}},
\end{equation}
where $v_j^{(k)}$, named a sign block, contains those successive components with the same sign, and $\zeta_0^{(k)}$ and $\zeta_K^{(k)}$
(defined in \eqref{eqn:z0zk}) are used for the boundary elements of the whole sequence
which by definition  do not belong to ${\zb \zeta}^{(k)}$ (see \eqref{tau-sign-general} and \eqref{tau-sign}). 
Note that $\#v_1^{(k)}\ge 2$ and $\#v_N^{(k)}\ge 2$.
\\

1. For $j\ge 2$, if $\#v_j^{(k)}\ge 3$, we consider $\zeta_{i^*}^{(k+1)}$ with $i_j\le i^*-1 < i^* < i^*+1 \le l_j$, i.e.,
$\zeta_{i^*-1}^{(k)} = \zeta_{i^*}^{(k)} = \zeta_{i^*+1}^{(k)}$.
WLOG let $\zeta_{i^*}^{(k)} = -1$.
Then we obtain by Lemma \ref{lem-ass} ii) that
$m_{i^*-1}^{(k)} \leq m_{i^*-1}^{(k-1)}$ and $m_{i^*}^{(k)} \leq m_{i^*}^{(k-1)}$.
Therefore
\[
\tau_{i^*}^{(k+1)} = \frac{m_{i^*-1}^{(k)} + m_{i^*}^{(k)}}{2} \leq \frac{m_{i^*-1}^{(k-1)} + m_{i^*}^{(k-1)}}{2}
= \tau_{i^*}^{(k)}
\]
and consequently $\zeta_{i^*}^{(k+1)} = -1$,
 or
$\zeta_{i_j}^{(k+1)} = \zeta_{i_j+1}^{(k+1)} = \cdots = \zeta_{i^*}^{(k+1)} =1$ (which is obtained by definition
\eqref{tau-sign} when $\tau_{i^*}^{(k+1)} = \tau_{i^*}^{(k)}$).

Specifically, for $j = 1$, if $\#v_1^{(k)}\ge 3$, we consider $\zeta_{i^*}^{(k+1)}$ with $0\le i^*-1 < i^* < i^*+1 \le l_1$.
Following the same lines above, we have, if $\zeta_{i^*}^{(k)} = -1$,
then $\zeta_{i^*}^{(k+1)} = -1$,
or
$\zeta_{0}^{(k+1)} = \zeta_{1}^{(k+1)} = \cdots = \zeta_{i^*}^{(k+1)} =1$. Moreover, in the latter case, we can also prove 
$ \zeta_{i^*}^{(k+1)} = \zeta_{i^*+1}^{(k+1)} = \cdots = \zeta_{l_1}^{(k+1)}  =1$,
using the definition in \eqref{tau-sign} for $\zeta_1^{(k+1)}$ when $\tau_1^{(k+1)} = \tau_1^{(k)}$
(which is just opposite to the way of defining $\zeta_i^{(k+1)}$ when  $\tau_i^{(k+1)} = \tau_i^{(k)}$ for $i>1$).
The case $\zeta_{i^*}^{(k)} = 1$ can be handled in the same way.

Therefore we obtain for a sequence $v_j^{(k)}$ with $\#v_j^{(k)}\ge 3$ the following properties from step $k$ to step $k+1$:
\begin{itemize}
\item[-] no sign changes if the boundary signs are kept;
\item[-] maximal one additional sign change if the left or right boundary sign is kept;
\item[-] maximal two additional sign changes if both boundary signs are not kept.
\end{itemize}
Here is an example:
\begin{align*}
-|++++|- \quad \rightarrow \quad & |++++|,  |+++-|, 
  |---+|, |--++|, \\
	& |-++-|, |--+-|, |----|. 
\end{align*}

2. Assume that there is a sequence  in \eqref{whole_seq} with no $j$ such that $\#v_j^{(k)} = 1$.
Consider
$\zeta_{l_j}^{(k)}$ and $\zeta_{i_{j+1}}^{(k)}$ which are different by definition.
WLOG, let
$\zeta_{l_{j}}^{(k)} = -1$ and $\zeta_{i_{j+1}}^{(k)} = +1$.
Then $\zeta_{l_{j} -1}^{(k)} = -1$ and $\zeta_{i_{j+1} +1}^{(k)} = +1$. From Lemma \ref{lem-ass} ii), we have
\[
m_{l_{j}-1}^{(k)} \leq m_{l_{j}-1}^{(k-1)}, \quad m_{i_{j+1}}^{(k)} \geq m_{i_{j+1}}^{(k-1)}.
\]
If
$m_{l_j}^{(k)} < m_{l_j}^{(k-1)}$ (or $m_{l_j}^{(k)} \geq m_{l_j}^{(k-1)}$),
then
$\tau_{l_j}^{(k+1)} < \tau_{l_j}^{(k)}$ (or $\tau_{i_{j+1}}^{(k+1)} \geq \tau_{i_{j+1}}^{(k)}$).
This means that
$\zeta_{l_{j}}^{(k+1)} \neq \zeta_{l_{j}}^{(k)}$
and
$\zeta_{i_{j+1}}^{(k+1)}\neq \zeta_{i_{j+1}}^{(k)}$ are not possible at the same time.

Let us call a subsequence containing consecutive sign blocks $\#v_j^{(k)} \ge 2$ for all $j$ within the whole sequence
a non-oscillating subsequence, short nosc-sequence.
Then we can use the  above arguments to show that the number of sign changes  for an inner nosc-sequence from step $k$ to step $k+1$ 
behaves as follows:
\begin{itemize}
\item[-] no additional sign changes if the boundary signs are kept;
\item[-] maximal one additional sign change if the left or right boundary sign is kept;
\item[-] maximal two additional sign changes  if both boundary signs are not kept.
\end{itemize}
Here is an example of the sign changes of a nosc-sequence containing two sign blocks for the first two boundary conditions:
\begin{align*}
|+++---|  \ \  \rightarrow  \ \  
& 
|+++ -- -|,  |{ +}++ +- -|, |+ ++ ++  -|,
 |++- -- -|,\\
&
|+++ - - +|,    |+++ +-{ +}|,  | +++ +++| ,  |++-  -- +|,\\
&
|- ++ - - -|, | - -+ - - -|, |- - + ++ -|, | - -+ + - -|, \\
&|- +++ - -|,|- - ++ + -| , |- + - - - -| , |- - - - - -|.
\end{align*}

In particular, accompanying the proof in part 1, we obtain that, if the whole sequence is a nosc-sequence, the sign changes of the whole sequence 
will not increase.

3. Now, we consider the case $\# v_j^{(k)} = 1$ for all $j_1 \le j \le j_2$, where
$\#v_{j_1 - 1}^{(k)} > 1$ and $\#v_{j_2 + 1}^{(k)} > 1$. 
We prove that from step $k$ to step $k+1$ the signs of
\begin{equation} \label{eqn:osc-se}
\zeta_{i_{j_1}-1}^{(\cdot)}, \  \zeta_{i_{j_1}}^{(\cdot)}, \ \ldots, \  \zeta_{i_{j_2}}^{(\cdot)}, \  \zeta_{i_{j_2}+1}^{(\cdot)}
\end{equation}
can not change at the same time. 
We call the above sequence \eqref{eqn:osc-se} oscillatory subsequence, short osc-sequence.
Here is an example, where the osc-sequence is underbraced and the inner pattern overbraced:
\begin{equation} \label{eqn:osc-se-example}
 \ldots + \underbrace{+|\overbrace{- + - +}|-} - \ldots, \quad  \mathrm{or} \quad \ldots + \underbrace{+|\overbrace{- + - + -  }|+} + \ldots.
\end{equation}
WLOG assume that
$\zeta_{i_{j_1}}^{(k)}=-1$ so that $\zeta_{i_j}^{(k)}=(-1)^{j-j_1+1}$ for
$j_1\le j \le j_2$, and $\zeta_{i_{j_1}-1}^{(k)} = \zeta_{i_{j_1}-2}^{(k)}=+1$ and
$\zeta_{i_{j_2}+1}^{(k)} = \zeta_{i_{j_2}+2}^{(k)}=(-1)^{j_2-j_1}$.
From Lemma \ref{lem-ass} ii), we know that
\begin{gather}\label{eqn-sign-single}
m_{i_{j_1}-2}^{(k)} \ge m_{i_{j_1}-2}^{(k-1)}, \quad {\rm and} \quad
\begin{cases}
m_{i_{j_2}+1}^{(k)} \ge m_{i_{j_2}+1}^{(k-1)} & \textrm{if} \ j_2-j_1 \ \textrm{is even}, \\
m_{i_{j_2}+1}^{(k)} \le m_{i_{j_2}+1}^{(k-1)} & \textrm{if} \ j_2-j_1 \ \textrm{is odd}.
\end{cases}
\end{gather}
Assume that $\zeta_{i_{j_1}-1}^{(\cdot)}$, $\zeta_{i_{j_1}}^{(\cdot)}$, $\cdots$, $\zeta_{i_{j_2}}^{(\cdot)}$,
$\zeta_{i_{j_2}+1}^{(\cdot)}$ all change signs at the same time from step $k$ to step $k+1$, 
e.g. for the left example in \eqref{eqn:osc-se-example}, we have
$$ 
\ldots + + | - + - + | - - \ldots \quad  \rightarrow  \quad  \ldots  - | + - + - |+  \ldots .$$
Then we deduce 
\begin{align*}
\tau_{ i_{j_1-1} }^{(k+1)} 
&= \frac{ m_{i_{j_1-2}}^{(k)} + m_{i_{j_1-1}}^{(k)}}{2 } 
\le 
\tau_{i_{j_1-1}}^{(k)} = \frac{m_{i_{j_1-2}}^{(k-1)} + m_{i_{j_1-1}}^{(k-1)}}{2 }
\end{align*}
and since 
$
m_{i_{j_1-2}}^{(k)}  \ge m_{i_{j_1-2}}^{(k-1)}
$
this implies
$$
m_{j_1 - 1}^{(k)} \le m_{j_1 - 1}^{(k)} .
$$
Taking into account that $\tau_{ i_{j_1} }^{(k+1)} = \tau_{ i_{j_1} }^{(k)}$ would require the same sign $\zeta_{ i_{j_1} }^{(k+1)}$
as in $\zeta_{ i_{j_1-1} }^{(k+1)}$ (see definition \eqref{tau-sign}), we get further
\begin{align*}
\tau_{ i_{j_1} }^{(k+1)} 
&= \frac{ m_{i_{j_1-1}}^{(k)} + m_{i_{j_1}}^{(k)}}{2 } 
>
\tau_{i_{j_1}}^{(k)} = \frac{m_{i_{j_1-1}}^{(k-1)} + m_{i_{j_1}}^{(k-1)}}{2 }
\end{align*}
and since 
$
m_{i_{j_1-1}}^{(k)}  \le m_{i_{j_1-1}}^{(k-1)}
$
this implies
$$
m_{i_{j_1}}^{(k)} > m_{i_{j_1}}^{(k-1)}.
$$
Continuing in this way, we get finally
the contradiction
\[
\begin{cases}
m_{i_{j_2}+1}^{(k)} < m_{i_{j_2}+1}^{(k-1)}, & \textrm{if} \ j_2-j_1 \ \textrm{is even}, \\
m_{i_{j_2}+1}^{(k)} > m_{i_{j_2}+1}^{(k-1)}, & \textrm{if} \ j_2-j_1 \ \textrm{is odd}.
\end{cases}
\]

Then the above arguments show that the number of sign changes  for an osc-sequence from step $k$ to step $k+1$ 
behaves as follows:
\begin{itemize}
\item[-] the number of sign changes cannot increase;
\item[-] the number of sign changes either not change or at least decreases by two if both boundary signs are kept;
\item[-] the number of sign changes at least decreases by two if both boundary signs are not kept;
\item[-] the number of sign changes at least decreases by one if one boundary sign is kept.
\end{itemize}

4. Now we can decompose the sequence $\zb \zeta^{(k)}$ into subsequences with osc-pattern  and nosc-pattern 
(note that the case of a sequence $\zb \zeta^{(k)}$ which is a nosc-sequence has been proved in part 2), 
where these patterns overlap in the boundary signs. 
This is illustrated by an example, where the osc-patterns are underbraced and the nosc-pattern are overbraced:
\begin{center}
$$
\zb + + 
+  - +  -  \overbrace{+   + + - 
 -}   + - +  -
 - + -  \overbrace{ +   + --- + 
  + } - + - \zb -.
$$
\vspace{-41pt}
$$
\phantom{\zb + + \!} 
\underbrace{ \phantom{+  - +  -  +} } 
\phantom{+ + -} 
\underbrace{\phantom{-   + - +  -}  } \phantom{\;}
\underbrace{\phantom{- + -   +} }
\phantom{+ - - - + \;} 
\underbrace{\phantom{+  - +  -}}
\phantom{\zb -}
$$
\end{center}
If nothing changes in the boundary of the osc-pattern,
then the number of sign changes cannot increase for $\zb \zeta^{(\cdot)}$ from step $k$ to step $k-1$.
If both signs in the boundary of an osc-sequence changes, there are at least two sign changes less within the osc-sequence
which can be at most added in the nosc-sequences, but we will not obtain more sign changes for the whole sequence $\zb \zeta^{(\cdot)}$.
If one sign at the osc-boundary is kept and the other is changed we can use counting arguments
to show that the number of sign changes of the whole sequence cannot increase.
This completes the proof.

ii)
If $\zeta_1^{(k+1)} \neq \zeta_1^{(k)}$, we first prove all the components
in $v_1^{(k)}$ must change signs at step $k+1$. 
If $\#v_1^{(k)} = 2$, we get the proof directly since ${\zeta_0^{(k+1)}}$ follows the 
sign of $\zeta_1^{(k+1)}$ according to the definition in \eqref{tau-sign};
if $\#v_1^{(k)} \ge 3$, the proof can be found in the proof of part 1.

By parts 1--4 of the proof, we have block $v_{1}^{(k)}$ will then be merged with 
all or part of $v_{2}^{(k)}$ and have the same sign, where the whole sequence $\zb \zeta^{(\cdot)}$ will 
have one sign change less, since the rest components in $v_{2}^{(k)}$,
if they exist, will be merged with all or part of $v_{3}^{(k)}$, keeping this process until 
merging with the last block $v_{N}^{(k)}$. Therefore we have $s_{k+1} < s_{k}$.
This completes the proof.
\end{proof}

Now we can prove the convergence of our T-ROF algorithm.

%----------------------------------------------------------
\begin{theorem} \label{thm-tau-converge}
The sequence $({\zb \tau}^{(k)})_{k\in \mathbb{N}}$
produced by the T-ROF Algorithm \ref{alg:t-rof} converges to a vector ${\zb \tau}^*$.
Then $({\zb \Sigma}^*,{\zb \tau}^*)$ is a solution of the T-ROF model \eqref{E-general}.
\end{theorem}
%----------------------------------------------------------

\begin{proof}
We prove the assertion by induction on the number of sign changes $s_k$ at some iteration step $k$.

i) Assume that $s_k = 0$. WLOG let
$\zeta_i^{(k)} = +1$, $i= 1, \ldots, K-1$,
i.e.,
$\tau_i^{(k)} \geq \tau_i^{(k-1)}$.
From Lemma \ref{lem-ass} ii), we obtain
$m_i^{(k)} \geq m_i^{(k-1)}$ and consequently $\tau_i^{(k+1)} \geq \tau_i^{(k)}$, $i= 1, \ldots, K-1$.
Therefore $s_{k+1} = 0$ and
$\zeta_i^{(k+1)} = +1$, $i= 1, \ldots, K-1$.
This means that each sequence $(\tau_i^{(k)})_k$ is monotone increasing.
Since the sequences are moreover bounded in [0, 1], we conclude that $({\zb\tau}^{(k)})_k$
converges. Note that $s_k=0$ when $K=2$.

ii) Assume that $({\zb\tau}^{(k)})_k$ converges if
$s_k \leq N-1$ for some $k \in \mathbb N$.

iii) We prove that $({\zb\tau}^{(k)})_k$ converges in case of $s_k = N$.
If there exists a $\hat{k}$ such that
$\zeta_1^{(\hat{k})} \neq \zeta_1^{(\hat{k} -1)}$,
we get $s_{\hat{k}}\leq N-1$ directly from Lemma \ref{lem-sign} ii); therefore
$({\zb\tau}^{(k)})_k$ converges by ii).
If $\zeta_1^{(k+1)} = \zeta_1^{(k)}$ for all $k>\hat k$, then
$(\tau_1^{(k)})_{k>\hat k}$ is monotone and bounded and converges consequently
to some threshold $\tau_1^*$. By the definition of $\Sigma_1^{(k)}$,
we have $\Sigma_1^{({k})}$ converges, and therefore
$m_0^{({k})}$ also converges. Since $\tau_1^{(k)} = (m_0^{({k})} + m_1^{({k})})/2$,
we have $m_1^{({k})}$ converges.
Now we prove $\tau_2^{(k)}$ also converges (here we assume there does not exist a $\hat{k}$
such that $\zeta_2^{(k+1)} = \zeta_2^{(k)}$ for all $k>\hat k$, since otherwise the convergence of
$\tau_2^{(k)}$ is obtained immediately).
If not, $\tau_2^{(k)}$ must contain at least two convergent subsequences converge
to say $\tau_2^*$ and $\tau_2^{'}$. WLOG, let $\tau_2^* < \tau_2^{'}$.
Then we have $\Sigma_2^* \supseteq \Sigma_2^{'}$ from Lemma \ref{tau-sigma-ordering}.
In what follows, we show by contradiction that it is impossible.
\begin{itemize}
\item[-] If $\Sigma_2^* \supset \Sigma_2^{'}$, then
from Lemmas \ref{lemma:assump} and \ref{lem-ass-seq}, we have
${\rm mean}_f (\Sigma_1 \backslash \Sigma_2^*) < {\rm mean}_f (\Sigma_{1} \backslash \Sigma_2^{'})$, which implies $m_1^{({k})}$ diverges.

\item[-] If $\Sigma_2^* = \Sigma_2^{'}$,  then there exists a $\hat{k}$ such that
$\Sigma_2^{(k)} = \Sigma_2^{'}$ for $k > \hat{k}$. Considering thresholds
$\tau_3^{(k)}, \cdots, \tau_{K-1}^{(k)}$ in set $\Sigma_2^*$, the convergence of
them is obtained immediately from ii) by their sign changes which are $\leq N-1$. Therefore
$m_2^{({k})}$ converges, with the convergence of $m_1^{({k})}$, so $\tau_2^{(k)}$ also converges. 
\end{itemize}
Analogously, we can have $\tau_3^{({k})}$ converges from the convergence of
$\tau_2^{(k)}$ and $m_2^{({k})}$.
Repeating this procedure up to the final index $K-1$, we obtain the assertion.
\end{proof}

%-------------------------------------------------------------------
\section{Numerical Results}\label{sec:numerics}
%-------------------------------------------------------------------
In this section we test our proposed T-ROF method on many kinds of images. More precisely, we use
the T-ROF Algorithm \ref{alg:t-rof} with a discrete ROF model (see e.g. \cite{Ch04})
whose minimizer is computed numerically by an ADMM algorithm with its inner parameter fixed to 2.
Speedups by using more sophisticated methods will be considered in future work.
The stopping criteria in the T-ROF algorithm for $u$ and $\zb \tau$ are
\begin{equation} \label{eqn:stop-criterion}
\| u^{(i)}-u^{(i-1)} \|_2 / \|u^{(i)}\|_2 \le \epsilon_u 	\quad {\rm and} \quad
\|{\zb \tau}^{(k)} - {\zb \tau}^{(k-1)} \|_2 \le \epsilon_{\tau},
\end{equation}
where $\epsilon_u$ and $\epsilon_{\tau}$ are fixed to $10^{-4}$ and $10^{-5}$,
respectively. The initialization of $\{\tau_i^{(0)}\}_{i=1}^{K-1}$ was computed by
the fuzzy C-means method \cite{BEF84} with 100 iterations.

We compare our method with the recently proposed multiphase segmentation methods
\cite{CCZ13,HHMSS12,LNZS10,PCCB09,YBTB10}. Note that the methods \cite{PCCB09,YBTB10} work with
the fixed fuzzy C-means codebook $\{m_i\}_{i=0}^{K-1}$ which are not updated (all the codebooks used in the
following examples are given in Appendix). Such update is however involved in \cite{HHMSS12}.
The default stopping criterion used in \cite{HHMSS12,LNZS10,PCCB09} is the maximum iteration steps,
which will be shown in each example;
the default stopping criterion used in \cite{CCZ13} is the relative error with tolerance set to $10^{-4}$;
and the default stopping criterion used in \cite{YBTB10} is the same as the one used in \cite{CCZ13},
together with maximum 300 iteration steps.
We choose the regularization parameter,  $\lambda/\mu$ (note that $\lambda$ is used in the PCMS model and $\mu$ is
used in the SaT method \cite{CCZ13} and our T-ROF model), in front of the fidelity term for all the methods
by judging the {\it segmentation accuracy} (SA) defined as
\begin{equation} \label{eqn:sa}
\textrm{SA} := \frac{\# \textrm{correctly classified pixels}}{\# \textrm{all pixels}}
\end{equation}
unless otherwise stated. 
We show the results for two two-phase and five multiphase images. 
Moreover, an extra example -- one three-phase segmentation problem based on retina manual segmentation --
is provided to further demonstrate the superior performance of our proposed method.   
All computations were run on
a MacBook with 2.4 GHz processor and 4GB RAM.

\subsection{Two-phase Image Segmentation}

{\it Example 1. Cartoon image with some missing pixel values.}
Fig. \ref{twophase-circle} (a) is the clean two-phase image, with constant value in each phase.
Fig. \ref{twophase-circle} (b) is the test corrupted image generated by
removing some pixel values (in this example, 80\% pixels are removed) randomly from Fig. \ref{twophase-circle} (a).
Fig. \ref{twophase-circle} (c)--(g) are the results of the methods \cite{LNZS10,PCCB09,YBTB10,HHMSS12,CCZ13}, respectively.
Fig. \ref{twophase-circle} (h)--(j) give the results of our T-ROF method respectively at iterations 1, 2 and 6 (final result)
with $\mu=1$, which clearly shows the
effectiveness of the updating strategy of our T-ROF method on $\zb \tau$ given in \eqref{trof-thd}.
For the T-ROF method, the iteration steps to find $u$ and $\zb \tau$ are 418 and 6 (simply represented as
418 (6) in Table \ref{time-examples-simple}), respectively.
From those results, we see that only
methods \cite{HHMSS12,CCZ13} and our T-ROF method give good results.
The quantitative comparison in terms of $\mu$, iteration steps, computation time, and SA for each method is given
in Table \ref{time-examples-simple}. We can see that the T-ROF method gives the highest SA, which
shows the effectiveness and necessity of updating the threshold $\zb \tau$, compared with the slightly poor results of
the SaT method \cite{CCZ13} which just uses the fixed thresholds selected by K-means.

\begin{figure}[!htb]
\begin{center}
\begin{tabular}{cccc}
\includegraphics[width=\ww, height=\ww]{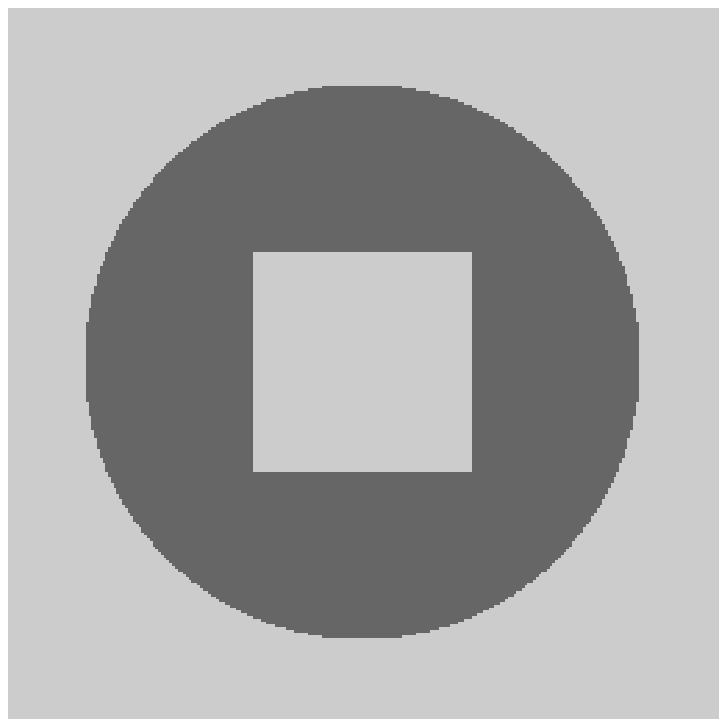} &
\includegraphics[width=\ww, height=\ww]{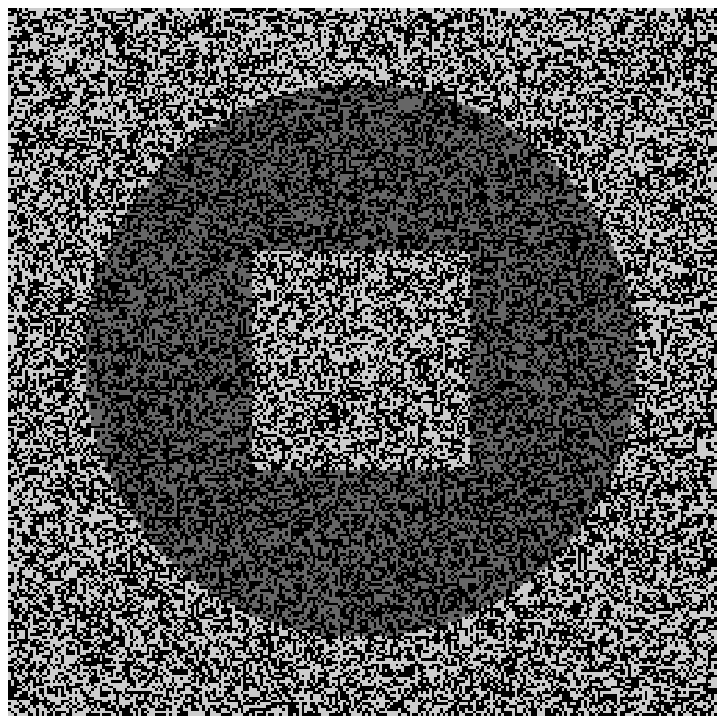} \\
(a) Clean image & (b) Corrupted image
\end{tabular}
\begin{tabular}{cccc}
\includegraphics[width=\ww, height=\ww]{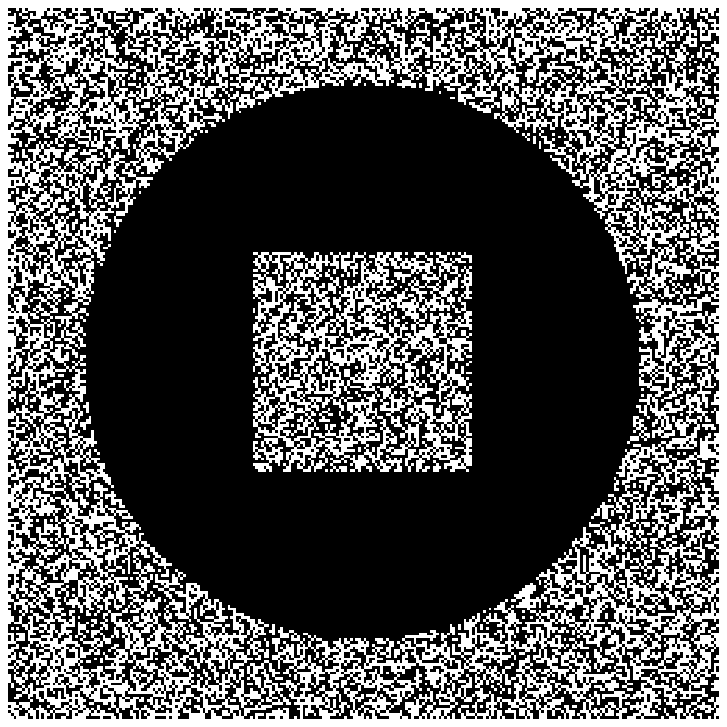} &
\includegraphics[width=\ww, height=\ww]{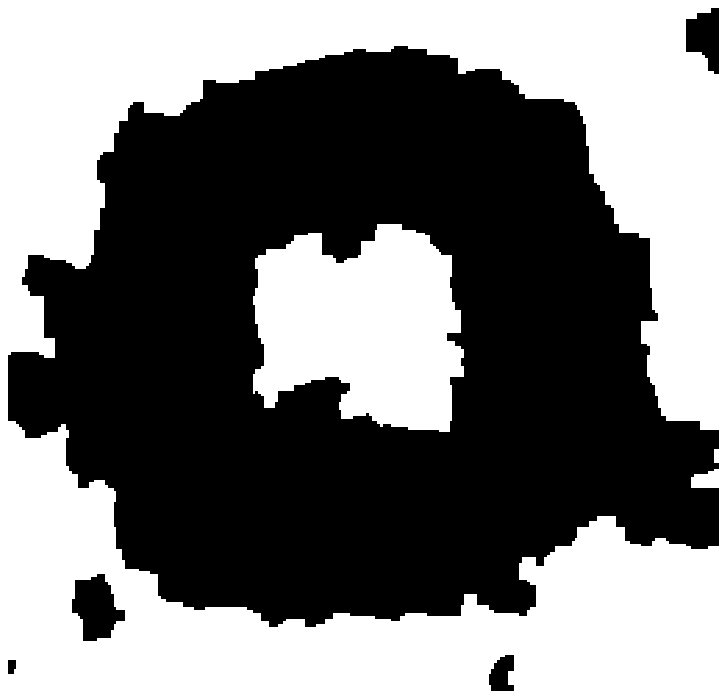} &
\includegraphics[width=\ww, height=\ww]{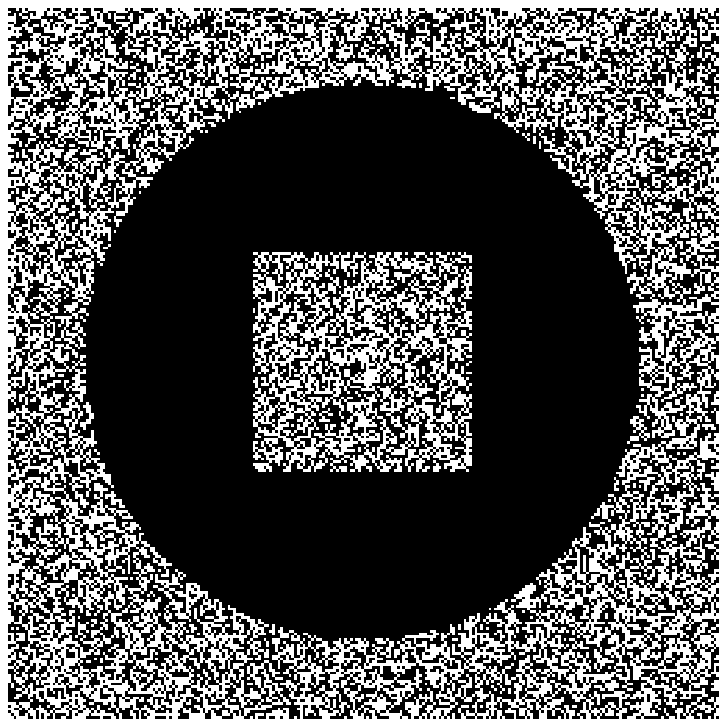} &
\includegraphics[width=\ww, height=\ww]{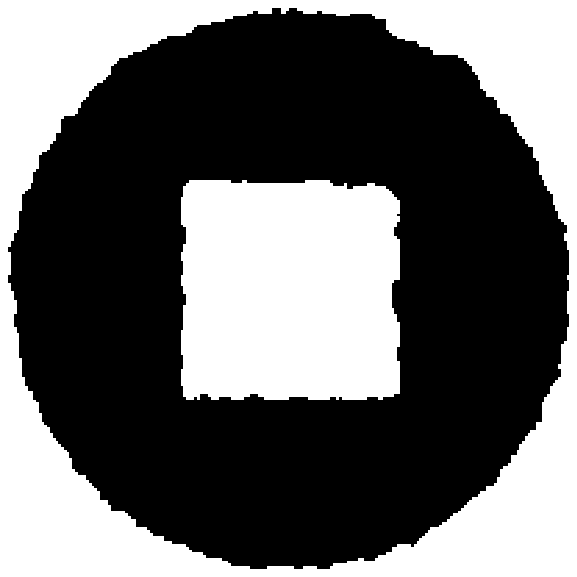} \\
(c) Li \cite{LNZS10} & (d) Pock \cite{PCCB09} & (e) Yuan \cite{YBTB10} & (f) He \cite{HHMSS12} \\
\includegraphics[width=\ww, height=\ww]{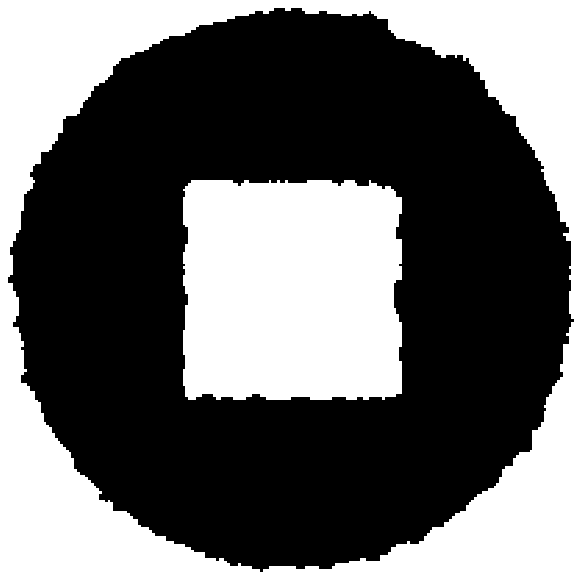} &
\includegraphics[width=\ww, height=\ww]{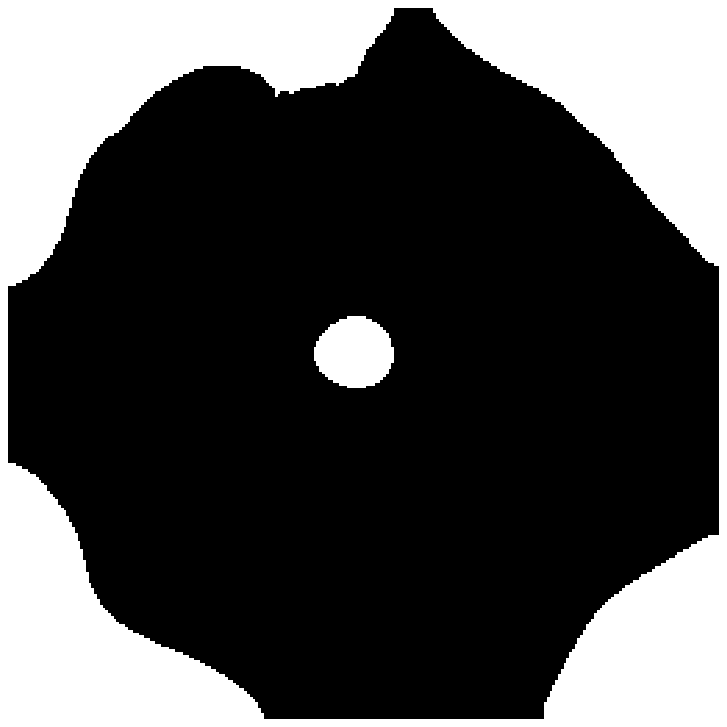} &
\includegraphics[width=\ww, height=\ww]{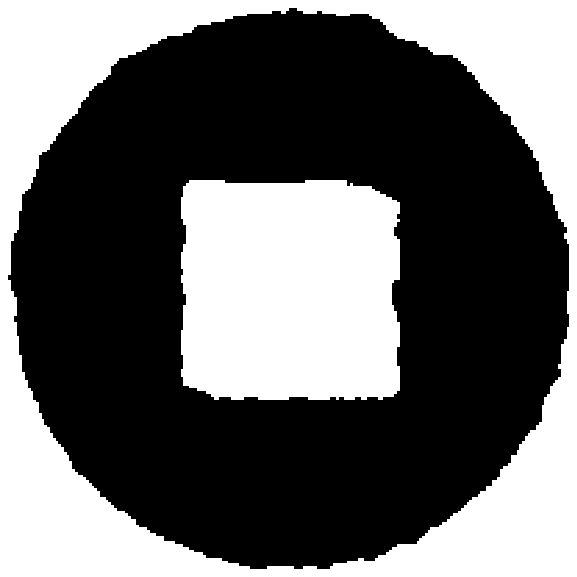} &
\includegraphics[width=\ww, height=\ww]{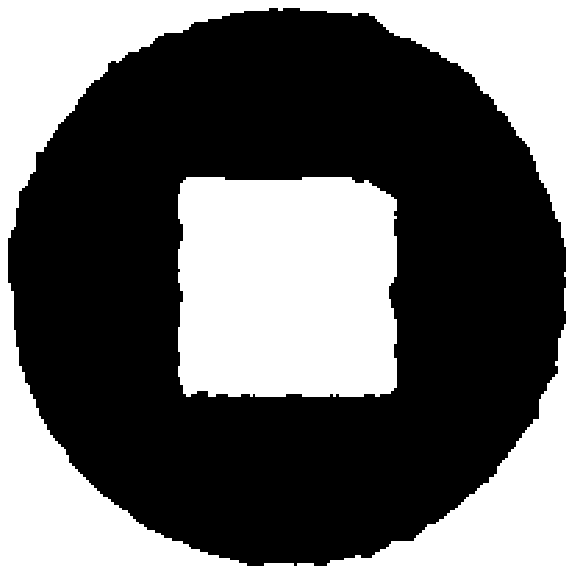} \\
 (g) Cai \cite{CCZ13} & (h) Ours (Ite. 1) & (i) Ours (Ite. 2) & (j) Ours (Ite. 6)
\end{tabular}
\end{center}
\caption{Segmentation of two-phase cartoon image with some missing pixel values (size $256\times 256$).
(a): clean image; (b): image (a) with some pixel values removed randomly;
(c)--(g): results of methods \cite{LNZS10,PCCB09,YBTB10,HHMSS12,CCZ13}, respectively;
(h)--(j): results of our T-ROF method at iterations 1, 2, and 6 (final result), respectively.
}\label{twophase-circle}
\end{figure}

{\it Example 2. Two-phase image with close intensities.}
Fig. \ref{twophase-shape} (a) is an image generated by adding Gaussian noise with mean 0 and
variance $10^{-8}$ onto a constant image with constant value 0.5. Fig. \ref{twophase-shape} (b) is a
two-phase mask separating the whole domain into two parts (the black and the white color parts).
The testing noisy image Fig. \ref{twophase-shape} (c) is generated from Fig. \ref{twophase-shape} (a)
by keeping the pixel values belonging to the white part in the mask and reducing the pixel values
belonging to the black part by a factor of $2\times10^{-4}$.
This way of generating test images has the following two main features:
i) the noise pattern (e.g. Gaussian) presented in the noisy image is locally changed slightly, which 
will make the test more challenging so as to better evaluate the performance of each method and to 
do comparison; ii) the way of lower or increase the intensities of some specified areas can help to generate 
test images easily which contain phases with close intensities, where these generated images  
are good candidates to test the performance of different methods in classifying the close phases.  
Fig. \ref{twophase-shape} (d)--(h) are the results of methods
\cite{LNZS10,PCCB09,YBTB10,HHMSS12,CCZ13}, respectively.
Fig. \ref{twophase-shape} (i)--(j) are the results of our T-ROF method with $\mu=8$ at iterations 1 and 6 (final result),
respectively, which again clearly show the effectiveness of the
updating strategy on $\zb \tau$ given in \eqref{trof-thd}.
Obviously, except method \cite{LNZS10}, all other methods can get good results.
The quantitative results given in Table \ref{time-examples-simple} show that our method is
the fastest and gives the highest SA, which again validates
 the necessity of updating the threshold $\zb \tau$ given in \eqref{trof-thd}
 against the way of obtaining thresholds by K-means used in the SaT method \cite{CCZ13}.
%For our T-ROF method, the iteration number to find $u$ and $\tau$ are 59 and 6, respectively.

\begin{figure}[!htb]
\begin{center}
\begin{tabular}{ccc}
\includegraphics[width=\ww, height=\ww]{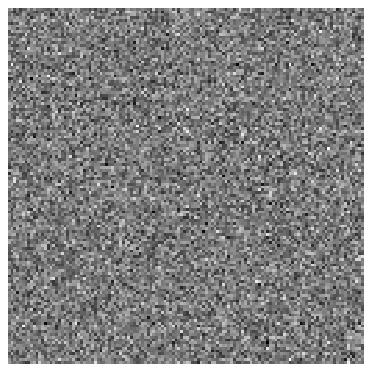} &
\includegraphics[width=\ww, height=\ww]{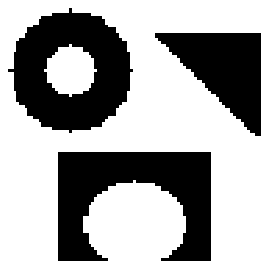} &
\includegraphics[width=\ww, height=\ww]{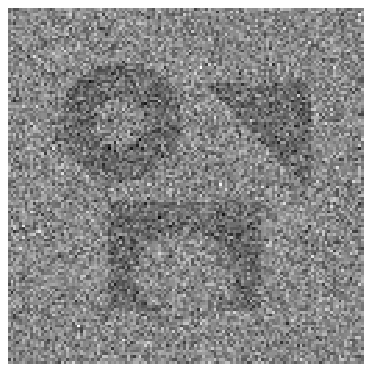} \\
(a) Gaussian noise & (b) Mask & (c) Noisy image\\
\includegraphics[width=\ww, height=\ww]{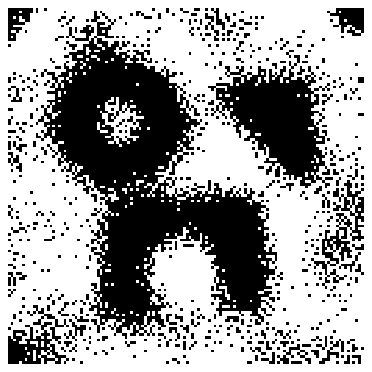} &
\includegraphics[width=\ww, height=\ww]{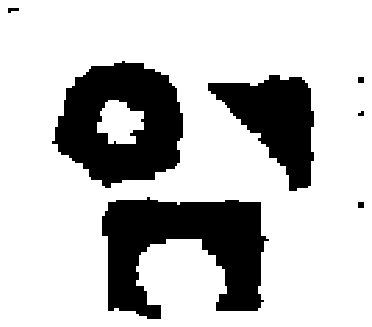} &
\includegraphics[width=\ww, height=\ww]{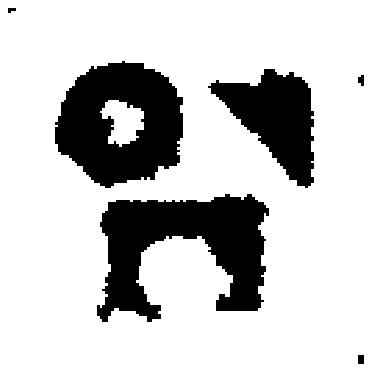} \\
(d) Li \cite{LNZS10} & (e) Pock \cite{PCCB09} & (f) Yuan \cite{YBTB10}
\end{tabular}
\begin{tabular}{cccc}
\includegraphics[width=\ww, height=\ww]{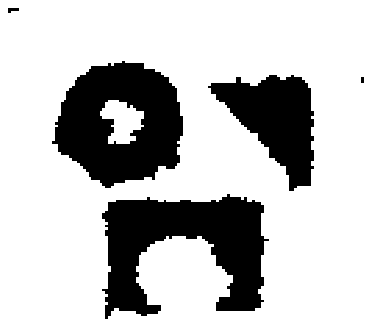} &
\includegraphics[width=\ww, height=\ww]{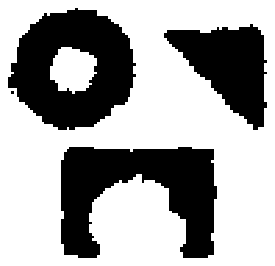} &
\includegraphics[width=\ww, height=\ww]{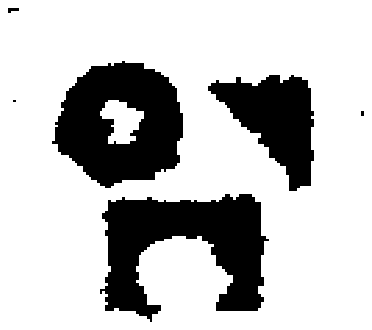} &
\includegraphics[width=\ww, height=\ww]{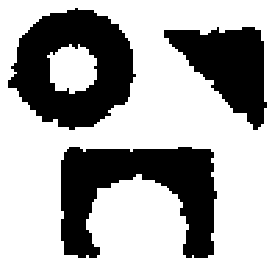} \\
(g) He \cite{HHMSS12} & (h) Cai \cite{CCZ13} & (i) Ours (Ite. 1) & (j) Ours (Ite. 6)
\end{tabular}
\end{center}
\caption{Segmentation of two-phase image with close intensities (size $128\times 128$).
(a): Gaussian noise imposed on a constant image; (b): mask; (c): noisy image generated from (a) and (b);
(d)--(h): results of methods \cite{LNZS10,PCCB09,YBTB10,HHMSS12,CCZ13}, respectively;
(i)--(j): results of our T-ROF method at iterations 1 and 6 (final result), respectively.
}\label{twophase-shape}
\end{figure}

\subsection{Multiphase Image Segmentation}

{\it Example 3. Five-phase noisy image segmentation.}
Fig. \ref{star-five} (a) and (b) are the clean image and Gaussian noisy image with mean 0 and
variance $10^{-2}$. Fig. \ref{star-five} (c)--(h) are the results of methods
\cite{LNZS10,PCCB09,YBTB10,HHMSS12,CCZ13} and our T-ROF method (with $\mu=8$),
respectively. From there results, we see that all the results are very good except the result of method \cite{LNZS10}.
From Table \ref{time-examples-simple}, we can see that the T-ROF method is the fastest.

\begin{figure}[!htb]
\begin{center}
\begin{tabular}{cccc}
\includegraphics[width=\ww, height=\ww]{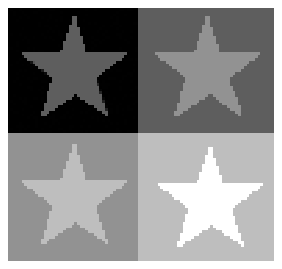} &
\includegraphics[width=\ww, height=\ww]{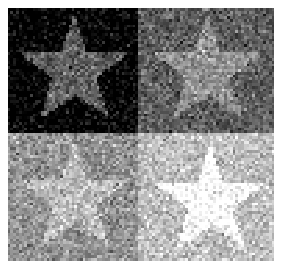} &
\includegraphics[width=\ww, height=\ww]{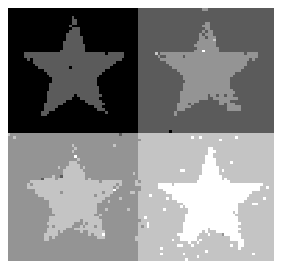} &
\includegraphics[width=\ww, height=\ww]{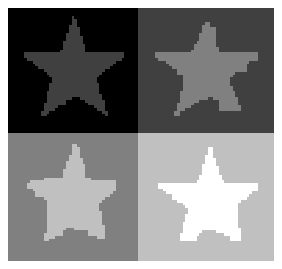} \\
(a) Clean image & (b) Noisy image & (c) Li \cite{LNZS10} & (d) Pock \cite{PCCB09} \\
\includegraphics[width=\ww, height=\ww]{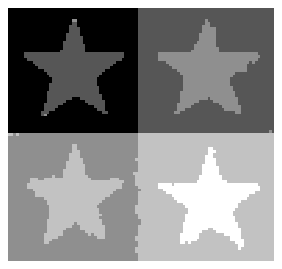} &
\includegraphics[width=\ww, height=\ww]{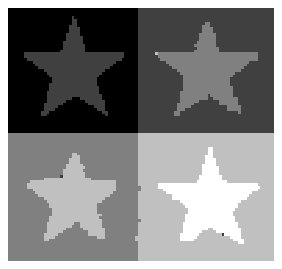} &
\includegraphics[width=\ww, height=\ww]{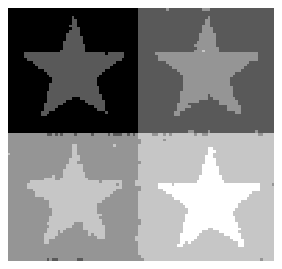} &
\includegraphics[width=\ww, height=\ww]{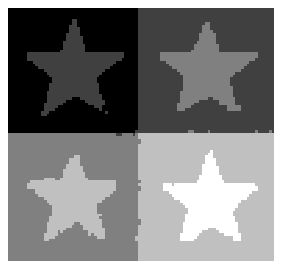} \\
(e) Yuan \cite{YBTB10} & (f) He \cite{HHMSS12} & (g) Cai \cite{CCZ13} & (h) Ours
\end{tabular}
\end{center}
\caption{Five-phase noisy cartoon image segmentation (size $91\times 96$).
(a): clean image; (b): noisy image of (a); (c)--(h): results of methods
\cite{LNZS10,PCCB09,YBTB10,HHMSS12,CCZ13} and our T-ROF method, respectively.
}\label{star-five}
\end{figure}

{\it Example 4. Four-phase gray and white matter segmentation for a brain MRI image.}
In this example, we test the four-phase brain MRI image used in \cite{PCCB09},
see Fig. \ref{brain} (a). The gray and white matter segmentation for this kind of image is
very important in medical imaging. Fig. \ref{brain} (b)--(g) are the results of methods
\cite{LNZS10,PCCB09,YBTB10,HHMSS12,CCZ13} and our T-ROF method (with $\mu=40$), respectively.
We can see that all the methods work well for this kind of image. In particular, the T-ROF method with 11 $\zb\tau$-value
updates is faster than other methods, e.g., three times faster than the algorithm of Pock et al. \cite{PCCB09}
with assigned parameters. Note that the SaT method \cite{CCZ13} is also very fast, due to the fact that it is akin to the
T-ROF method.

\begin{figure}[!htb]
\begin{center}
\begin{tabular}{ccc}
\includegraphics[width=\ww, height=35mm]{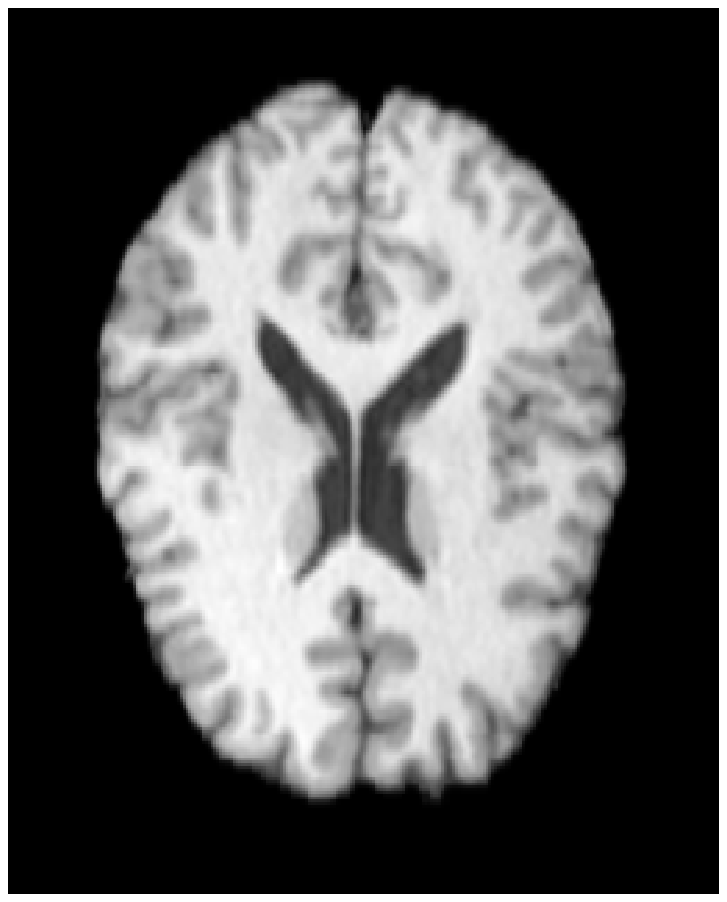} &
\includegraphics[width=\ww, height=35mm]{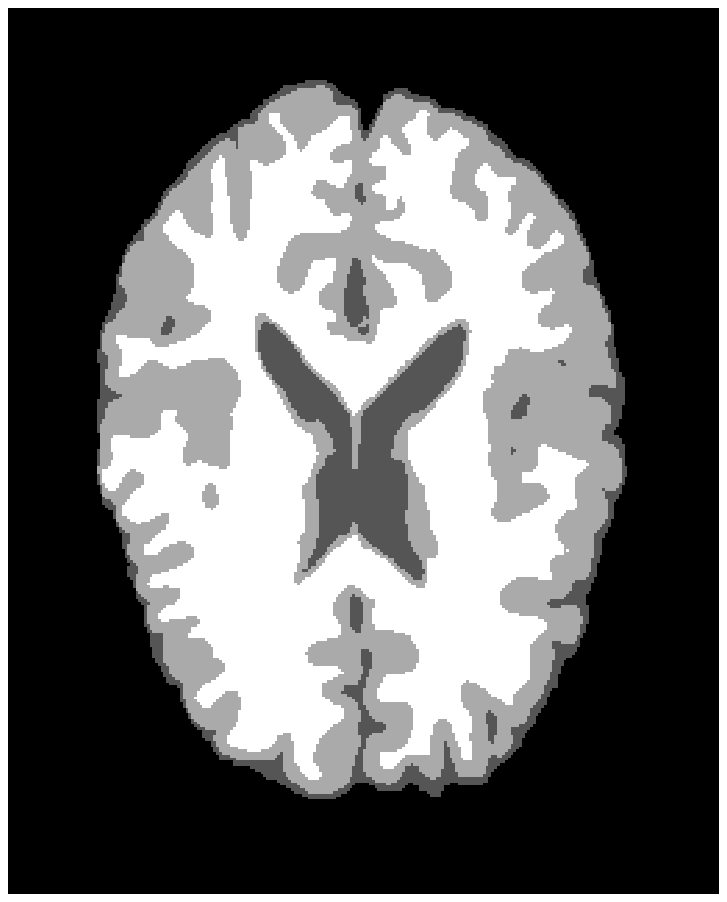} &
\includegraphics[width=\ww, height=35mm]{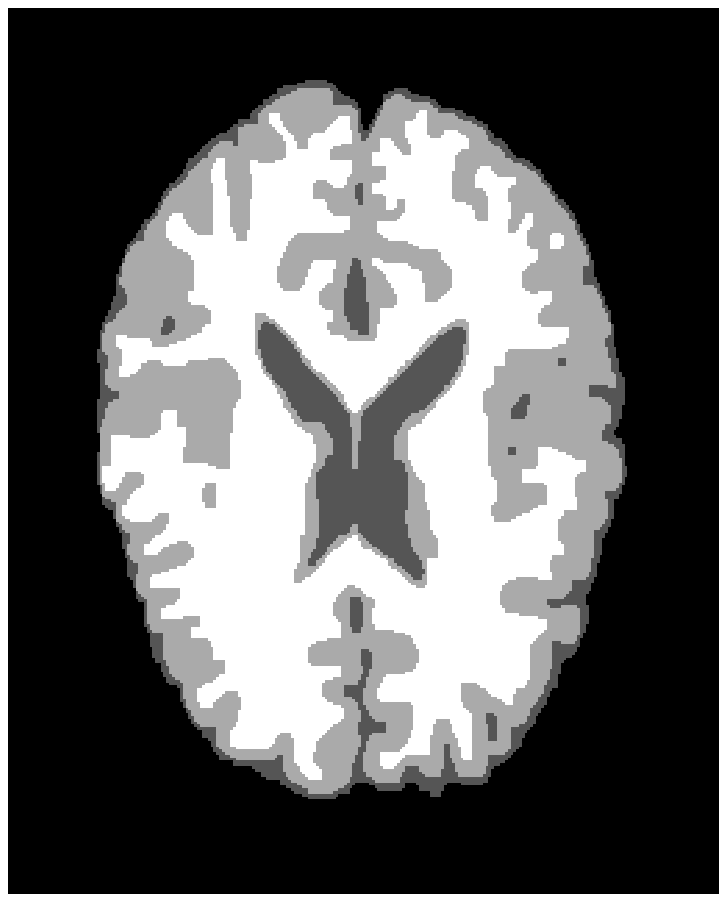} \\
(a) Given image & (b) Li \cite{LNZS10} & (c) Pock \cite{PCCB09}
\end{tabular}
\begin{tabular}{cccc}
\includegraphics[width=\ww, height=35mm]{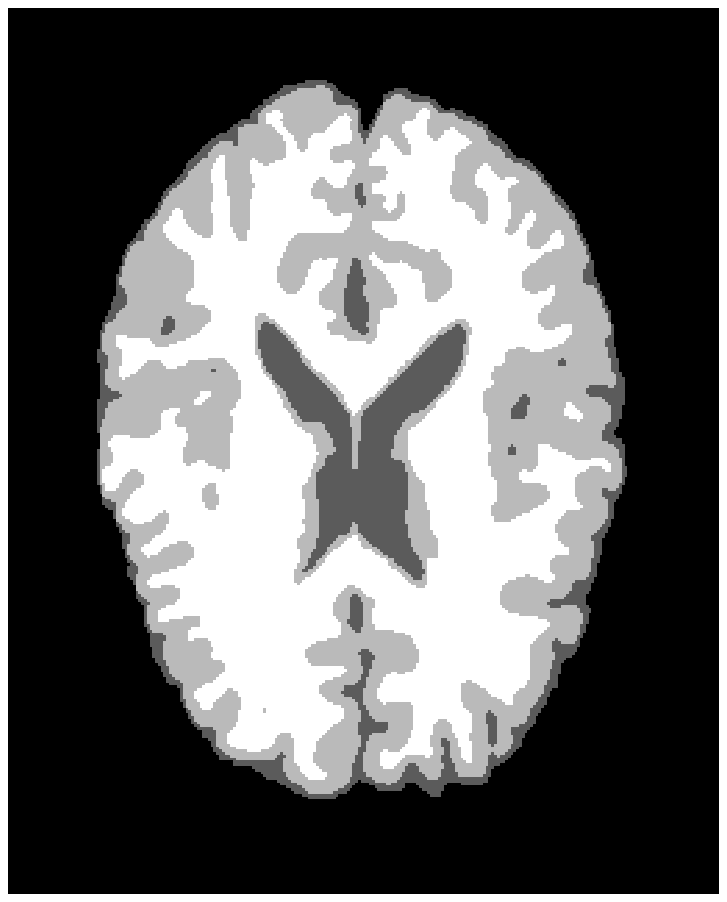} &
\includegraphics[width=\ww, height=35mm]{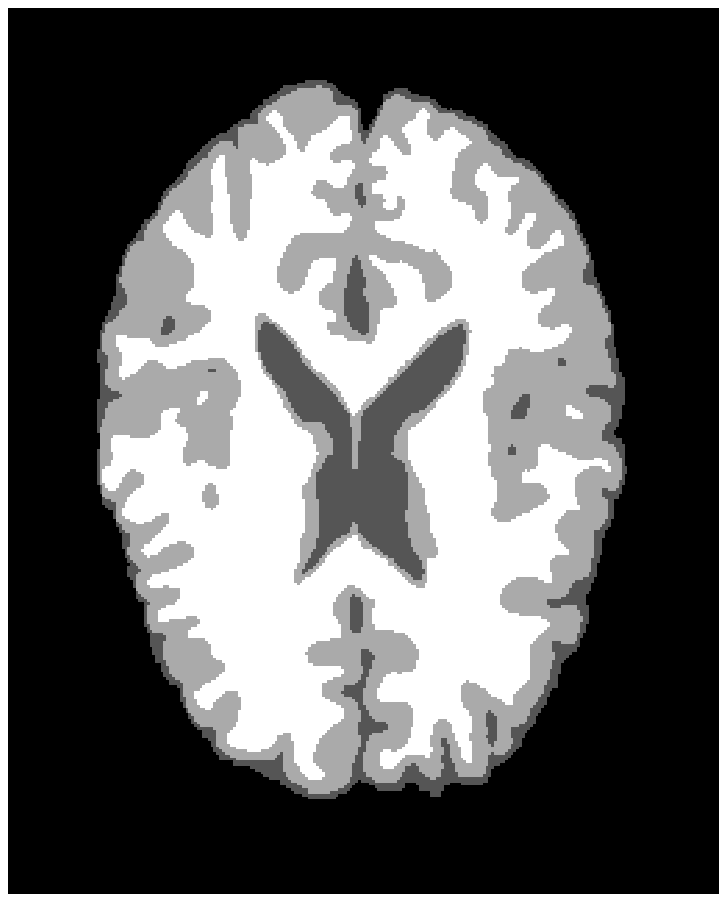} &
\includegraphics[width=\ww, height=35mm]{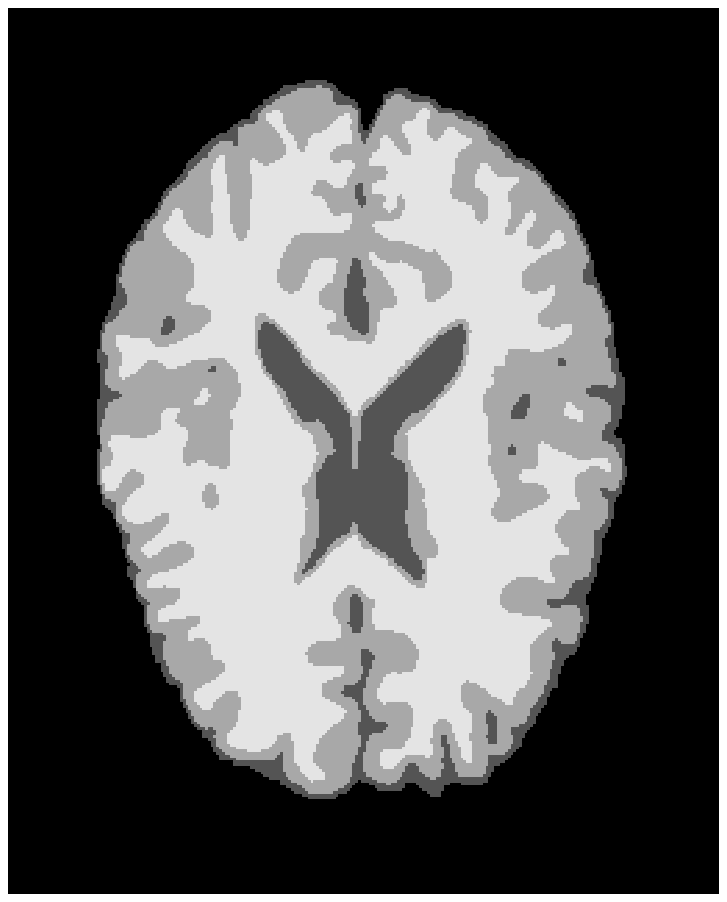} &
\includegraphics[width=\ww, height=35mm]{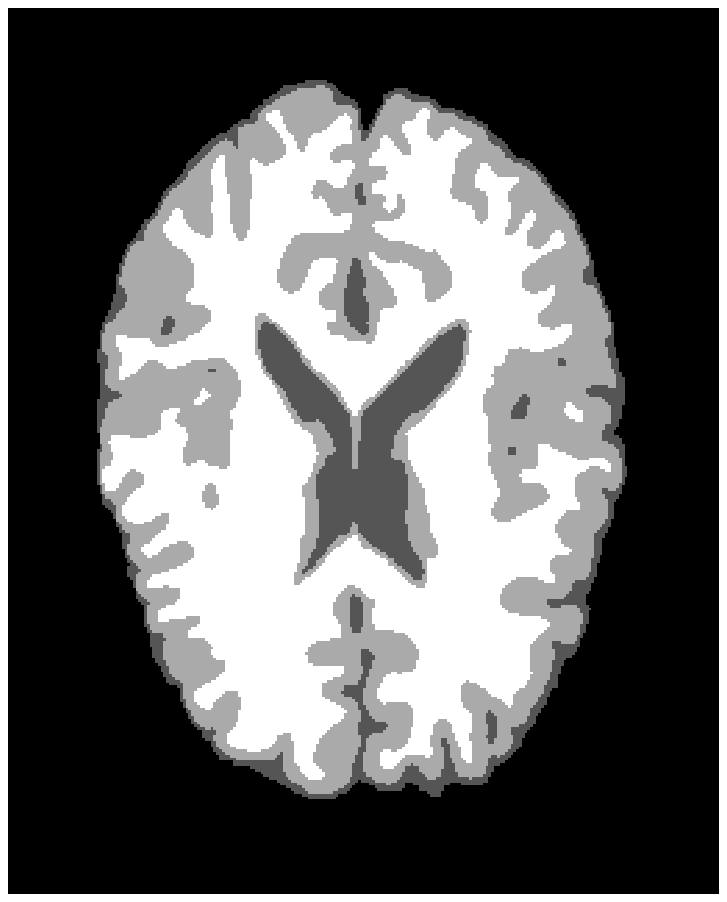} \\
(d) Yuan \cite{YBTB10} & (e) He \cite{HHMSS12} & (f) Cai \cite{CCZ13} & (g) Ours
\end{tabular}
\end{center}
\caption{Four-phase segmentation of MRI image (size $319\times 256$).
(a): given image; (b)--(g): results of methods \cite{LNZS10,PCCB09,YBTB10,HHMSS12,CCZ13} and our T-ROF method,
respectively.
}\label{brain}
\end{figure}

\begin{table}[!h]
\centering \caption{Quantitative comparison: $\lambda/\mu$, iteration (Ite.) steps, CPU time in seconds, and SA in Examples (Exa.) 1--4.
The iteration steps of our T-ROF method e.g. 418 (6) mean that 418 and 6 iterations are respectively executed to find $u$ and $\zb \tau$ in 
Algorithm \ref{alg:t-rof}.}
\begin{tabular}{|c|c|r|r|r|r|r|r|}
\cline{3-8}
 \multicolumn{2}{r|}{} & Li \cite{LNZS10} & Pock \cite{PCCB09} & Yuan \cite{YBTB10} & He \cite{HHMSS12} &
 Cai \cite{CCZ13} & Our method \\ \hline
 \multirow{4}{*}{\rotatebox{90}{Exa. 1}} & $\lambda/\mu$ & 20 & 2 & 8 & 5 & 2 & 1 \\ \cline{2-8}
& Ite. & 50 & 150 & 47 & 100 & 162 & 418 (6) \\ \cline{2-8}
& Time & 1.71 & 6.36 & 2.55 & 13.76 & 5.42 & 8.34 \\ \cline{2-8}
&SA & 0.6918 & 0.8581 & 0.6915 & 0.9888 & 0.9878 & 0.9913 \\ \hline \hline
 \multirow{4}{*}{\rotatebox{90}{Exa. 2}} & $\lambda/\mu$ & 200 & 50 & 15 & 120 & 7 & 8 \\ \cline{2-8}
& Ite. & 100 & 100 & 267 & 150 & 81 & 59 (6) \\ \cline{2-8}
& Time & 0.73 & 0.97 & 2.11 & 6.66 & 1.21 & 0.38 \\ \cline{2-8}
&SA & 0.7867 & 0.9658 & 0.9598 & 0.9663 & 0.9816 & 0.9845 \\ \hline \hline
 \multirow{4}{*}{\rotatebox{90}{Exa. 3}} & $\lambda/\mu$ & 120 & 40 & 15 & 50 & 15 & 8 \\ \cline{2-8}
& Ite. & 100 & 50 & 101 & 100 & 57 & 61 (4) \\ \cline{2-8}
& Time & 0.87 & 0.71 & 0.99 & 4.11 & 0.46 & 0.32 \\ \cline{2-8}
&SA & 0.9729 & 0.9826 & 0.9819 & 0.9872 & 0.9827 & 0.9831 \\ \hline \hline
 \multirow{4}{*}{\rotatebox{90}{Exa. 4}} & $\lambda/\mu$ & 200 & 100 & 20 & 200 & 40 & 40 \\ \cline{2-8}
& Ite. & 100 & 50 & 46 & 50 & 46 & 20 (11) \\ \cline{2-8}
& Time & 8.49 & 7.41 & 5.24 & 22.13 & 2.75 & 1.96 \\ \hline
\end{tabular}
\label{time-examples-simple}
\end{table}

{\it Example 5. Stripe image.}
In this example, we test methods on segmenting the noisy stripe image in Fig. \ref{thirtyphase} (b),
which is generated by imposing Gaussion noise with mean 0 and variance $10^{-3}$ on
the clean image Fig. \ref{thirtyphase} (a) with 30 stripes.
The results are generated by segmenting Fig. \ref{thirtyphase} (b) into five, ten, and fifteen phases, respectively.
The rows two to seven of Fig. \ref{thirtyphase} are the results of methods
\cite{LNZS10,PCCB09,YBTB10,HHMSS12,CCZ13} and our T-ROF method (with $\mu=8$),
respectively. The quantitative comparison is shown in Table \ref{time-example-thirtyphase},
from which we can see that methods \cite{PCCB09,HHMSS12} and our T-ROF method give much better
results in terms of SA; note, importantly, that our method is always the fastest compared with methods \cite{PCCB09,HHMSS12}.
Moreover, Table \ref{time-example-thirtyphase} shows clearly the great advantage of the T-ROF method and SaT method \cite{CCZ13}
in computation time: their computation time is independent to
the required number of phases $K$, whereas this is not the case for other methods
(as the number of phases goes larger, their computation time significantly increases inevitably).
From Table \ref{time-example-thirtyphase}, we also see that the T-ROF method gives much better results than the SaT method,
which again shows the excellent performance and necessity of updating the threshold $\zb \tau$
benefited from the rule proposed in \eqref{trof-thd}.

\begin{table}[!h]
\centering \caption{Quantitative comparison: $\lambda/\mu$, iteration (Ite.) steps, CPU time in seconds, and SA in Example 5.
The iteration steps of our T-ROF method e.g. 84 (4) means that 84 and 4 iterations 
are respectively executed to find $u$ and $\zb \tau$ in Algorithm \ref{alg:t-rof}.}
\begin{tabular}{|c|c|r|r|r|r|r|r|}
\cline{3-8}
 \multicolumn{2}{r|}{} & Li \cite{LNZS10} & Pock \cite{PCCB09} & Yuan \cite{YBTB10} & He \cite{HHMSS12} &
 Cai \cite{CCZ13} & Our method \\ \hline
 \multirow{4}{*}{\rotatebox{90}{5 phases}} & $\lambda/\mu$ & 80 & 100 & 10 & 50 & 10 & 8 \\ \cline{2-8}
& Ite. & 100 & 100 & 87 & 100 & 41 & 84 (4) \\ \cline{2-8}
& Time & 3.87 & 6.25 & 4.33 & 16.75 & 1.33 & 1.39 \\ \cline{2-8}
&SA & 0.9946 & 0.9965 & 0.9867 & 0.9968 & 0.9770 & 0.9986 \\ \hline \hline
 \multirow{4}{*}{\rotatebox{90}{10 phases}} & $\lambda/\mu$ & 80 & 100 & 10 & 50 & 10 & 8 \\ \cline{2-8}
& Ite. & 100 & 100 & 102 & 100 & 41 & 84 (5) \\ \cline{2-8}
& Time & 7.71 & 15.41 & 9.79 & 38.52 & 2.11 & 2.33 \\ \cline{2-8}
&SA & 0.8545 & 0.9984 & 0.9715 & 0.9848 & 0.8900 & 0.9967 \\ \hline \hline
 \multirow{4}{*}{\rotatebox{90}{15 phases}} & $\lambda/\mu$ & 80 & 100 & 10 & 50 & 10 & 8 \\ \cline{2-8}
& Ite. & 100 & 100 & 208 & 100 & 41 & 84 (5) \\ \cline{2-8}
& Time & 11.56 & 28.21 & 33.21 & 63.67 & 3.06 & 3.74 \\ \cline{2-8}
&SA & 0.7715 & 0.9993 & 0.9730 & 0.9904 & 0.5280 & 0.9933 \\ \hline
\end{tabular}
\label{time-example-thirtyphase}
\end{table}

\begin{figure}[!htb]
\begin{center}
{\setlength\fboxsep{0pt}
\begin{tabular}{cc}
\fbox{\includegraphics[width=39mm, height=20mm]{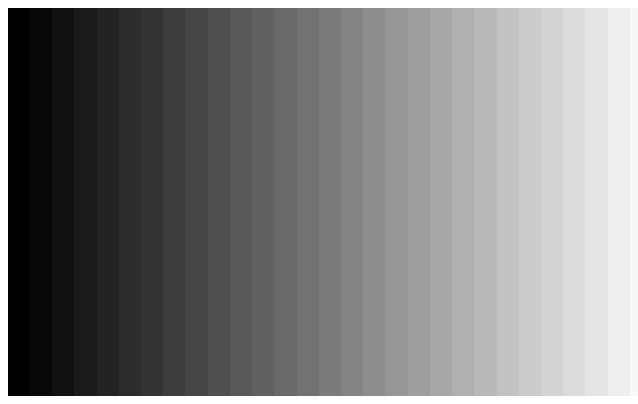}} &
\fbox{\includegraphics[width=39mm, height=20mm]{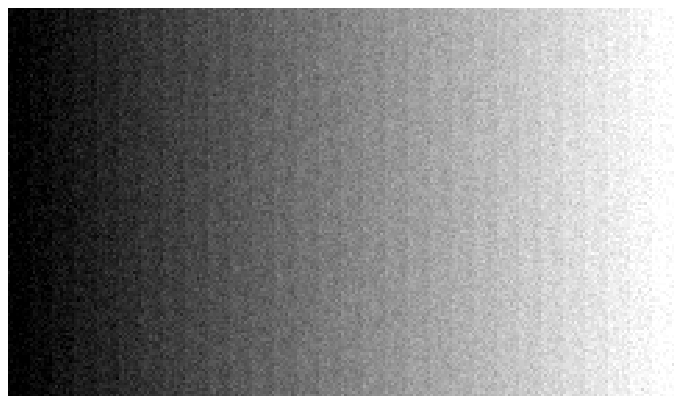} }
\end{tabular}
\begin{tabular}{ccc}
\fbox{\includegraphics[width=39mm, height=20mm]{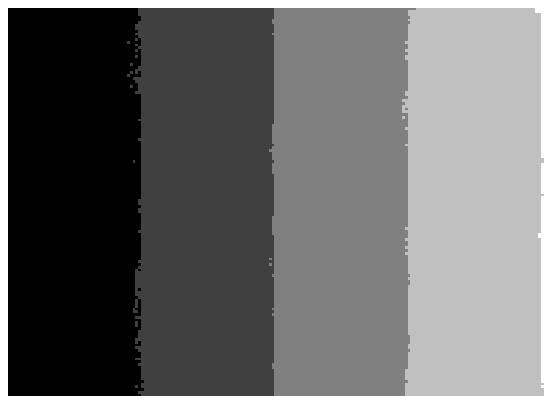}} &
\fbox{\includegraphics[width=39mm, height=20mm]{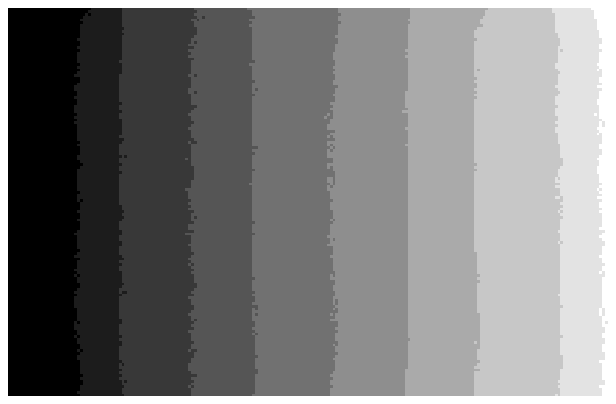}} &
\fbox{\includegraphics[width=39mm, height=20mm]{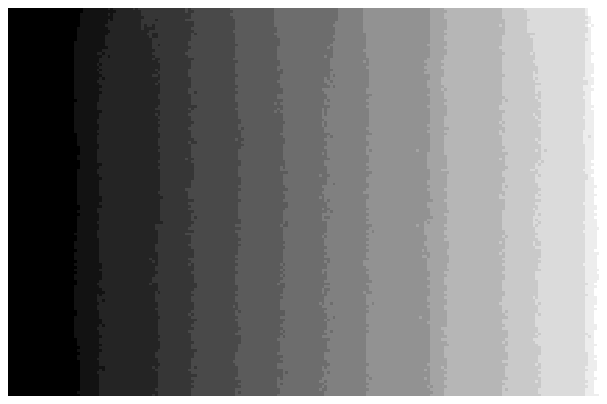}} \\
\fbox{\includegraphics[width=39mm, height=20mm]{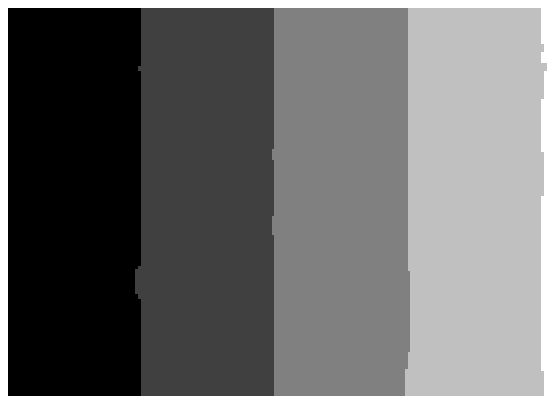}} &
\fbox{\includegraphics[width=39mm, height=20mm]{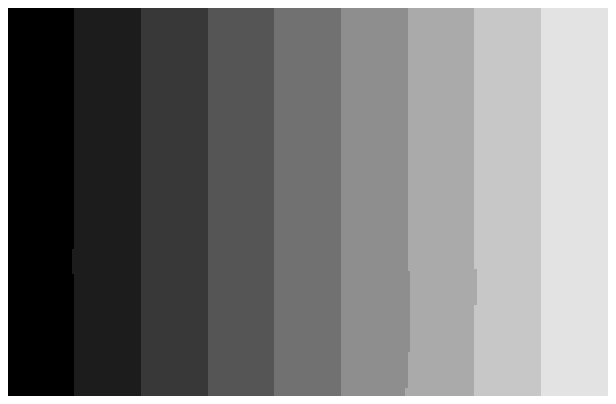}} &
\fbox{\includegraphics[width=39mm, height=20mm]{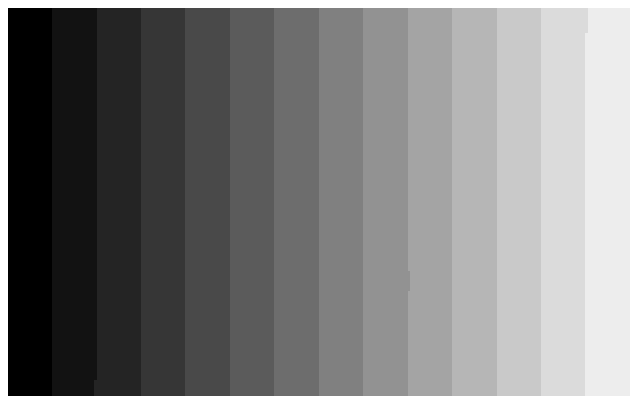}} \\
\fbox{\includegraphics[width=39mm, height=20mm]{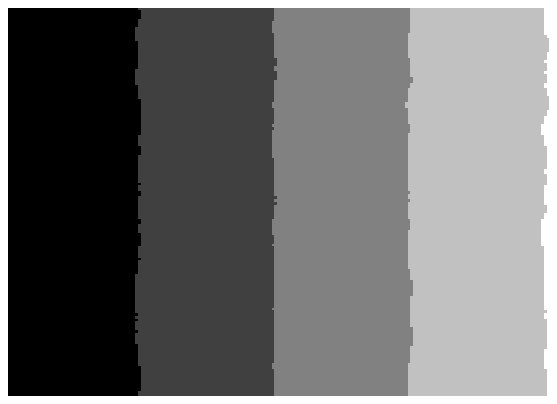}} &
\fbox{\includegraphics[width=39mm, height=20mm]{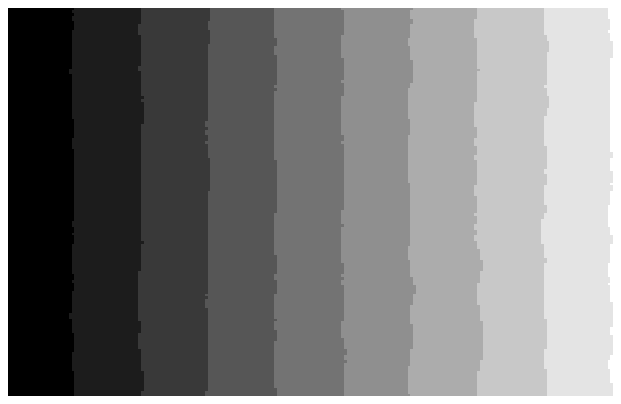}} &
\fbox{\includegraphics[width=39mm, height=20mm]{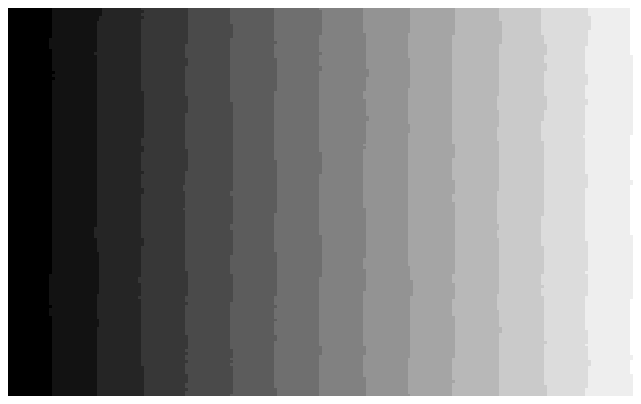}} \\
\fbox{\includegraphics[width=39mm, height=20mm]{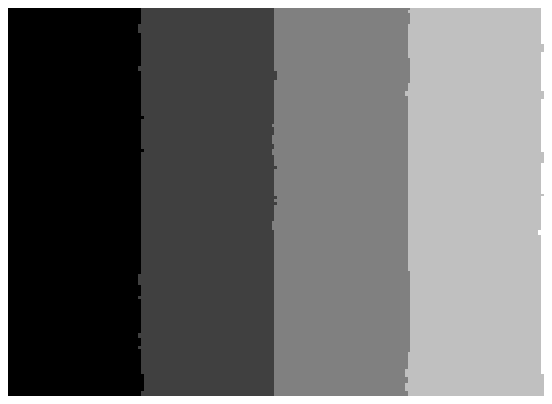}} &
\fbox{\includegraphics[width=39mm, height=20mm]{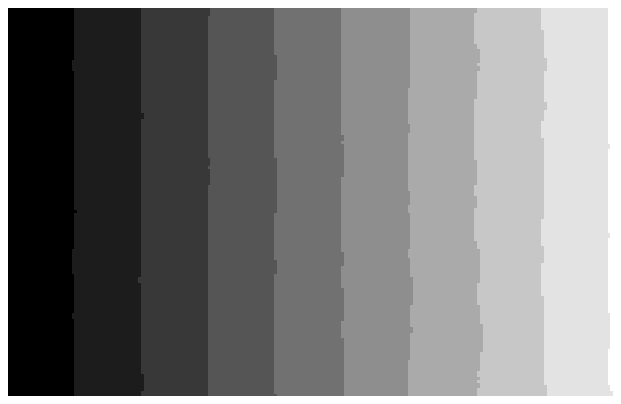}} &
\fbox{\includegraphics[width=39mm, height=20mm]{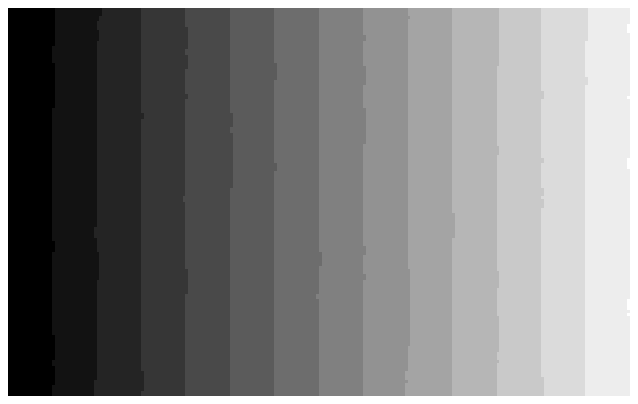}} \\
\fbox{\includegraphics[width=39mm, height=20mm]{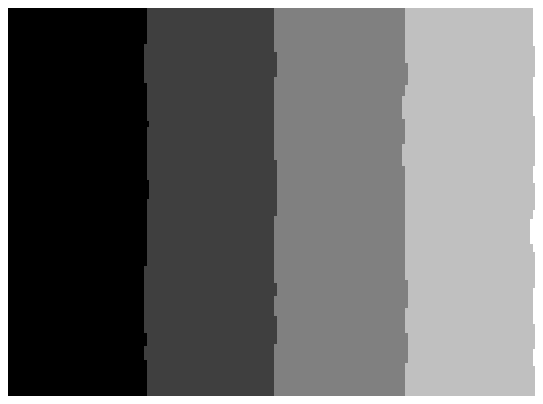}} &
\fbox{\includegraphics[width=39mm, height=20mm]{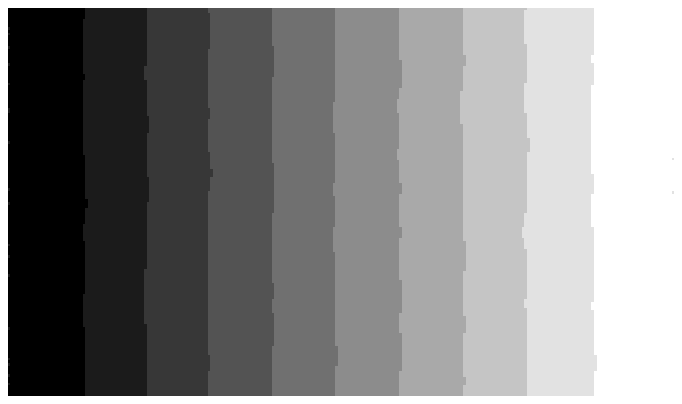}} &
\fbox{\includegraphics[width=39mm, height=20mm]{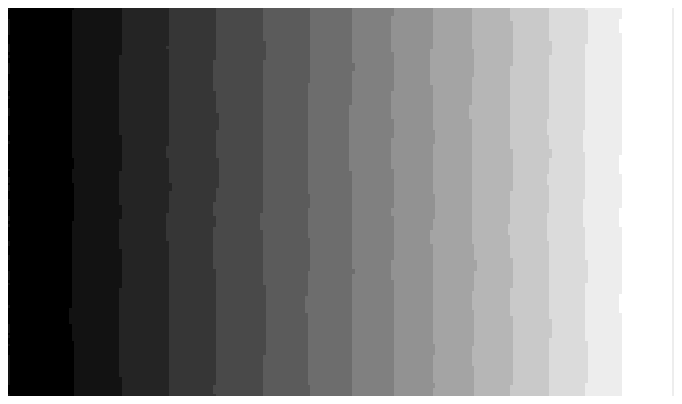}} \\
\fbox{\includegraphics[width=39mm, height=20mm]{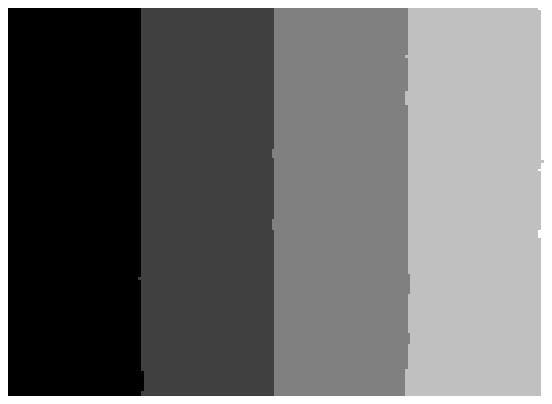}} &
\fbox{\includegraphics[width=39mm, height=20mm]{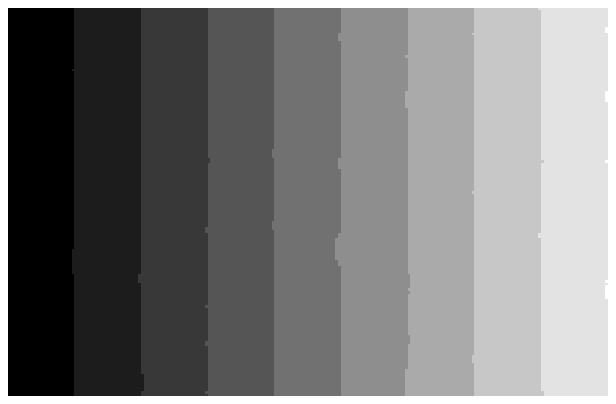}} &
\fbox{\includegraphics[width=39mm, height=20mm]{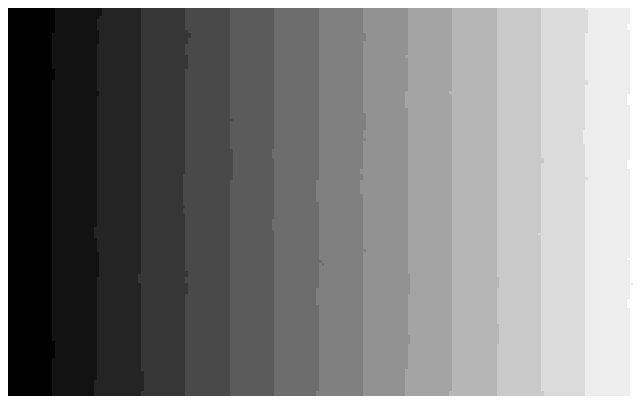}} \\
\end{tabular}}
\end{center}
\caption{Stripe image segmentation (size $140\times 240$).
Row one: clean image (left) and noisy image (right); Rows two to seven: results
of methods \cite{LNZS10,PCCB09,YBTB10,HHMSS12,CCZ13} and our T-ROF method, respectively.
}\label{thirtyphase}
\end{figure}

{\it Example 6. Three-phase image containing phases with close intensities.}
In this example, we test a three-phase image, where two phases of it have very close intensities.
The test image in Fig. \ref{threephase-close} (c) is generated using the same way as that in Example 2 with
Gaussian noise of mean 0 and variance $10^{-2}$, and the factors used in the black and
white color parts in the mask are 0.1 and 0.6, respectively.
For the results of methods \cite{LNZS10,PCCB09,YBTB10,HHMSS12,CCZ13}, in order to reveal the performance of these methods
clearly, we give two representative results for each method compared using different regularization
parameters, see Fig. \ref{threephase-close} (d)--(m).
From Fig. \ref{threephase-close}, we can see that our T-ROF method gives the best result, see 
Fig. \ref{threephase-close} (o) (with $\mu=8$). In detail, we notice that method \cite{LNZS10} gives very poor results, see Fig. \ref{threephase-close} (d)--(e).
For methods \cite{PCCB09, HHMSS12}, no matter how their parameters are tuned, they all cannot
achieve good results as the result of our T-ROF method given in Fig. \ref{threephase-close} (o). 
More specifically, methods \cite{PCCB09, HHMSS12} either give results
separating different phases unclearly, or give results not removing the noise successfully.
The main reason is that, the regularization parameters in methods \cite{PCCB09, HHMSS12} are constant, see 
the parameter $\lambda$ used in the PCMS model \eqref{pcms}; however, constant parameter used to penalize all the phases equally is
obviously not appropriate for this case.
The results of methods \cite{YBTB10,CCZ13}, Fig. \ref{threephase-close} (h) and (l), are better than the results of methods
\cite{LNZS10,PCCB09,HHMSS12}. Again, after comparing
Fig. \ref{threephase-close} (h), (l) and Fig. \ref{threephase-close} (o) from visual-validation and the quantitative results given in
Table \ref{time-examples-complex}, we see that our T-ROF method gives the best result
in terms of segmentation quality and computation time.

\begin{figure}[!htb]
\begin{center}
\begin{tabular}{cccc}
\includegraphics[width=\ww, height=\ww]{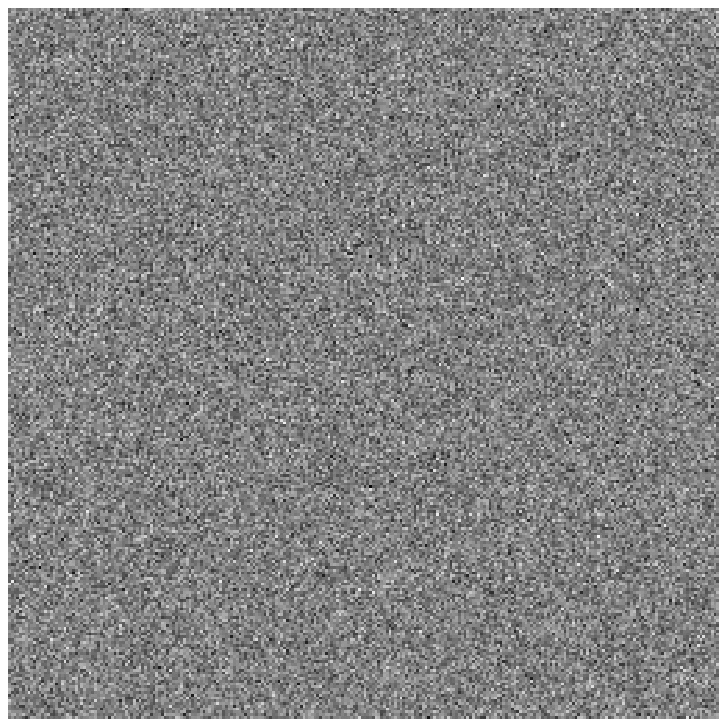} &
\includegraphics[width=\ww, height=\ww]{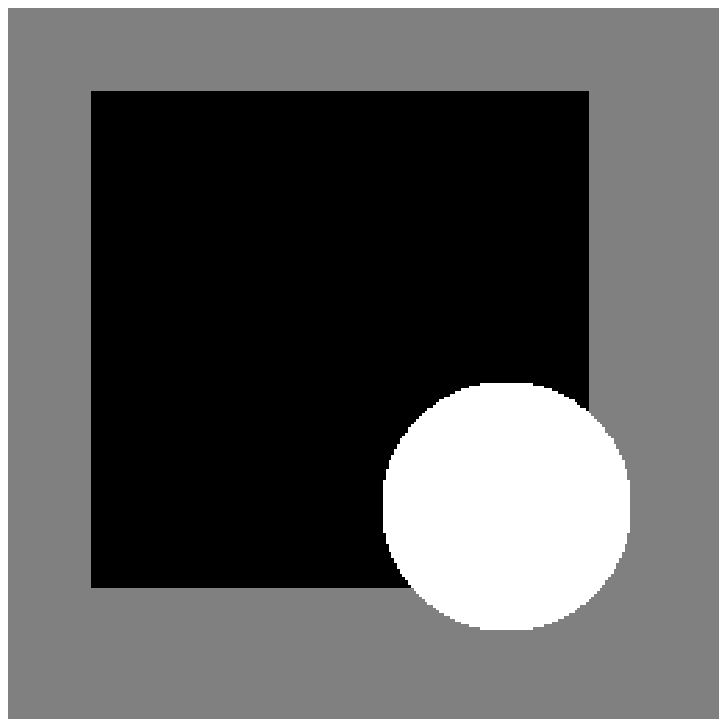} &
\includegraphics[width=\ww, height=\ww]{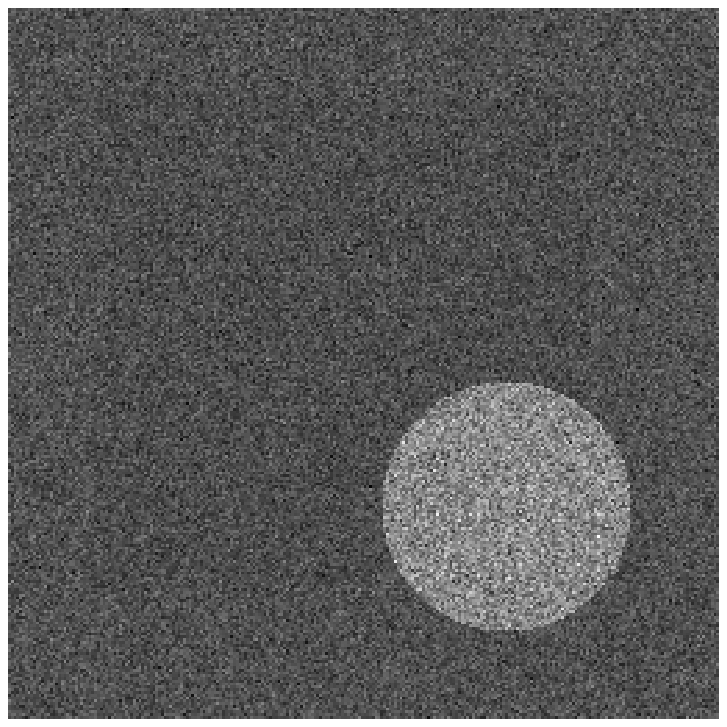} \\
(a) Gaussian noise & (b) Mask & (c) Noisy image\\
\end{tabular}
\begin{tabular}{cccc}
\includegraphics[width=\ww, height=\ww]{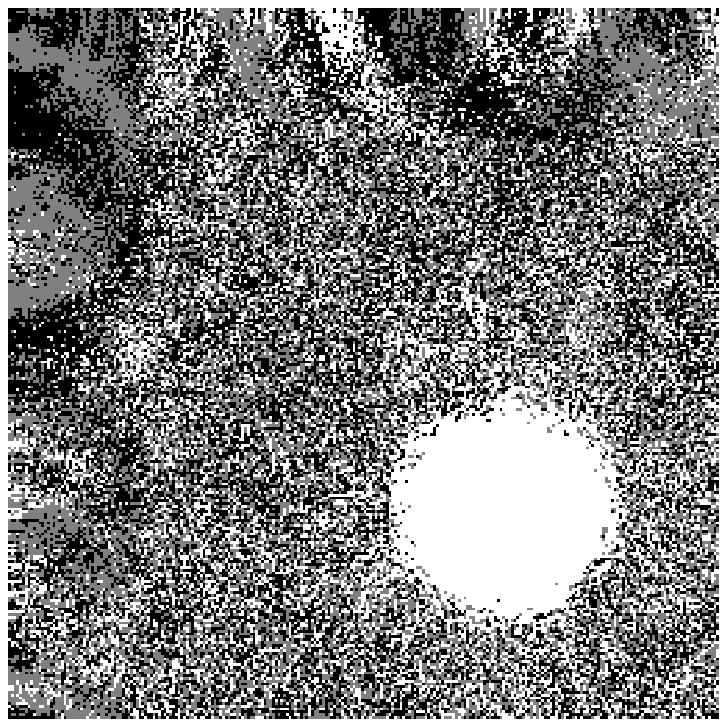} &
\includegraphics[width=\ww, height=\ww]{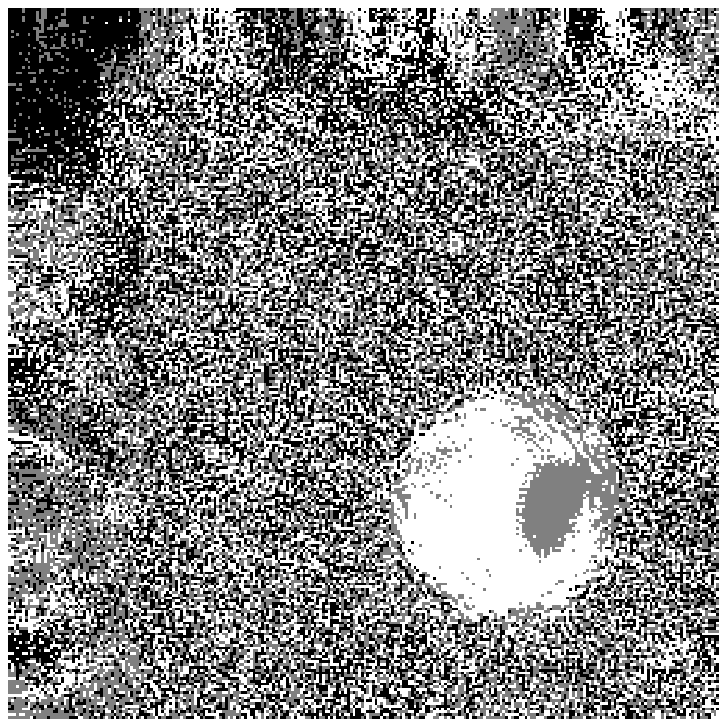} &
\includegraphics[width=\ww, height=\ww]{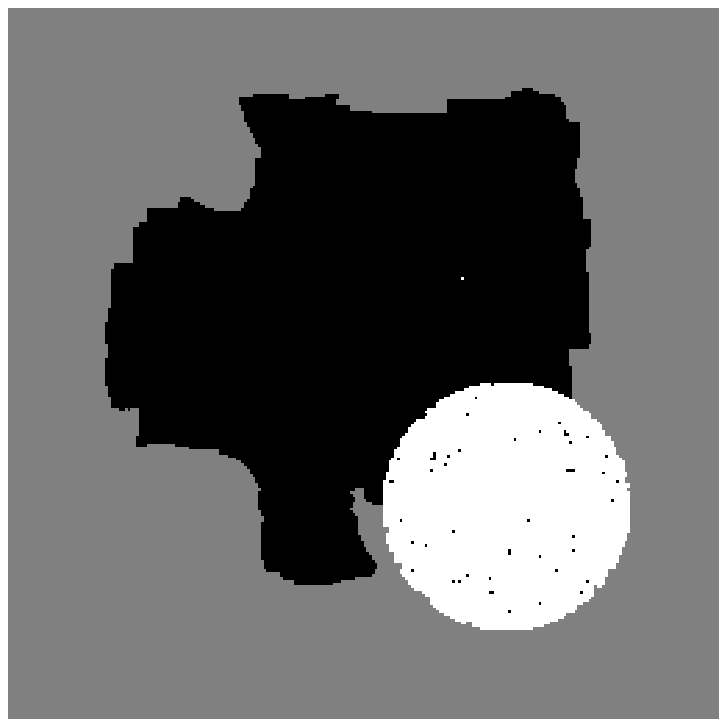} &
\includegraphics[width=\ww, height=\ww]{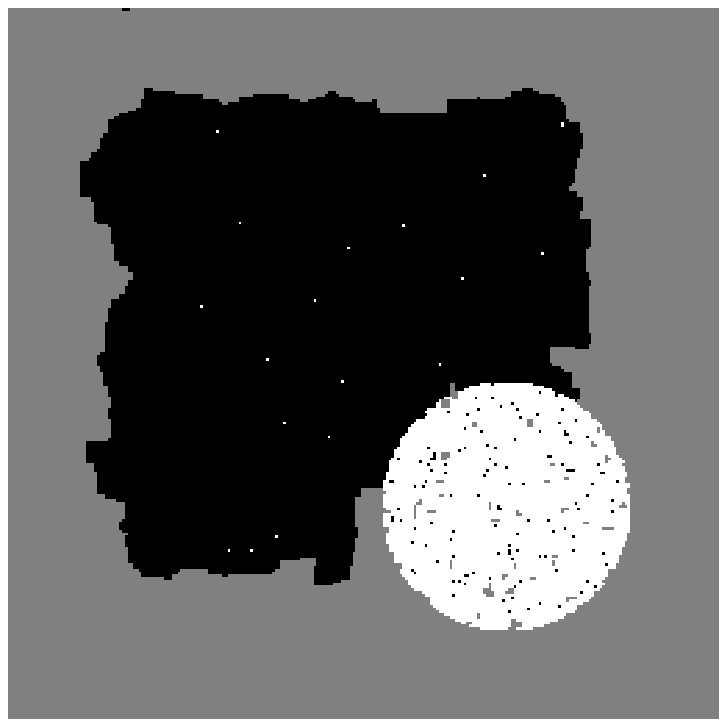} \\
(d) Li \cite{LNZS10} & (e) Li \cite{LNZS10} & (f) Pock \cite{PCCB09} & (g) Pock \cite{PCCB09} \\
\includegraphics[width=\ww, height=\ww]{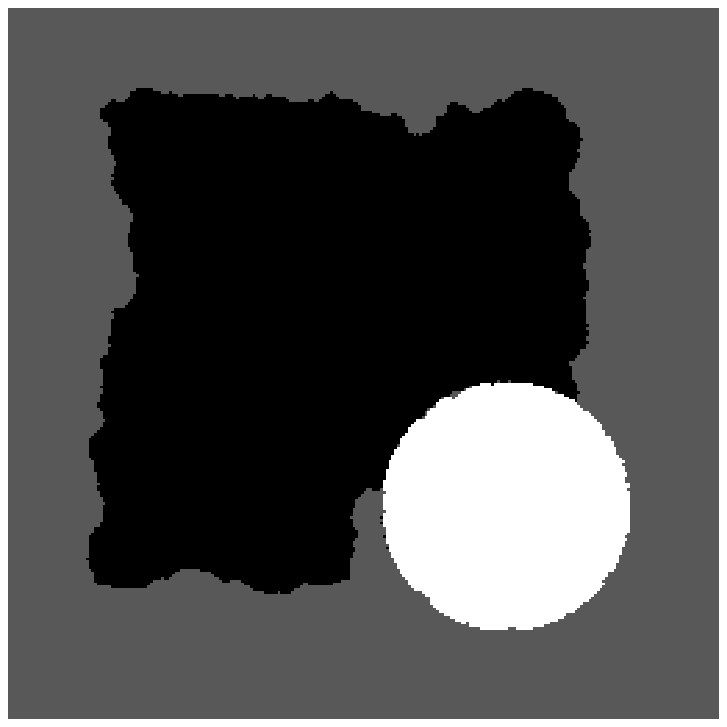} &
\includegraphics[width=\ww, height=\ww]{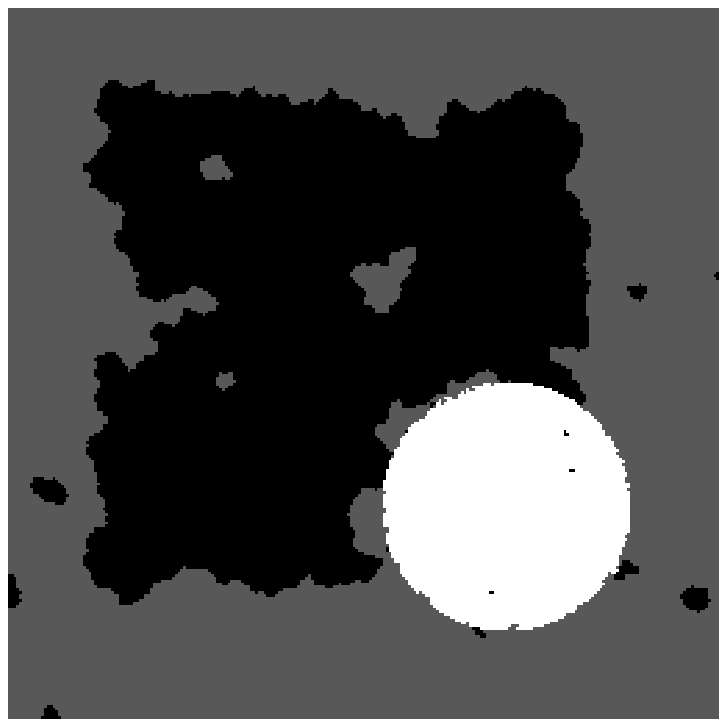} &
\includegraphics[width=\ww, height=\ww]{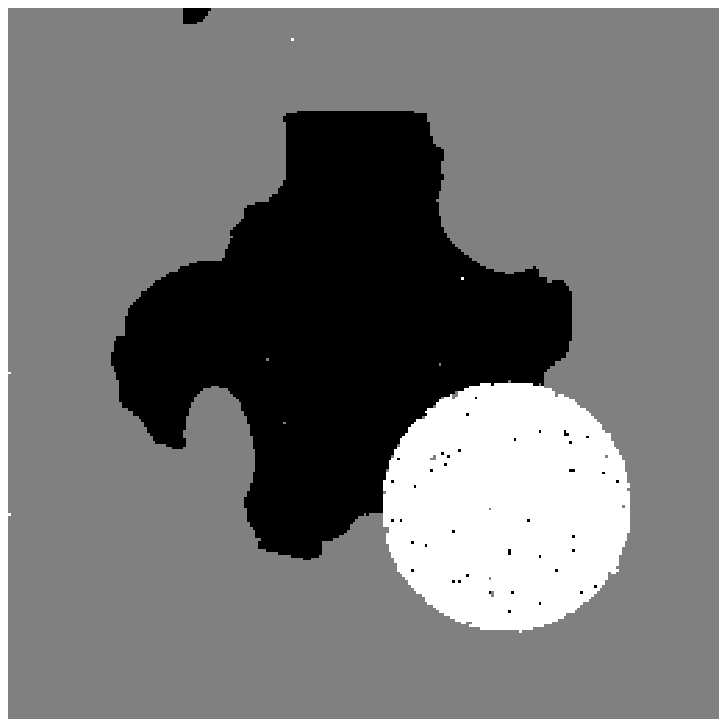} &
\includegraphics[width=\ww, height=\ww]{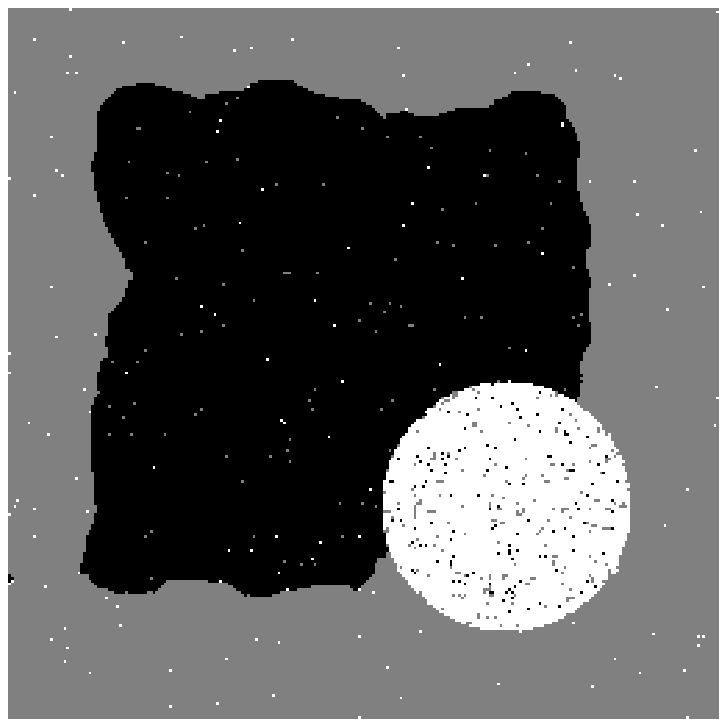} \\
(h) Yuan \cite{YBTB10} & (i) Yuan \cite{YBTB10} & (j) He \cite{HHMSS12} &
(k) He \cite{HHMSS12}\\
\includegraphics[width=\ww, height=\ww]{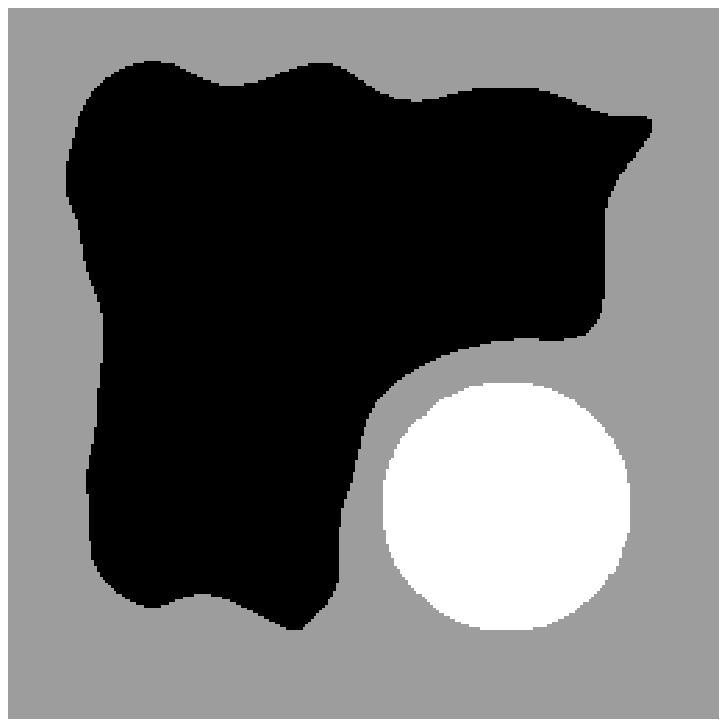} &
\includegraphics[width=\ww, height=\ww]{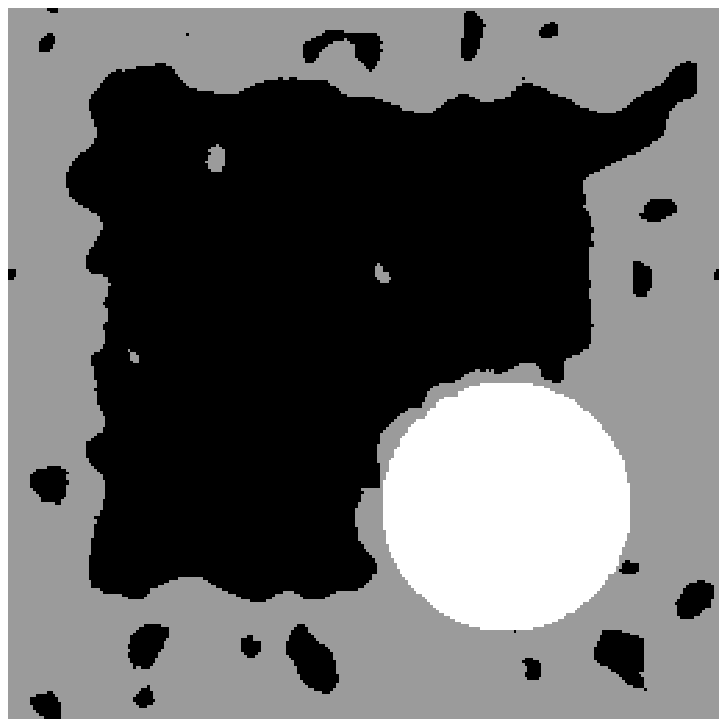} &
\includegraphics[width=\ww, height=\ww]{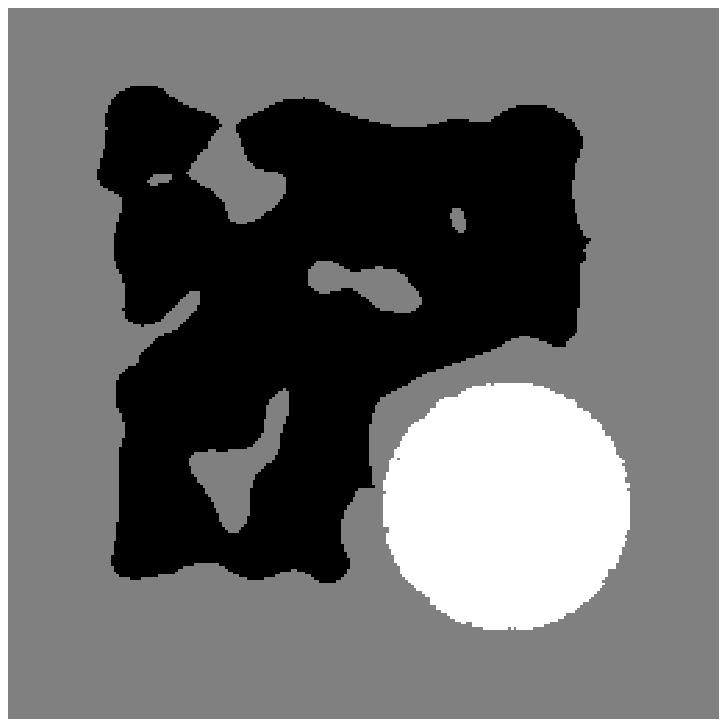} &
\includegraphics[width=\ww, height=\ww]{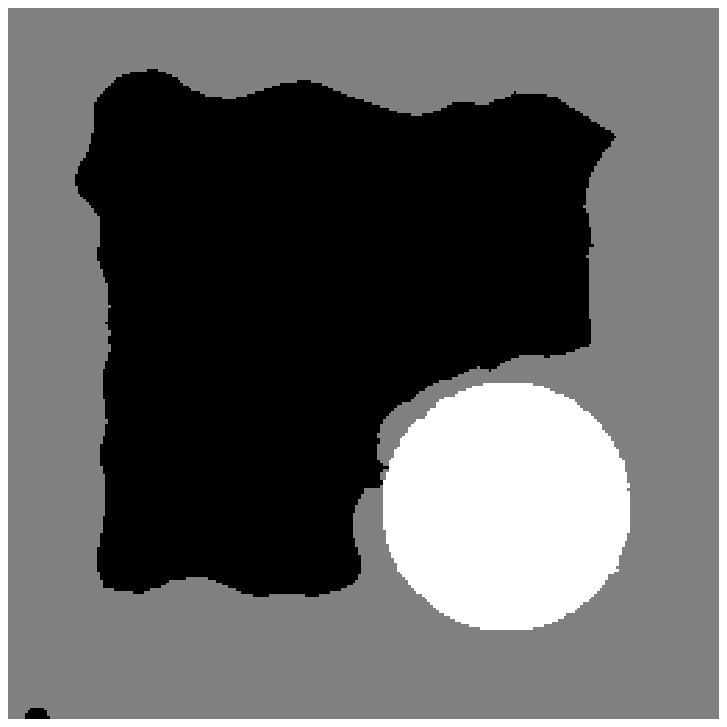} \\
(l) Cai \cite{CCZ13} & (m) Cai \cite{CCZ13} & (n) Ours (Ite. 1) & (o) Ours (Ite. 6)
\end{tabular}
\end{center}
\caption{Segmentation of three-phase image containing phases with close intensities
(size $256\times 256$).
(a): Gaussian noise imposed on a constant image; (b): three-phase mask;
(c): noisy image generated from (a) and (b);
(d)--(e): results of \cite{LNZS10} with $\lambda = 50$ and 100, respectively;
(f)--(g): results of \cite{PCCB09} with $\lambda = 70$ and 100, respectively;
(h)--(i): results of \cite{YBTB10} with $\lambda = 15$ and 20, respectively;
(j)--(k): results of \cite{HHMSS12} with $\lambda = 100$ and 200, respectively;
(l)--(m): results of \cite{CCZ13} with $\mu = 3$ and 10, respectively;
(n)--(o): results of our T-ROF method, with $\mu = 8$, at iterations 1 and 6 (final result), respectively.
}\label{threephase-close}
\end{figure}
	
{\it Example 7. Four-phase image containing phases with close intensities.}	
In order to show the powerful of our T-ROF method in handing images containing phases with close intensities, 
in this example, we test these methods on a four-phase image where each two phases with close intensities.
Fig. \ref{fourphase-close} (a) and (b) are respectively the clean image and the noisy image
generated by adding Gaussian noise with mean 0 and variance $3\times 10^{-2}$.
Similar to Example 6, this example also provides two results for each method compared using different
representative parameters, see Fig. \ref{fourphase-close} (c)--(l).
From Fig. \ref{fourphase-close} and Table \ref{time-examples-complex} (which gives quantitative results), we see that
methods \cite{LNZS10,PCCB09,YBTB10,HHMSS12} all give poor results compared with
our result in terms of segmentation quality and computation time.
In particular, our method gives much better results than that of its akin SaT method \cite{CCZ13},
see Fig. \ref{fourphase-close} (k) and (n), which further verifies the excellent performance of updating the threshold $\zb \tau$
using the rule proposed in \eqref{trof-thd}.

\begin{figure}[!htb]
\begin{center}
\begin{tabular}{cc}
\includegraphics[width=\ww, height=\ww]{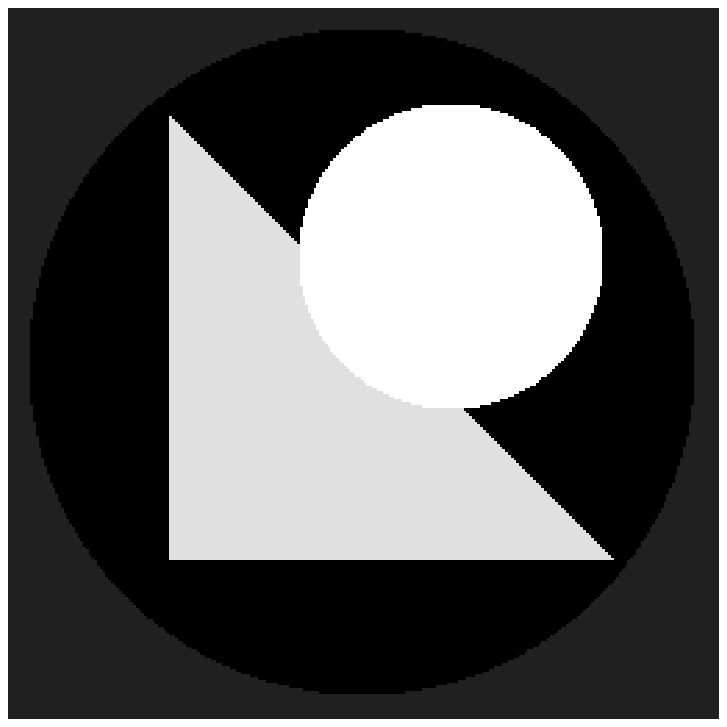} &
\includegraphics[width=\ww, height=\ww]{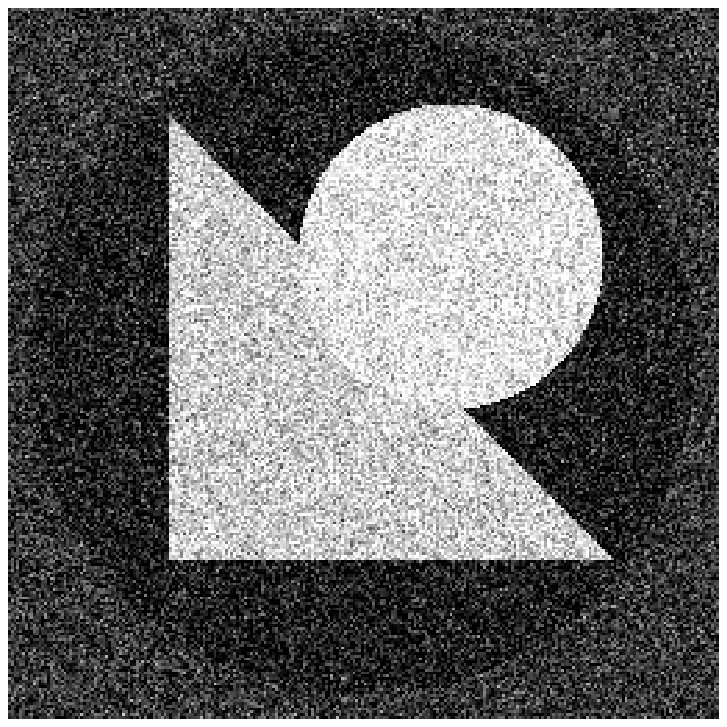} \\
(a) Clean image & (b) Noisy image
\end{tabular}
\begin{tabular}{cccc}
\includegraphics[width=\ww, height=\ww]{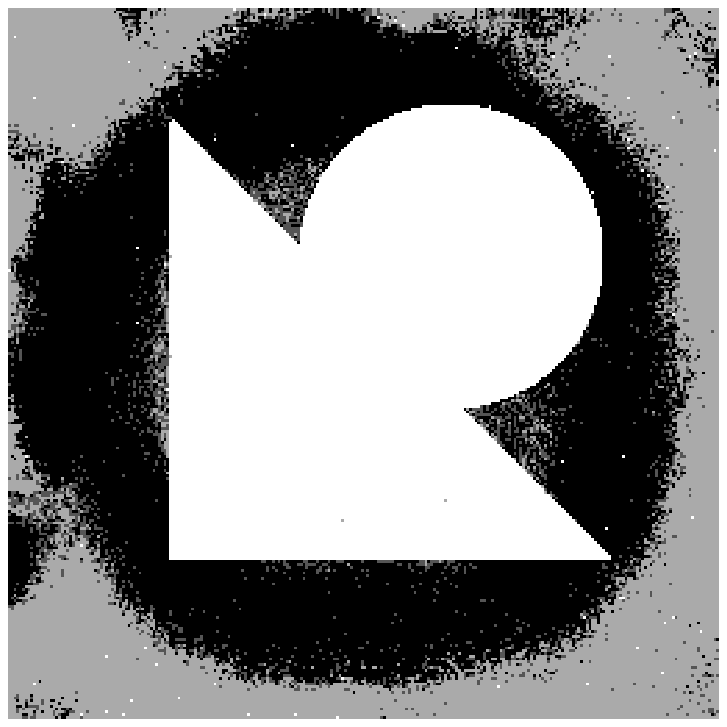} &
\includegraphics[width=\ww, height=\ww]{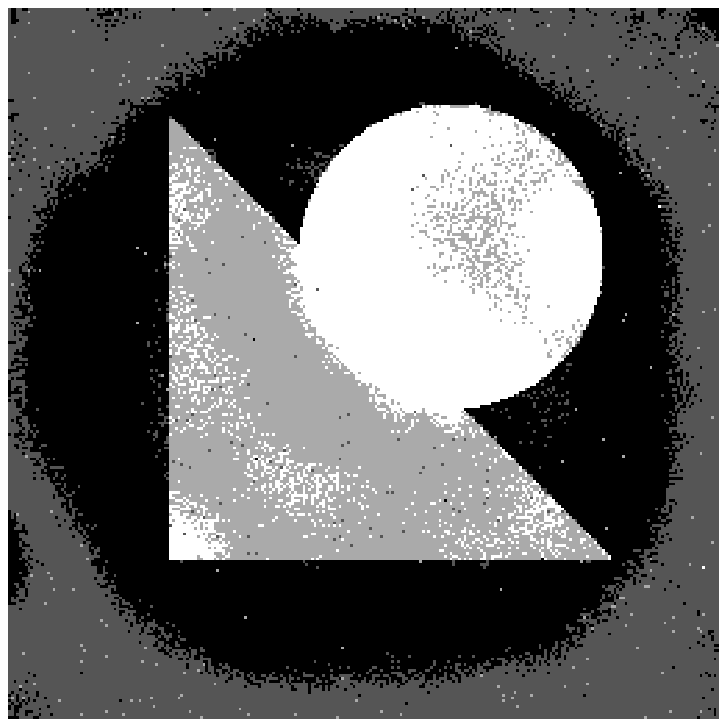} &
\includegraphics[width=\ww, height=\ww]{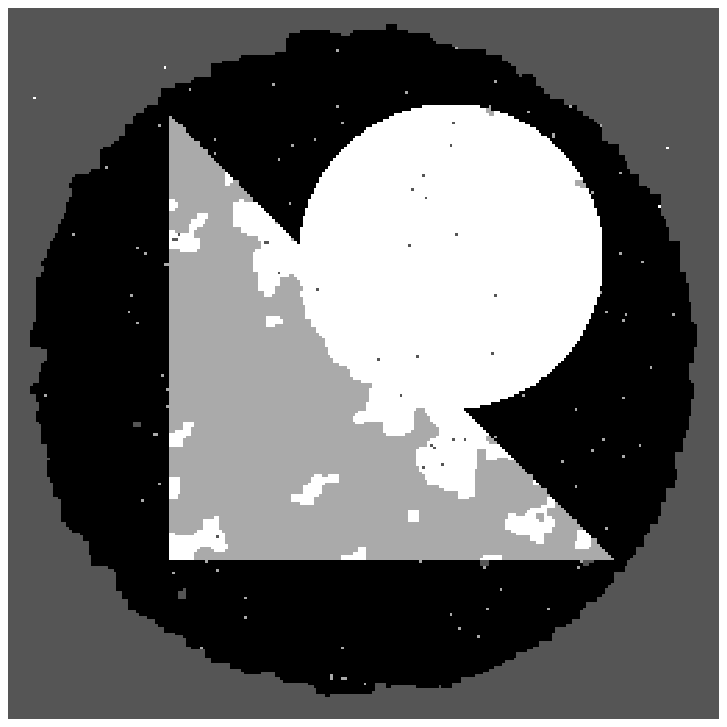} &
\includegraphics[width=\ww, height=\ww]{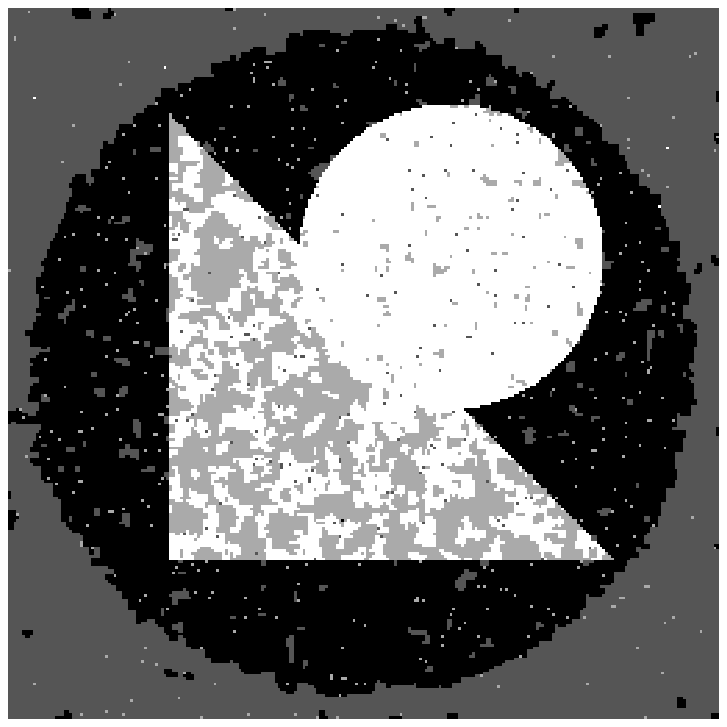} \\
(c) Li \cite{LNZS10} & (d) Li \cite{LNZS10} & (e) Pock \cite{PCCB09} & (f) Pock \cite{PCCB09} \\
\includegraphics[width=\ww, height=\ww]{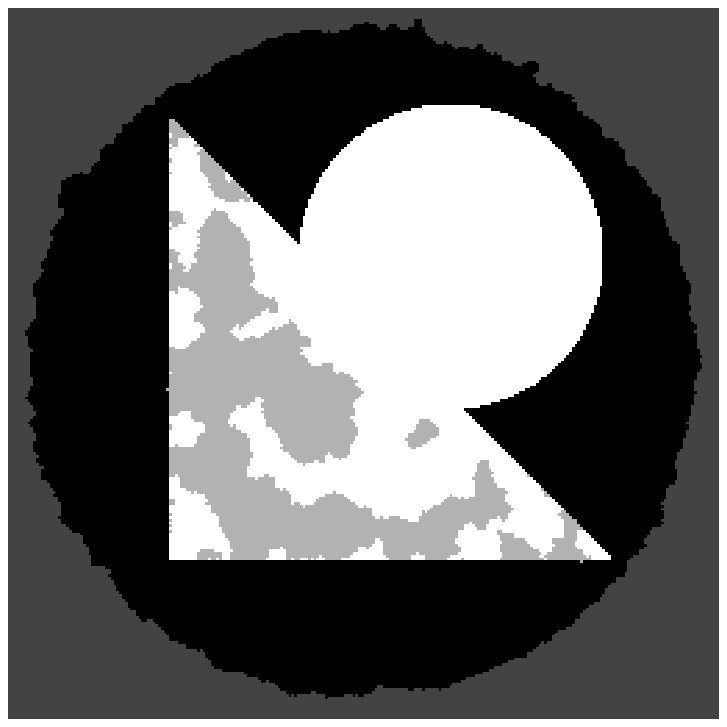} &
\includegraphics[width=\ww, height=\ww]{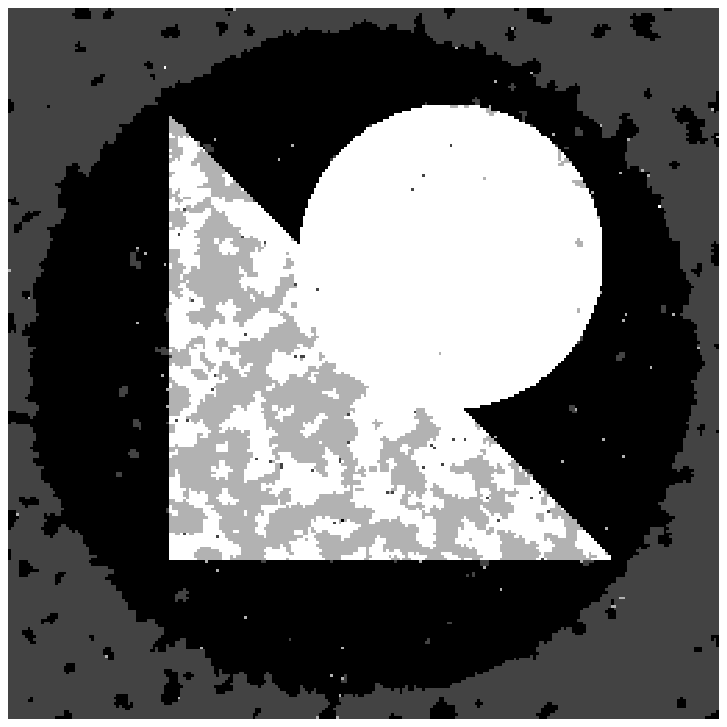} &
\includegraphics[width=\ww, height=\ww]{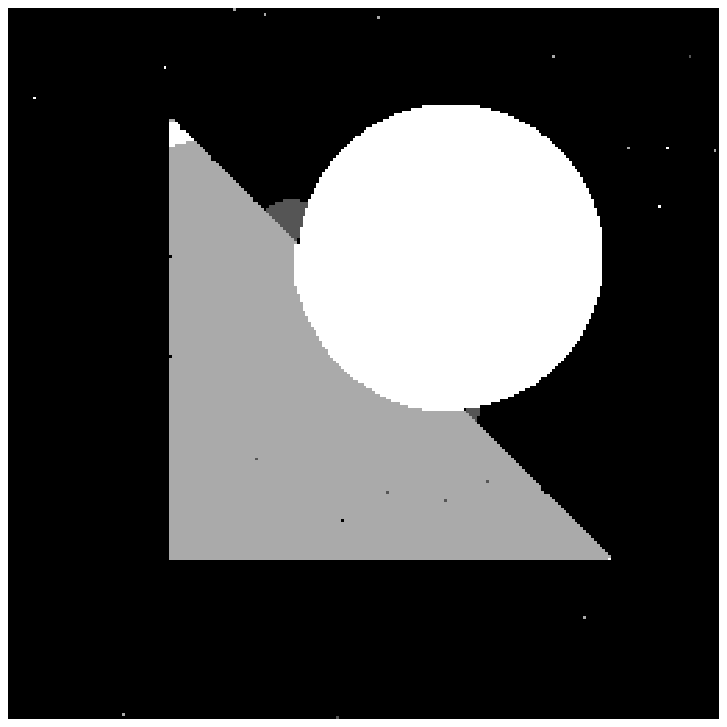} &
\includegraphics[width=\ww, height=\ww]{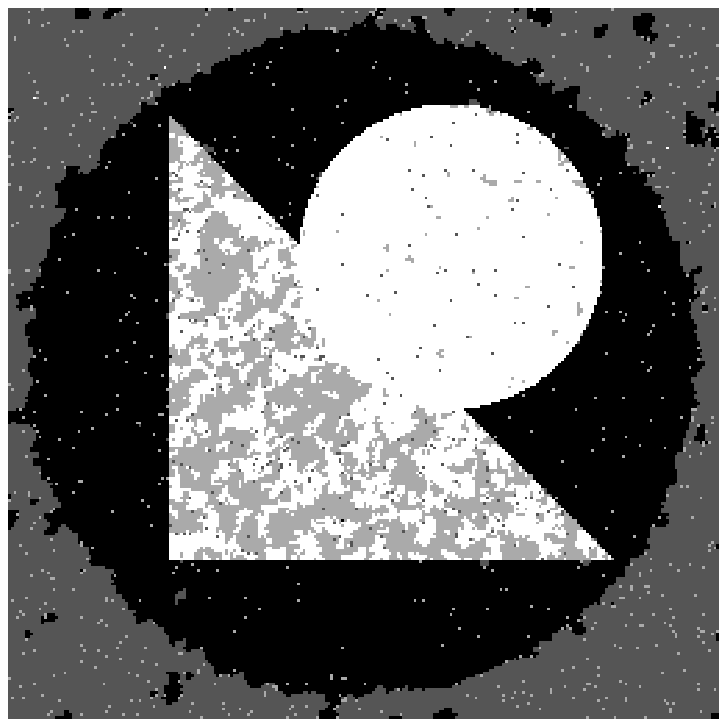} \\
(g) Yuan \cite{YBTB10} & (h) Yuan \cite{YBTB10} & (i) He \cite{HHMSS12} &
(j) He \cite{HHMSS12}\\
\includegraphics[width=\ww, height=\ww]{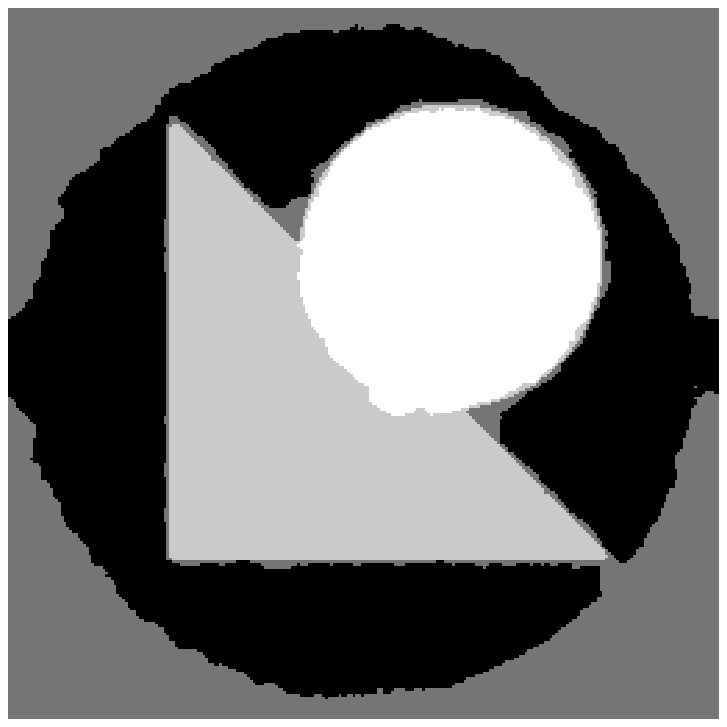} &
\includegraphics[width=\ww, height=\ww]{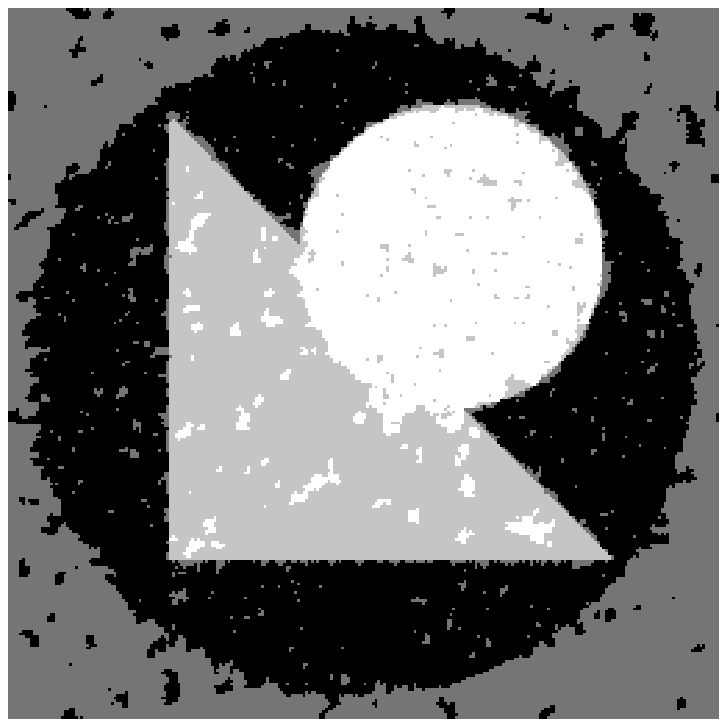} &
\includegraphics[width=\ww, height=\ww]{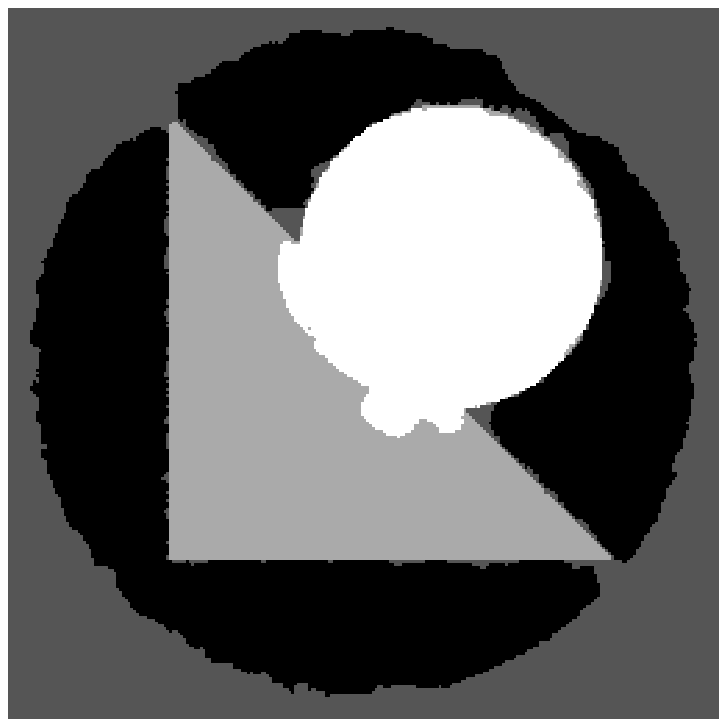} &
\includegraphics[width=\ww, height=\ww]{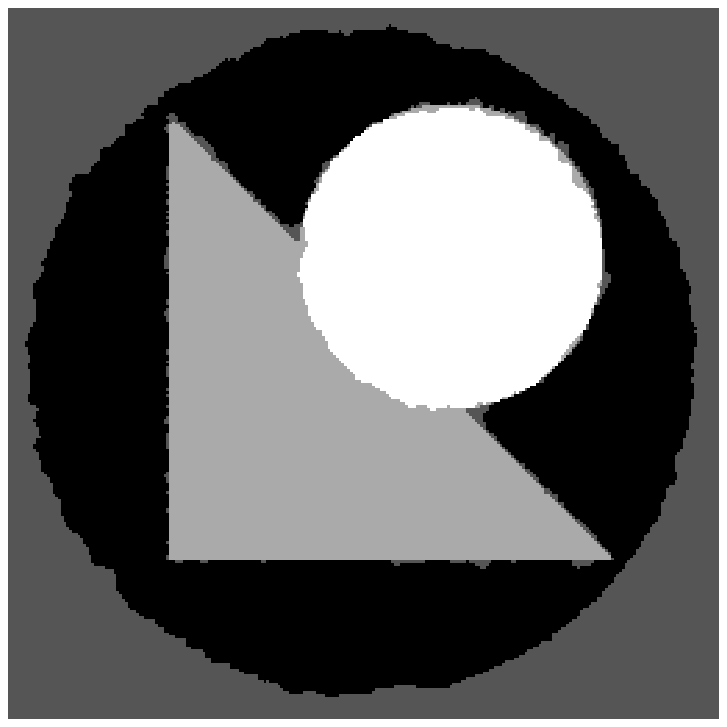} \\
(k) Cai \cite{CCZ13} & (l) Cai \cite{CCZ13} & (m) Ours (Ite. 1) & (n) Ours (Ite. 5)
\end{tabular}
\end{center}
\caption{Segmentation of four-phase image containing phases with close intensities
(size $256\times 256$). (a): clean image; (b): noisy image;
(c)--(d): results of \cite{LNZS10} with $\lambda = 50$ and 100, respectively;
(e)--(f): results of \cite{PCCB09} with $\lambda = 30$ and 50, respectively;
(g)--(h): results of \cite{YBTB10} with $\lambda = 10$ and 15, respectively;
(i)--(j): results of \cite{HHMSS12} with $\lambda = 20$ and 100, respectively;
(k)--(l): results of \cite{CCZ13} with $\mu = 3$ and 10, respectively;
(m)--(n): results of our T-ROF method, with $\mu = 4$, at iterations 1 and 5 (final result), respectively.
}\label{fourphase-close}
\end{figure}

\begin{table}[!h]
\centering \caption{Quantitative comparison: $\lambda/\mu$, iteration (Ite.) steps, CPU time in seconds, and SA in Examples (Exa.) 6--7.
The iteration steps of our T-ROF method e.g. 68 (6) mean that 68 and 6 iterations are respectively executed to find $u$ and $\zb \tau$ in 
Algorithm \ref{alg:t-rof}. }
\begin{tabular}{|c|c|r|r|r|r|r|r|}
\cline{3-8}
 \multicolumn{2}{r|}{} & Li \cite{LNZS10} & Pock \cite{PCCB09} & Yuan \cite{YBTB10} & He \cite{HHMSS12} &
 Cai \cite{CCZ13} & Our method \\ \hline
 \multirow{4}{*}{\rotatebox{90}{Exa. 6}} & $\lambda/\mu$ & 50 & 70 & 15 & 100 & 3 & - \\ \cline{2-8}
& Ite. & 100 & 200 & 300 & 100 & 103 & - \\ \cline{2-8}
& Time & 6.08 & 14.21 & 25.64 & 24.02 & 3.55 & - \\ \cline{2-8}
&SA & 0.4420 & 0.8840 & 0.9557 & 0.7939 & 0.9248 & - \\ \hline \hline
 \multirow{4}{*}{\rotatebox{90}{Exa. 6}} & $\lambda/\mu$ & 100 & 100 & 20 & 200 & 10 & 8 \\ \cline{2-8}
& Ite. & 100 & 200 & 300 & 100 & 83 & 68 (6) \\ \cline{2-8}
& Time & 4.96 & 14.62 & 21.84 & 22.63 & 3.06 & 2.07 \\ \cline{2-8}
&SA & 0.3746 & 0.9485 & 0.9359 & 0.9637 & 0.9232 & 0.9550 \\ \hline \hline
 \multirow{4}{*}{\rotatebox{90}{Exa. 7}} & $\lambda/\mu$ & 50 & 30 & 10 & 20 & 3 & - \\ \cline{2-8}
& Ite. & 100 & 150 & 194 & 100 & 128 & - \\ \cline{2-8}
& Time & 6.71 & 14.95 & 16.86 & 32.53 & 3.92 & - \\ \cline{2-8}
&SA & 0.4900 & 0.9549 & 0.9043 & 0.6847 & 0.9545 & - \\ \hline \hline
 \multirow{4}{*}{\rotatebox{90}{Exa. 7}} & $\lambda/\mu$ & 100 & 50 & 15 & 100 & 10 & 4 \\ \cline{2-8}
& Ite. & 100 & 150 & 246 & 100 & 65 & 111 (5) \\ \cline{2-8}
& Time & 6.64 & 14.91 & 21.32 & 32.79 & 2.28 & 3.13 \\ \cline{2-8}
&SA & 0.9023 & 0.8769 & 0.8744 & 0.8709 & 0.9273 & 0.9798 \\ \hline
\end{tabular}
\label{time-examples-complex}
\end{table}

\begin{figure}[!htb]
\begin{center}
\begin{tabular}{ccc}
\includegraphics[trim={{.03\linewidth} {.05\linewidth} {.12\linewidth} {.01\linewidth}}, clip, width=0.23\linewidth, height = 0.12\linewidth, width=40mm, height=32mm]{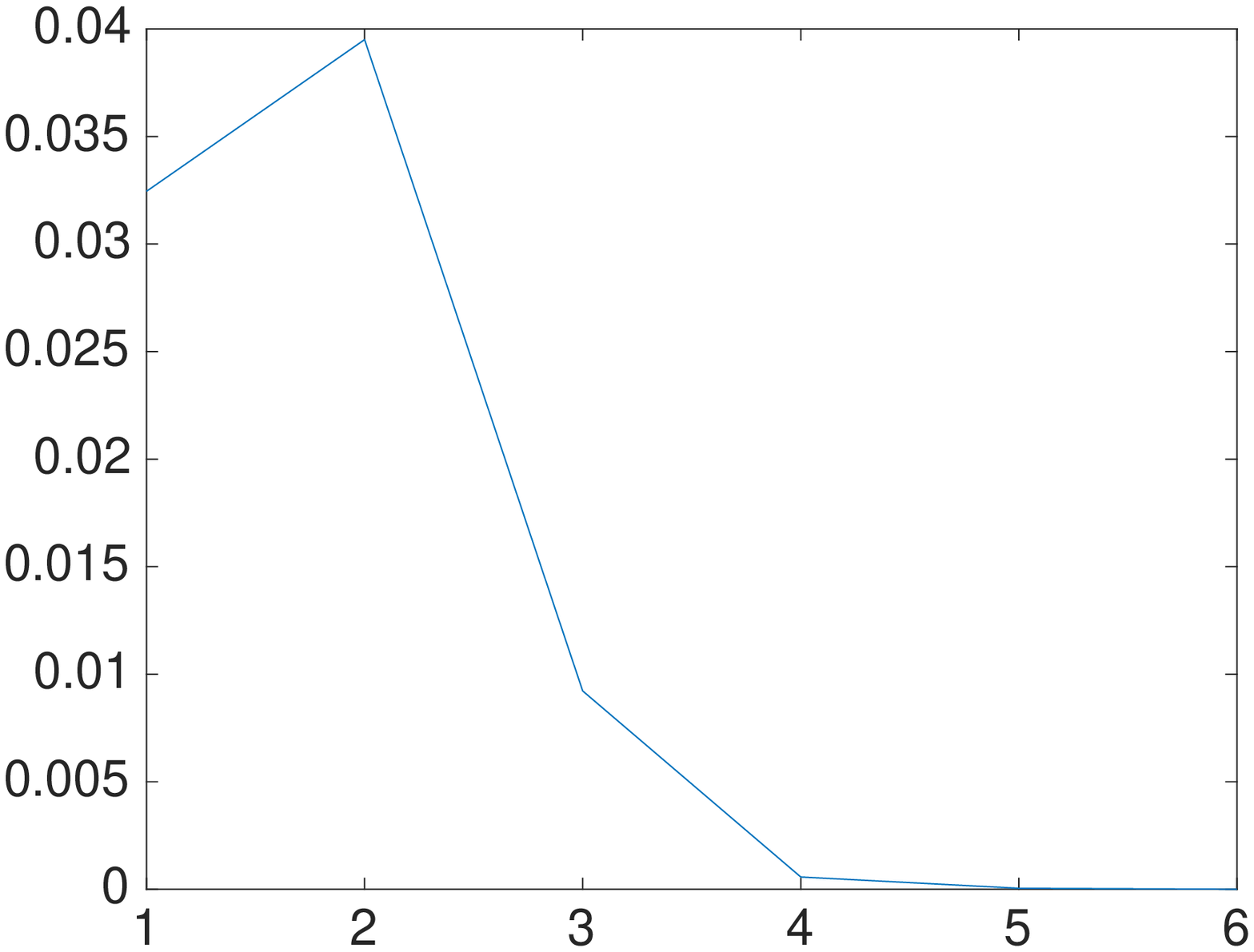} \put(-50,75){\small Example 1}&
\includegraphics[trim={{.03\linewidth} {.05\linewidth} {.12\linewidth} {.01\linewidth}}, clip, width=0.23\linewidth, height = 0.12\linewidth, width=40mm, height=32mm]{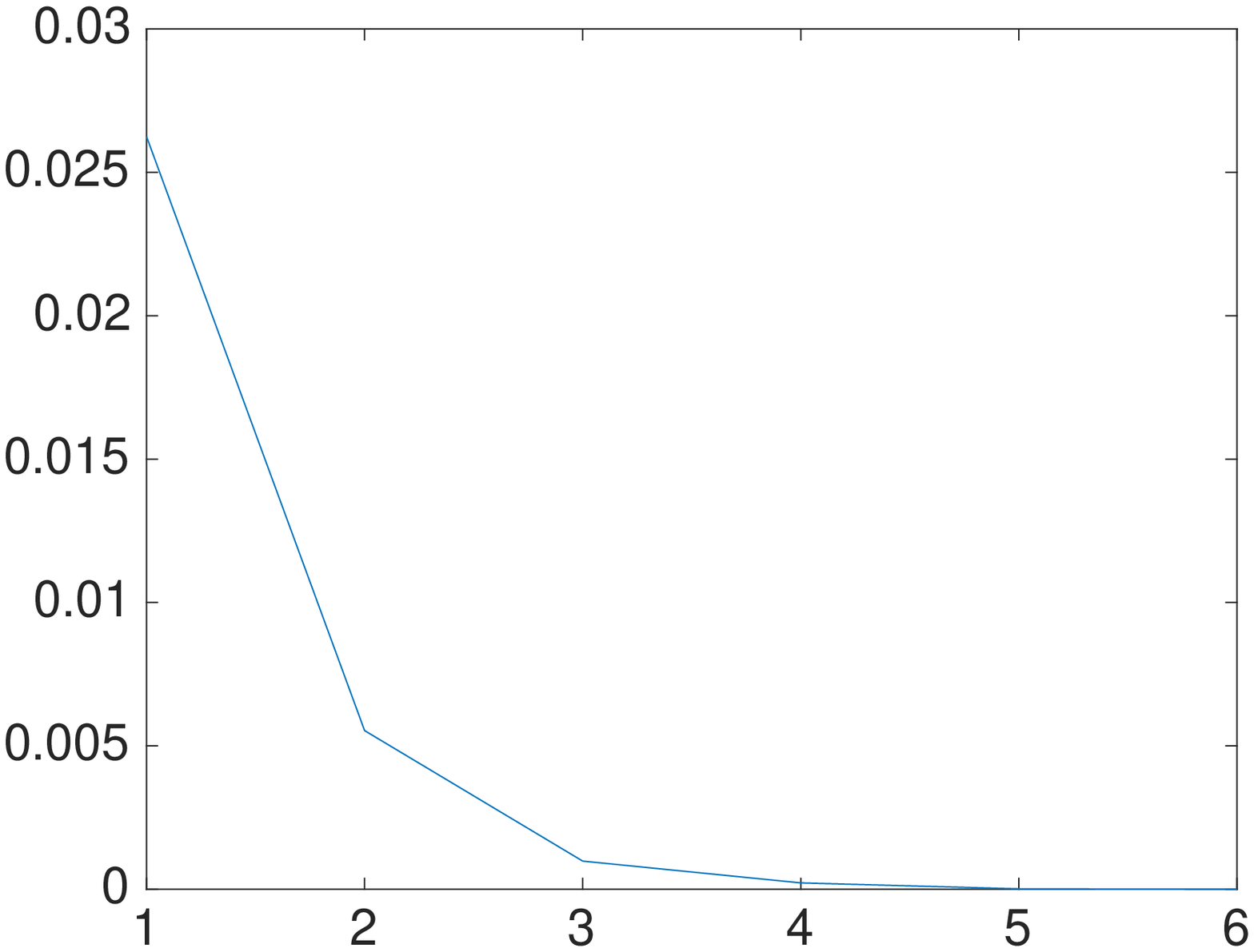} \put(-50,75){\small Example 2} &
\includegraphics[trim={{.03\linewidth} {.05\linewidth} {.12\linewidth} {.01\linewidth}}, clip, width=0.23\linewidth, height = 0.12\linewidth, width=40mm, height=32mm]{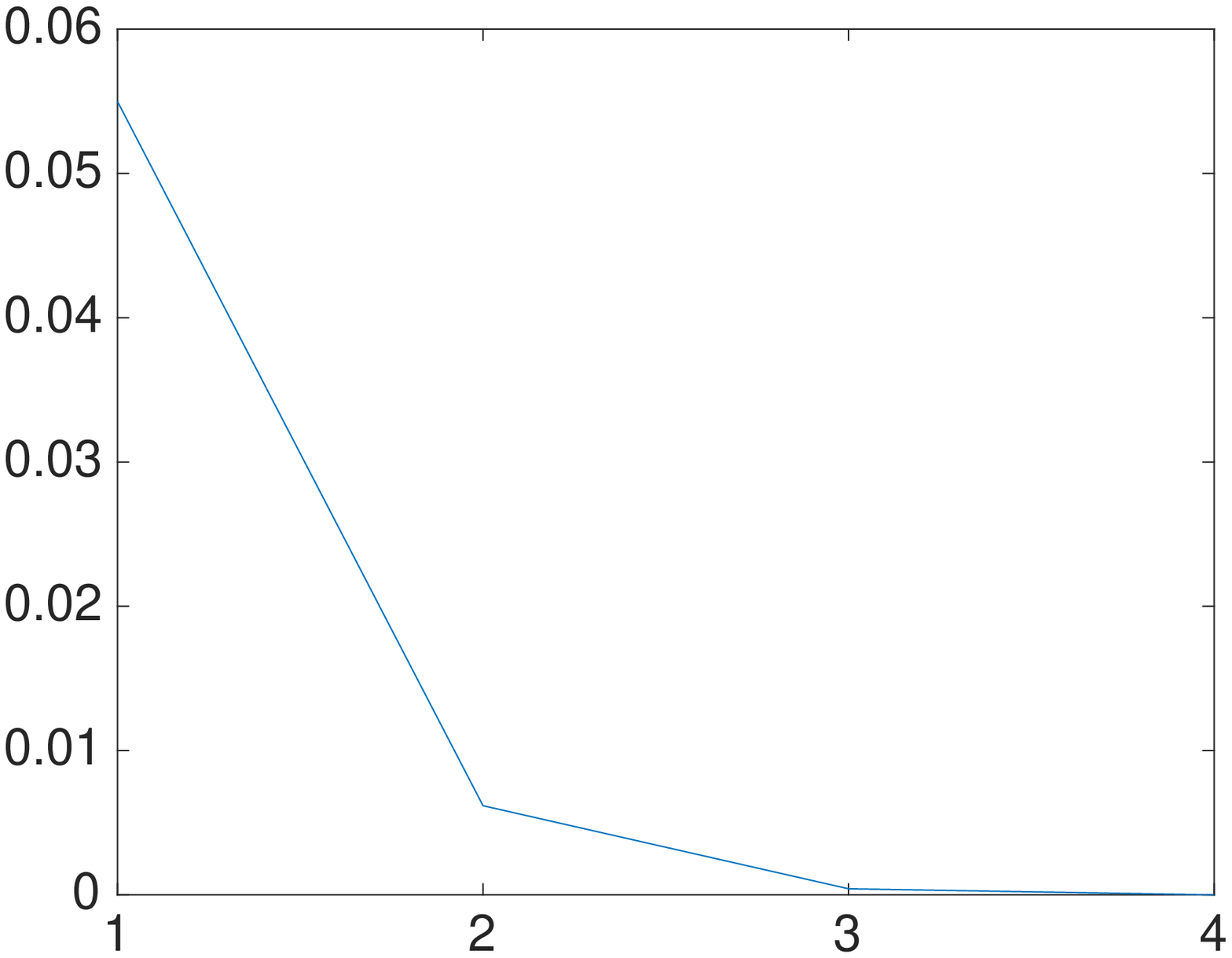} \put(-50,75){\small Example 3} \\
\includegraphics[trim={{.03\linewidth} {.05\linewidth} {.12\linewidth} {.01\linewidth}}, clip, width=0.23\linewidth, height = 0.12\linewidth, width=40mm, height=32mm]{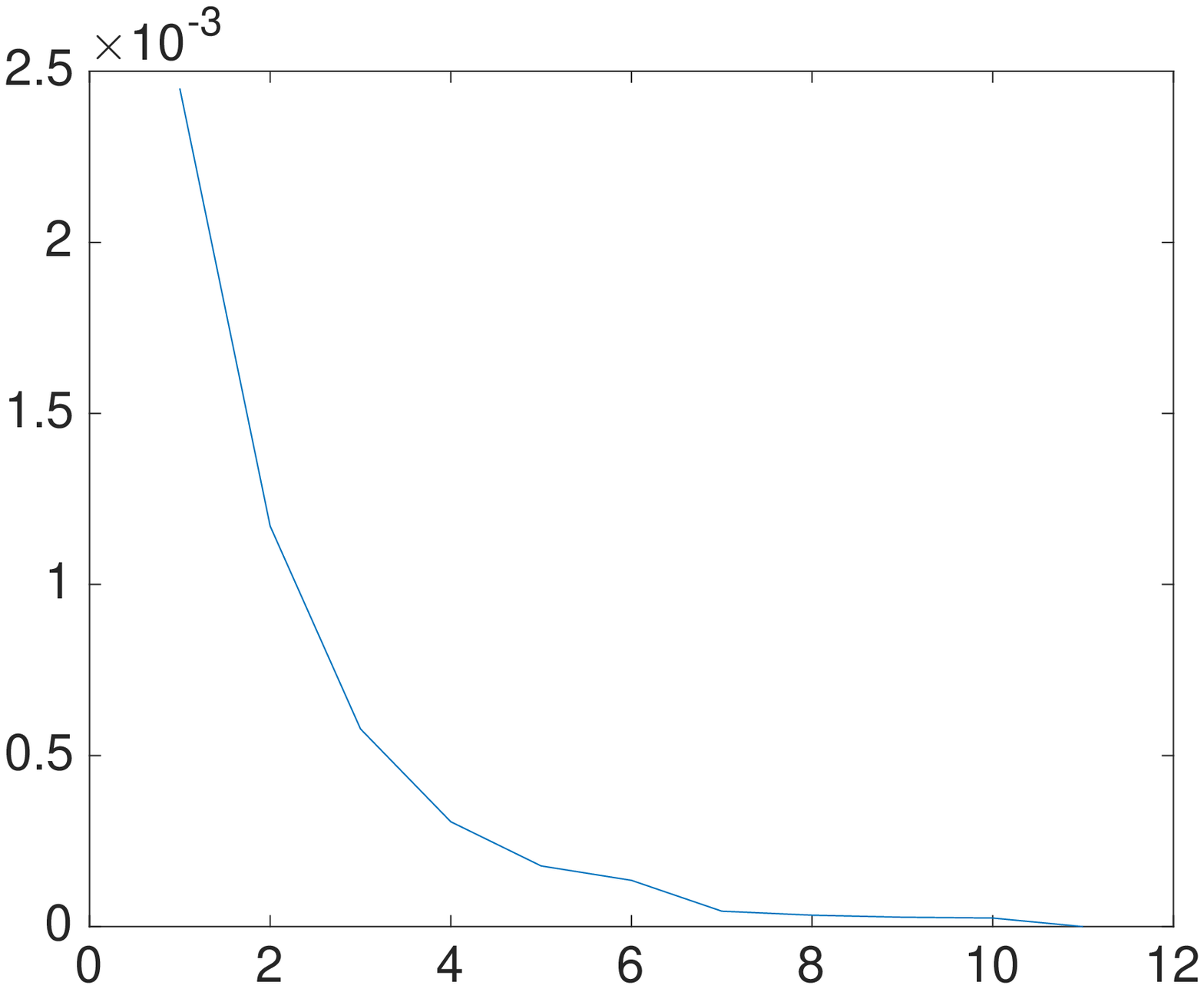} \put(-50,75){\small Example 4} \put(-122,43) {\rotatebox[origin=c]{90}{\small $\|{\zb \tau}^{(k)} - {\zb \tau}^{(k-1)} \|_2 $}}&
\includegraphics[trim={{.03\linewidth} {.05\linewidth} {.12\linewidth} {.01\linewidth}}, clip, width=0.23\linewidth, height = 0.12\linewidth, width=40mm, height=32mm]{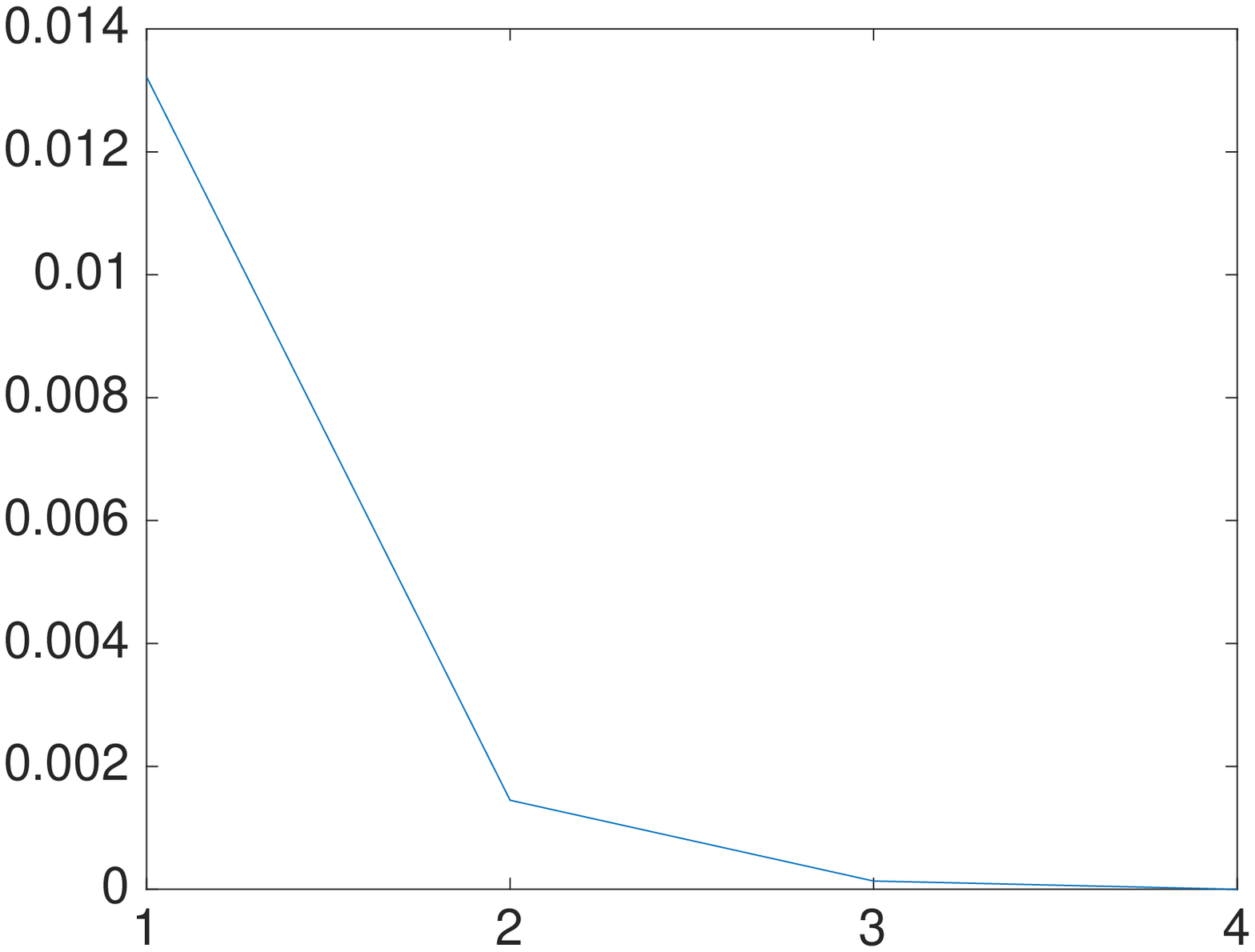} \put(-50,75){\small Example 5} \put(-50,65){\small (5-phase)} &
\includegraphics[trim={{.03\linewidth} {.05\linewidth} {.12\linewidth} {.01\linewidth}}, clip, width=0.23\linewidth, height = 0.12\linewidth, width=40mm, height=32mm]{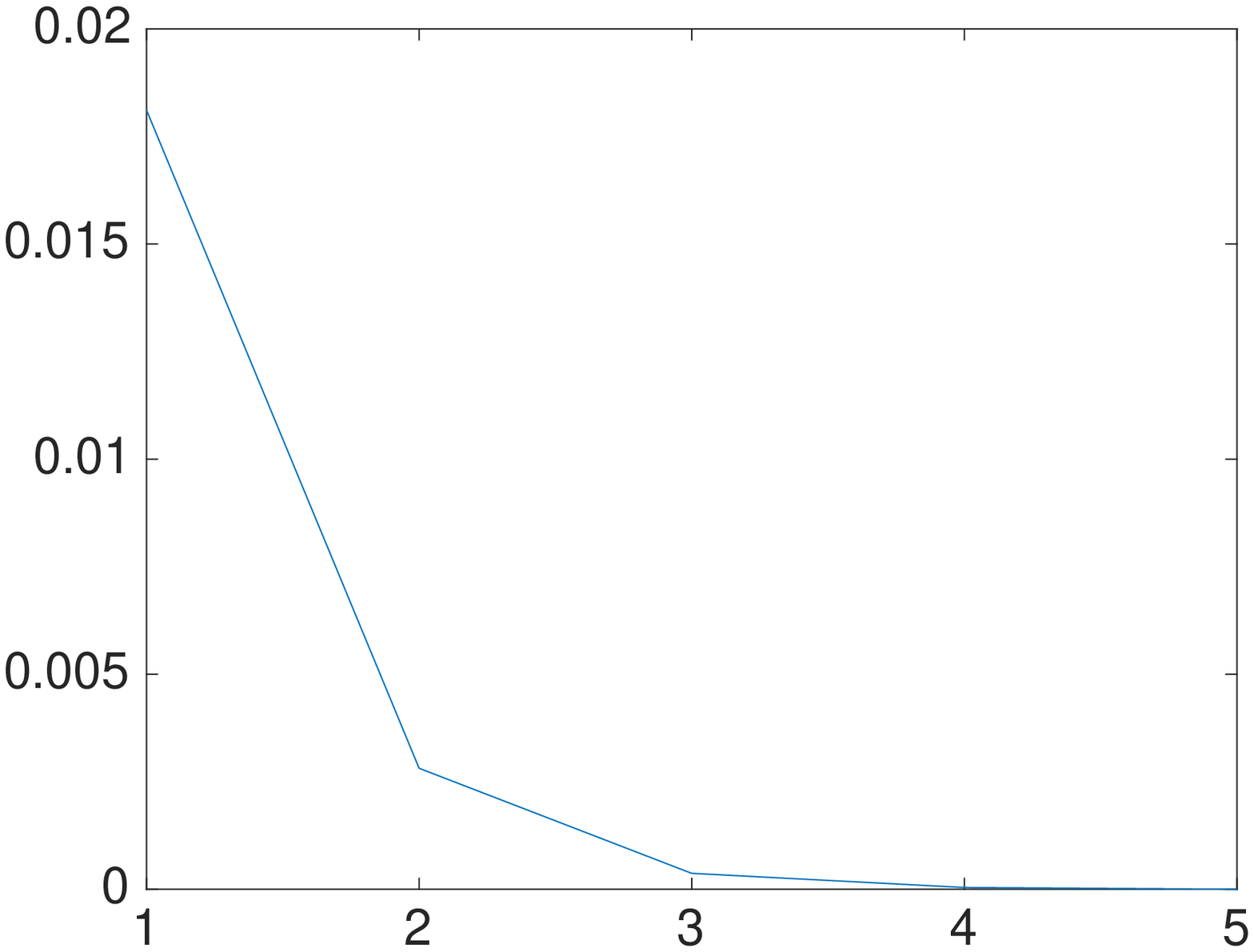}  \put(-50,75){\small Example 5} \put(-50,65){\small (10-phase)}  \\
\includegraphics[trim={{.03\linewidth} {.05\linewidth} {.12\linewidth} {.01\linewidth}}, clip, width=0.23\linewidth, height = 0.12\linewidth, width=40mm, height=32mm]{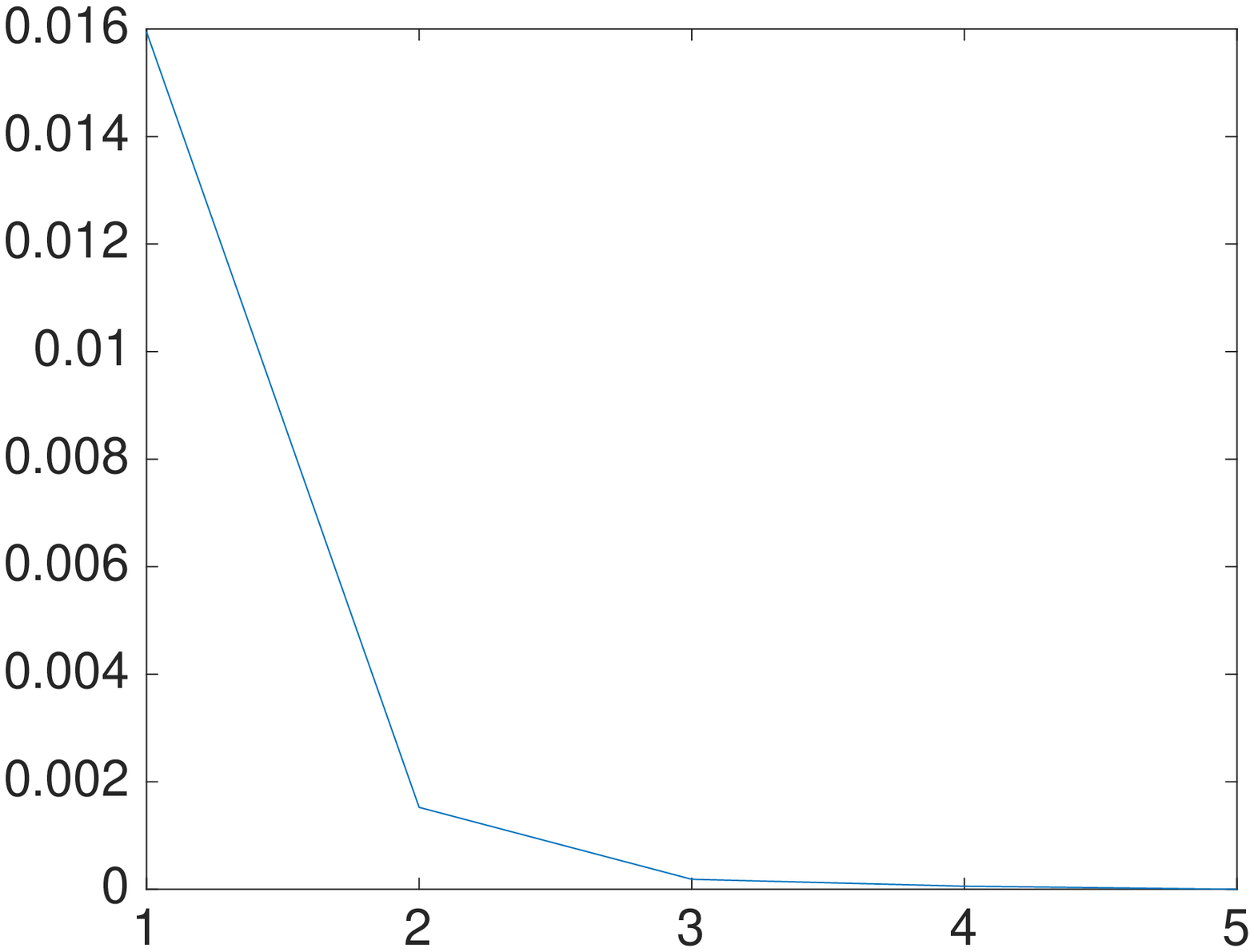}  \put(-50,75){\small Example 5} \put(-50,65){\small (15-phase)} &
\includegraphics[trim={{.03\linewidth} {.05\linewidth} {.12\linewidth} {.01\linewidth}}, clip, width=0.23\linewidth, height = 0.12\linewidth, width=40mm, height=32mm]{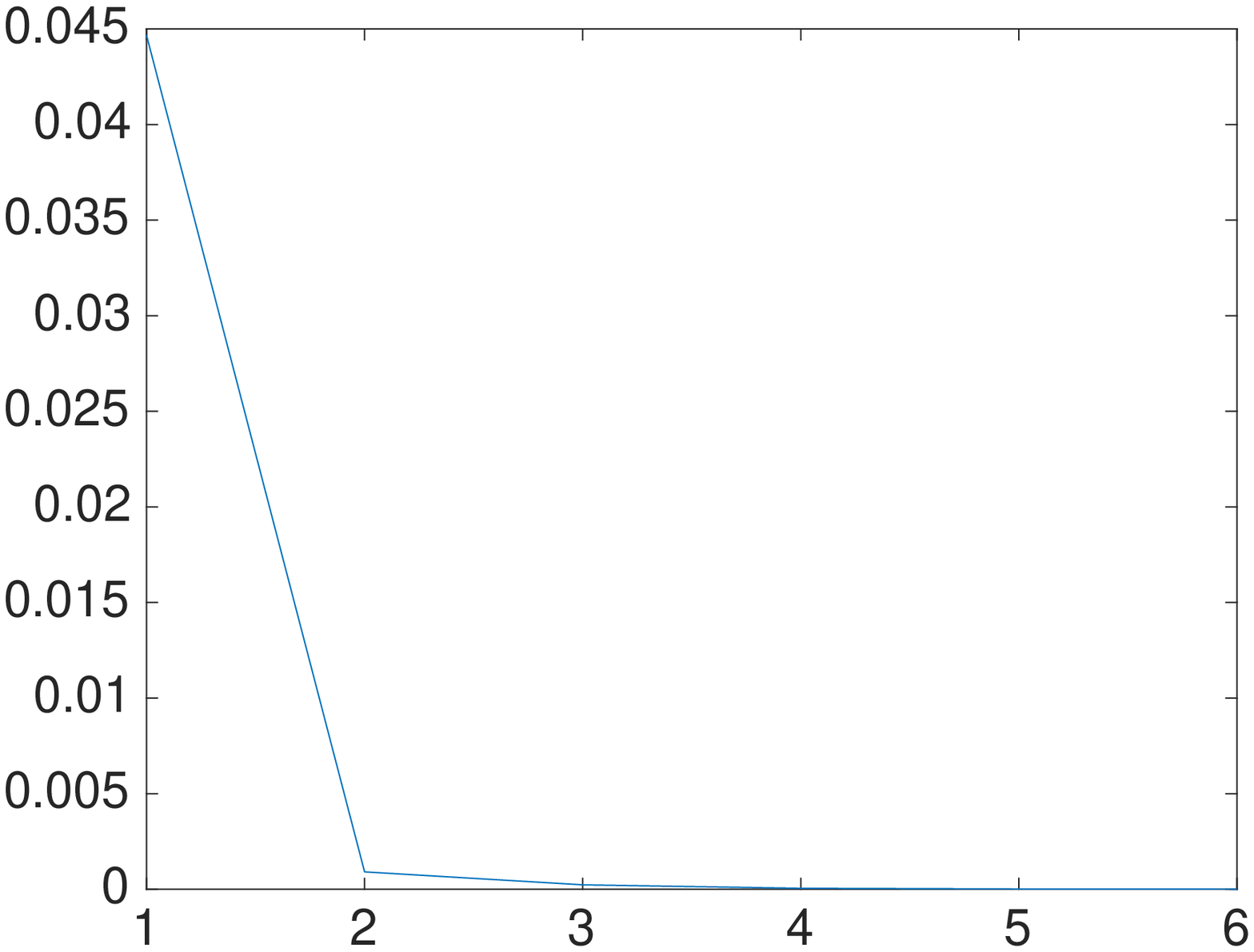} \put(-50,75){\small Example 6}  \put(-80,-8){\small Iteration steps} &
\includegraphics[trim={{.03\linewidth} {.05\linewidth} {.12\linewidth} {.01\linewidth}}, clip, width=0.23\linewidth, height = 0.12\linewidth, width=40mm, height=32mm]{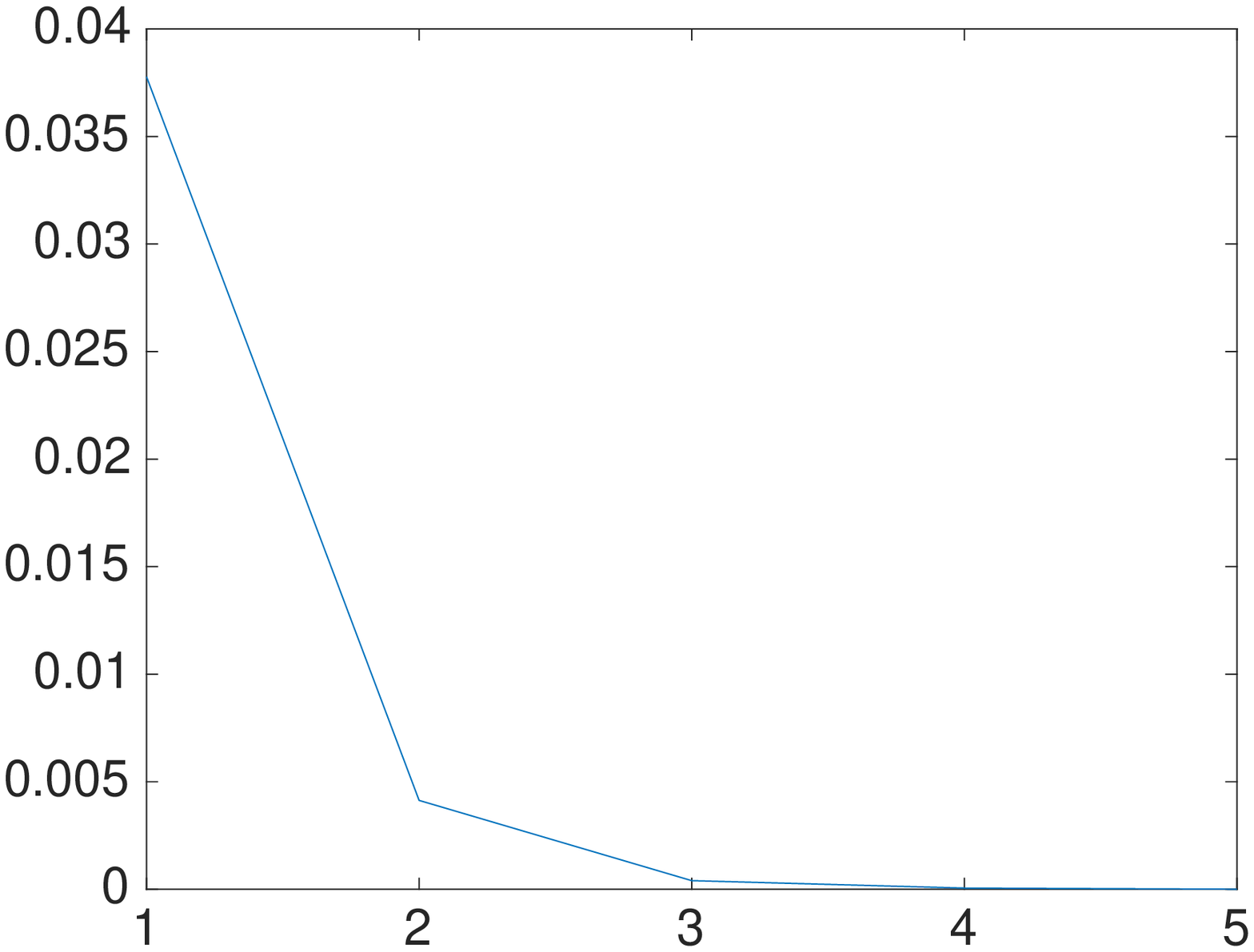}  \put(-50,75){\small Example 7} 
\end{tabular}
\end{center}
\caption{Convergence history of {\zb \tau} of our T-ROF algorithm corresponding to iteration steps in Examples 1--7. 
The vertical axis represents the absolute error of ${\zb \tau}$ in two consecutive iterations, and the horizontal axis represents the iteration steps.
}\label{error-tau}
\end{figure}

%-------------------------------------------------------------------
\subsection{Discussion About the Threshold ${\zb \tau}$ }
%-------------------------------------------------------------------
We report, in Fig. \ref{error-tau}, the convergence history of ${\zb \tau}$ of our T-ROF algorithm \ref{alg:t-rof} corresponding to iteration steps
for Examples 1--7. From Fig. \ref{error-tau}, we can see that ${\zb \tau}$ always
not only converges, but converges very quickly -- just a few steps are enough (generally within ten iterations for the examples
shown in this paper).

It is worth mentioning that, for all the Examples 1--7, the rule \eqref{eqn:seg-meaningful-omega} and the criterion derived in \eqref{eqn:add-rule}
are always satisfied at each iteration. This means that these rules are indeed very easy to fulfil, if one uses meaningful number of phases $K$
(here $K$ is our prior knowledge for each segmentation example) and 
initializations (here we use K-means to obtain initializations). The chances that these criterions are likely to be invoked based on our tests are the cases, 
for example, meaningless number of phases $K$ and/or initializations. For example, given a test image which has two phases, using $K = 3$ would be highly likely to 
obtain an empty phase at some iteration and one would finally obtain a segmentation result with two phases. An empty phase that is also likely to be obtained
is when two thresholds are too close in the initialization. This performance is actually a kind of
automatic error correction, which is an innate ability of our T-ROF algorithm. In contrast, the methods solving the PCMS model 
(like the ones we compared in this paper) cannot have.

%-------------------------------------------------------------------
\subsection{Real-World Example}
%-------------------------------------------------------------------
In the previous examples, we mainly investigated the performance of the aforementioned segmentation methods
using synthetic images and quantified their accuracy using SA (segmentation accuracy) defined in \eqref{eqn:sa}. 
Before closing this section, to complement the test, we hereby test those methods using a manual segmentation of a real-world image, 
the retina image, which is from the DRIVE data-set\footnote{http://www.isi.uu.nl/Research/Databases/DRIVE/}. 
Fig. \ref{example:retina} (a) and (b) are the clean manual segmentation image and the noisy image
generated by adding Gaussian noise with mean 0 and variance 0.1. Note that in Fig. \ref{example:retina} (a), 
we changed the original binary manual segmentation image to three phases by lowering the intensity of those vessels on the right hand side
from 1 to 0.3; the intensities of the background and the vessels on the left hand side are respectively 0 and 1. 
Obviously, segmenting the noisy three-phase image in Fig. \ref{example:retina} (b) is extremely challenging 
due to those thin blood vessels which have a big chance of being smoothed out. 

To quantify the segmentation accuracy, in addition to the measure SA in \eqref{eqn:sa} which
quantifies the whole segmentation accuracy, we here also use the DICE score which is able
to quantify the accuracy for individually segmented phases, i.e., 
\begin{equation*}
\textrm{DICE}(\Omega_i, \Omega_i^\prime) := \frac{2 |\Omega_i \cap \Omega_i^\prime|}{|\Omega_i| + |\Omega_i^\prime|},
\end{equation*}
where $\Omega_i$ and $\Omega_i^\prime$ are respectively the segmentation result and the ground truth 
corresponding to phase $i$, and $|\Omega_i|$ is the cardinality of set $\Omega_i$ in discrete setting.
Note that the DICE score not only involves the correctly segmented areas like the measure SA in \eqref{eqn:sa}, but the incorrectly areas. 

Fig. \ref{example:retina} shows the segmentation results of the methods compared, and Table \ref{exp:real-image} gives the quantitative results
including the SA and the DICE score which quantifies the three individually segmented phases, i.e.,
the areas of the background, vessels on the right hand side with lower intensity and vessels on the left hand side.
We clearly see that all the methods successfully segmented the vessels on the left hand side, but
 failed to segment the vessels on the right hand side (please refer to
Section \ref{subsec:related} for the theoretical discussion about the drawback of the PCMS model for case $K > 2$), except for method \cite{CCZ13} and 
our proposed T-ROF method (also with much faster speed). Particularly, to segment the vessels on the right hand side which have lower intensity, 
the proposed T-ROF method achieved higher DICE score compared to the SaT method \cite{CCZ13} -- 0.7749 versus 0.5673 -- 
which further verifies the excellent performance of updating the threshold $\zb \tau$
using the rule proposed in \eqref{trof-thd}.

\begin{figure}[!htb]
\begin{center}
\begin{tabular}{cc}
\includegraphics[width=\www, height=\www]{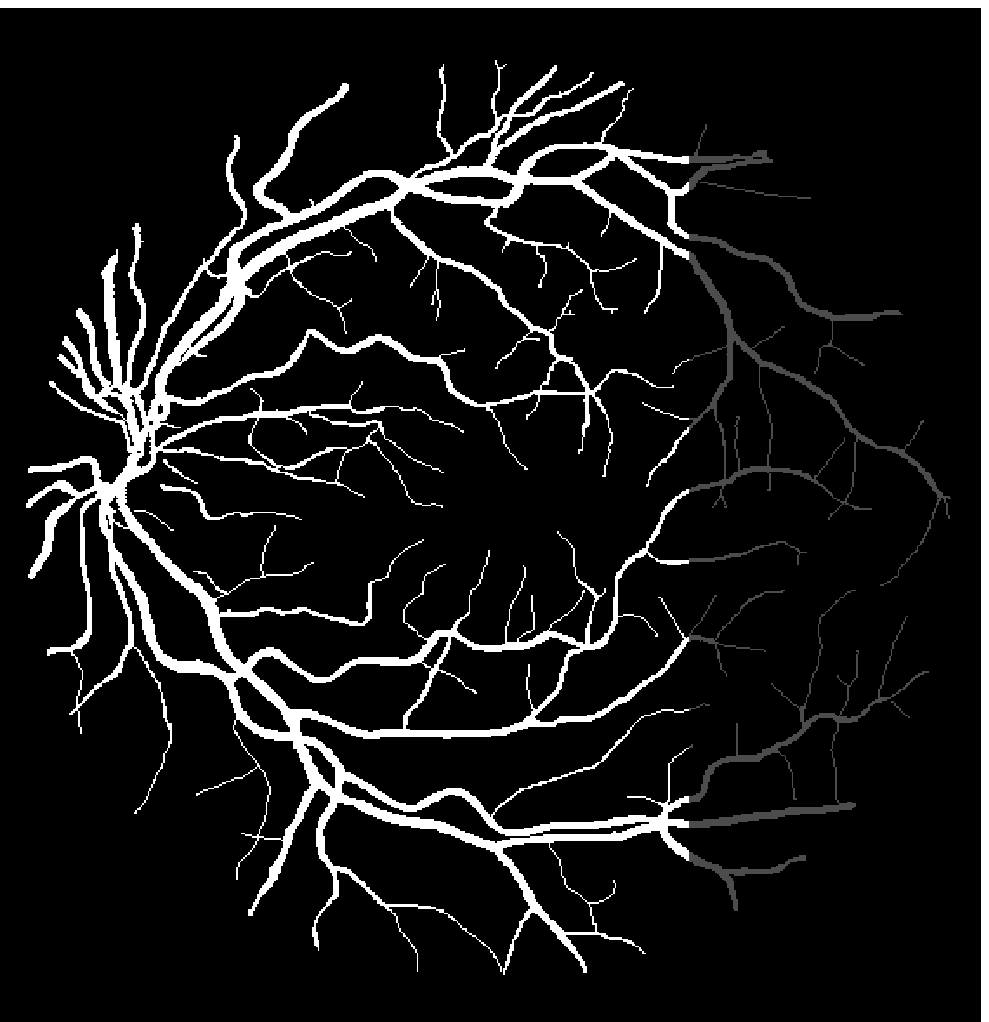} &
\includegraphics[width=\www, height=\www]{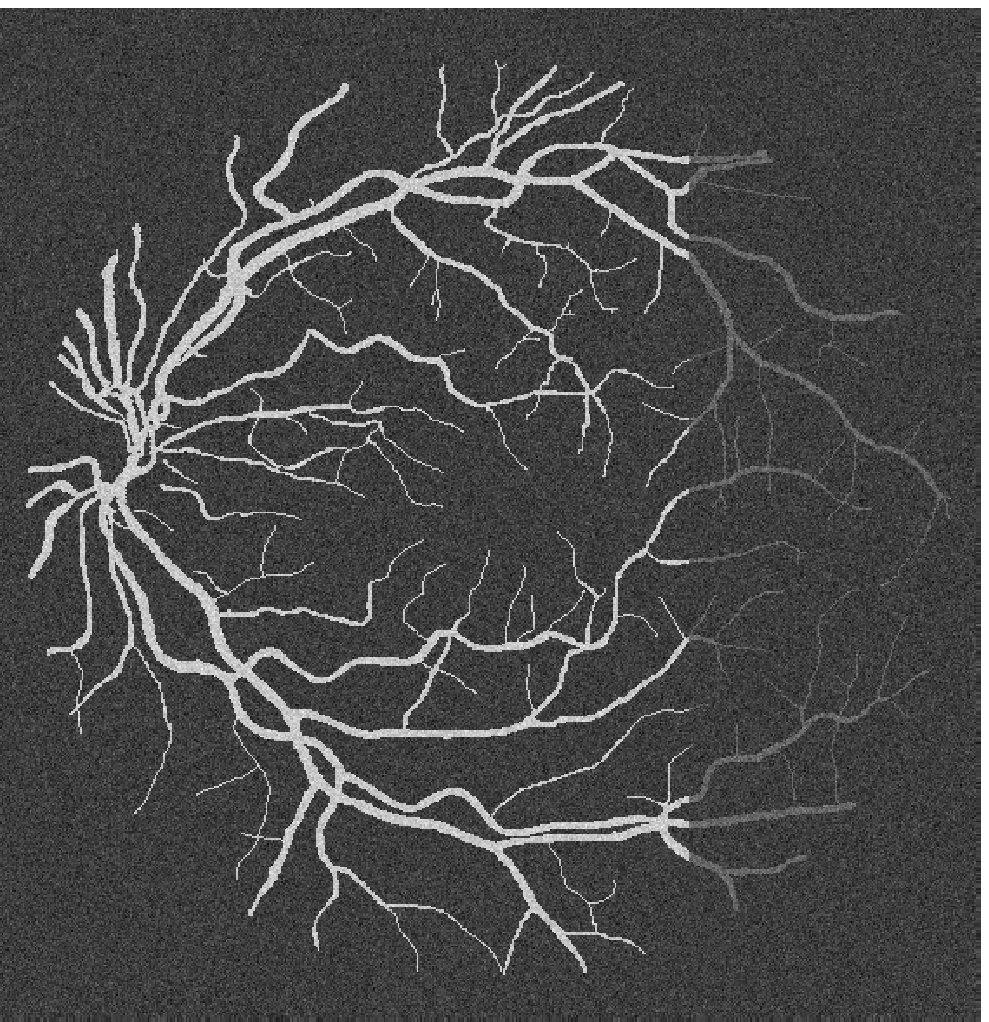}\\
(a) Clean image & (b) Noisy image 
\end{tabular}
\begin{tabular}{ccc}
\includegraphics[width=\www, height=\www]{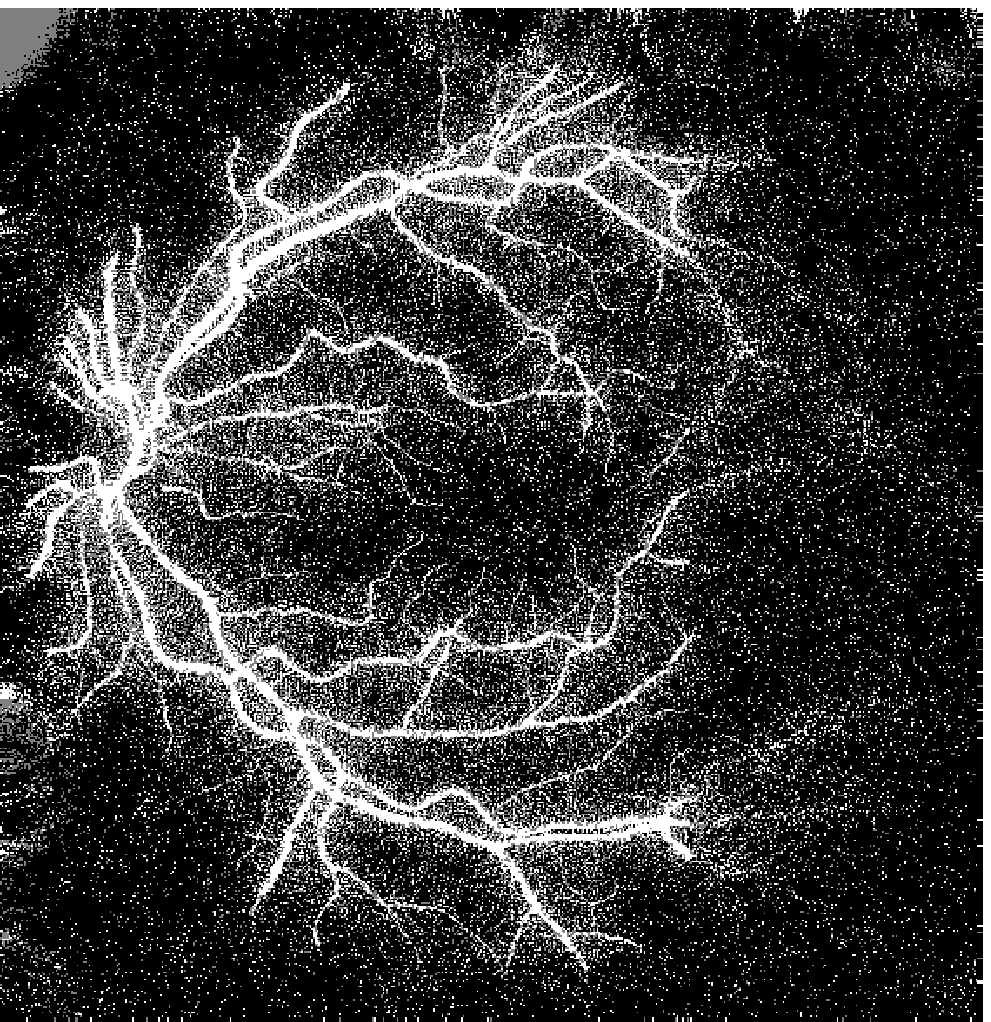} &
\includegraphics[width=\www, height=\www]{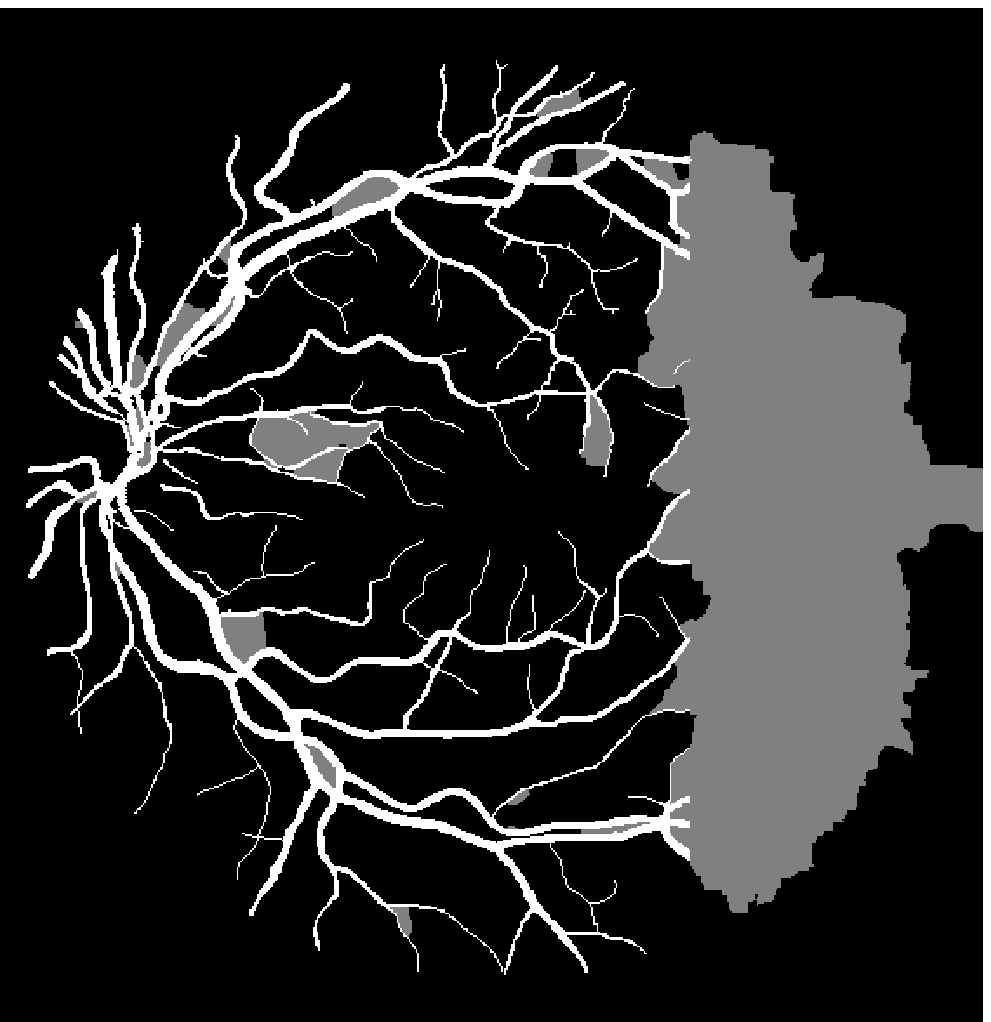} &
\includegraphics[width=\www, height=\www]{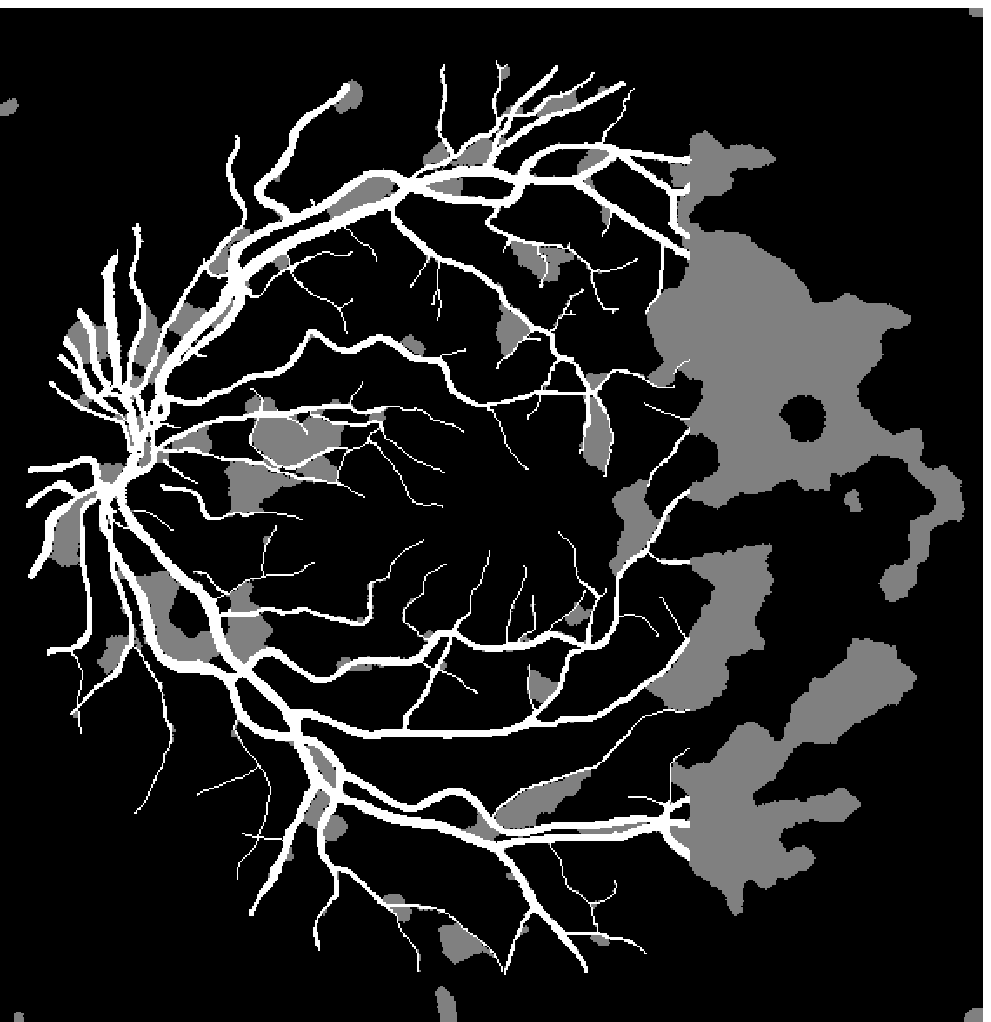}\\
(c) Li \cite{LNZS10} & (d) Pock \cite{PCCB09}  &  (e) Yuan \cite{YBTB10}  \\
\includegraphics[width=\www, height=\www]{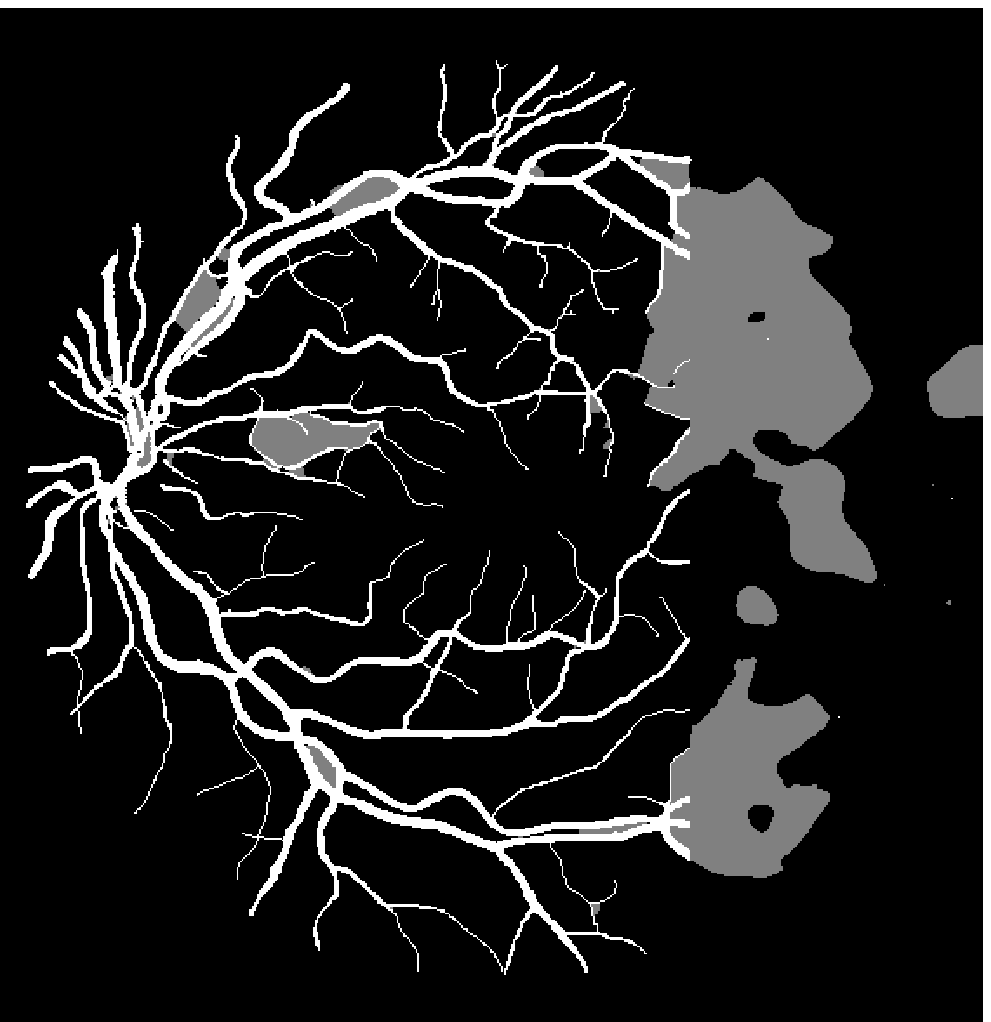} &
\includegraphics[width=\www, height=\www]{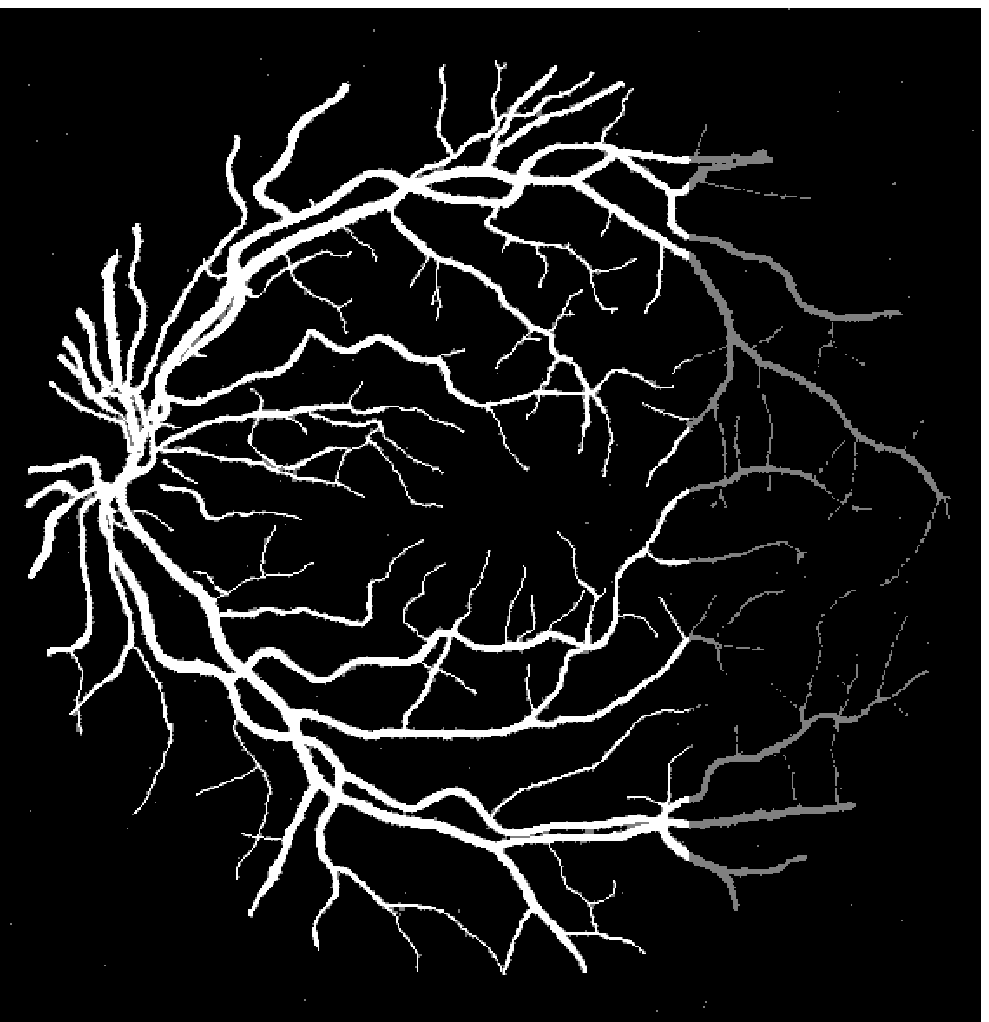} &
\includegraphics[width=\www, height=\www]{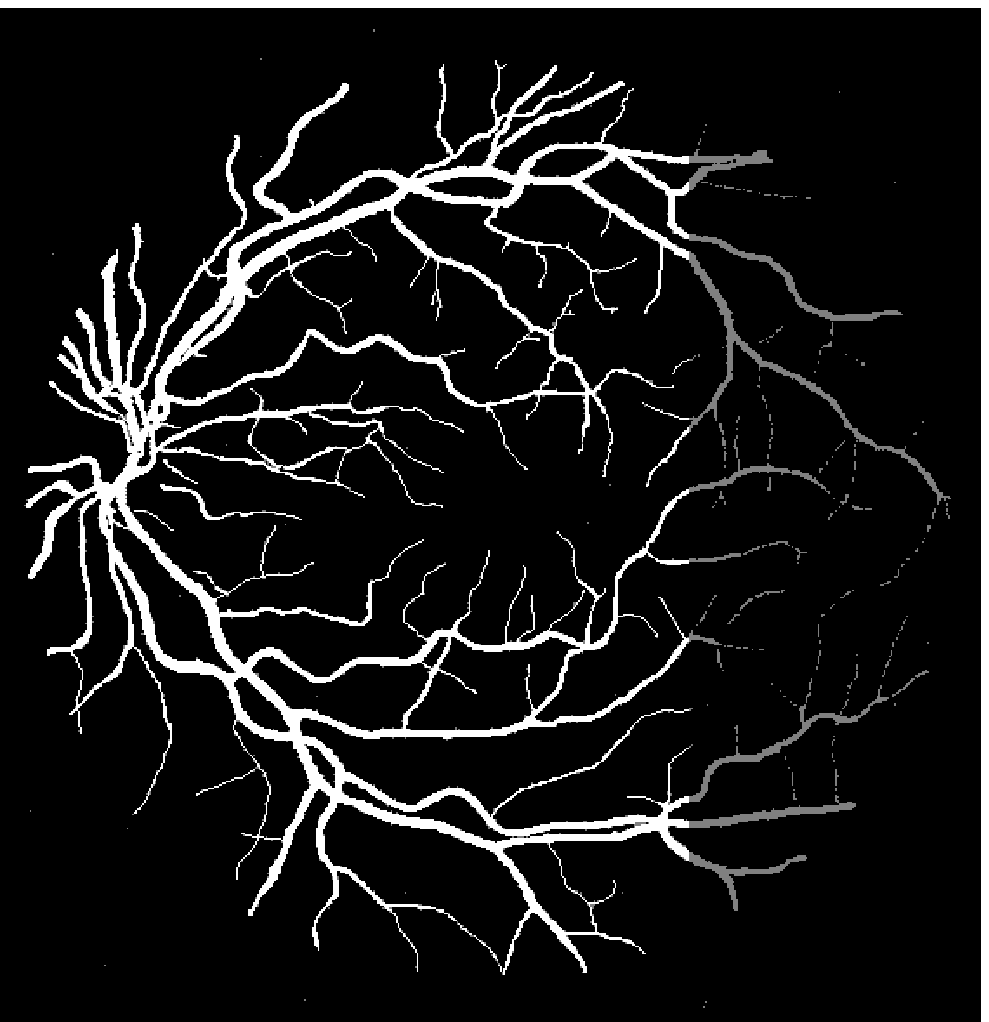}\\
(f) He \cite{HHMSS12} & (g) Cai \cite{CCZ13} & (h) Ours (Ite. 15)
\end{tabular}
\end{center}
\caption{Retina image segmentation which contains extremely thin vessels
(size $584\times 565$). (a): clean image; (b): noisy image; (c)--(h): results of methods \cite{LNZS10,PCCB09,YBTB10,HHMSS12,CCZ13} 
and our T-ROF method, respectively.
}\label{example:retina}
\end{figure}

\begin{table}[!h]
\centering \caption{Quantitative comparison for the retina test example in Fig. \ref{example:retina} : $\lambda/\mu$, iteration (Ite.) steps, CPU time in seconds, 
SA, and DICE score for individually segmented phases (phases $\Omega_0, \Omega_1$ and $\Omega_2$ are the areas of the background, vessels on the right hand side with
lower intensity and vessels on the left hand side, respectively) .
The iteration steps of our T-ROF method 35 (15) mean that 35 and 15 iterations are respectively executed to find $u$ and $\zb \tau$ in 
Algorithm \ref{alg:t-rof}.}
\begin{tabular}{|c|r|r|r|r|r|r|}
\cline{2-7}
 \multicolumn{1}{r|}{}   & Li \cite{LNZS10} & Pock \cite{PCCB09} & Yuan \cite{YBTB10} & He \cite{HHMSS12} &
 Cai \cite{CCZ13} & Our method \\ \hline
 $\lambda/\mu$ & 100 & 100 & 20 & 100 & 25 & 25 \\ \cline{1-7}
 Ite. & 100 & 150 & 300 & 100 & 75 & 35 (15) \\ \hline
 Time & 2.67 & 18.67 & 16.79 & 22.84 & 3.51 & 2.09 \\ \hline \hline
 SA & 0.7790  & 0.8462   & 0.8823   & 0.9116  & 0.9803 & 0.9929 \\ \hline \hline
 ${\rm DICE}_{\Omega_0}$ &  0.8768  & 0.9080  &  0.9311  &  0.9494  & 0.9891 & 0.9962 \\ \hline 
  ${\rm DICE}_{\Omega_1}$ &  0.0278 & 0.1487  & 0.1764  &  0.1435 & 0.5673 & 0.7749 \\ \hline 
   ${\rm DICE}_{\Omega_2}$ & 0.5045  & 1.0000  & 1.0000  &  0.9999 & 0.9996 & 0.9991 \\ \hline 
\end{tabular}
\label{exp:real-image}
\end{table}

%-------------------------------------------------------------------
\section{Summary and Conclusions}\label{sec:conclusions}
%-------------------------------------------------------------------
In this paper, we unveiled a linkage between the PCMS model and the ROF model,
which is important to build the connection between image segmentation and image
restoration problems. The built linkage theoretically validates our proposed novel segmentation
methodology -- pursuing image segmentation by applying image restoration
plus thresholding.
This new segmentation methodology can circumvent the innate non-convex property of
the PCMS model, and thus improves the segmentation performance in both efficiency and effectiveness.
In particular, as a direct by-product of the built linkage, we proposed a segmentation method named
T-ROF method. The convergence of this method has also been proved.
Elaborate experimental results were presented which all support the excellent performance of
the proposed T-ROF method in terms of segmentation quality and computation time.
For the future work, considering using other thresholding rules, e.g. using the median instead of the mean in
the T-ROF method, may be worthwhile. Moreover, similar to the linkage built in this paper between the
PCMC and ROF models, investigating the relationship
between other models, such as those variants of the PCMS and ROF models
respectively in image segmentation and image restoration, will also be of interest.

%-------------------------------------------------------------------
\section*{Appendix}
%-------------------------------------------------------------------
The codebooks (${\zb m} := \{m_i\}_{i=0}^{K-1}$) computed by \cite{BEF84} and used for methods \cite{PCCB09,YBTB10} and our T-ROF method
(as initializations) in Examples 1-7 are listed in Table \ref{table-cb}.

\begin{table}[!h]
\centering \caption{The codebooks (${\zb m} := \{m_i\}_{i=0}^{K-1}$) computed by \cite{BEF84} and used for 
methods \cite{PCCB09,YBTB10} and our T-ROF method (as initializations) in Examples (Exa.) 1--7.}
\begin{tabular}{|l|l|}
\hline
  \multicolumn{1}{|c|}{Example}  & \multicolumn{1}{c|}{Codebook } \\ \hline
  {Exa. 1} (2-phase)  & $(0.0290, 0.7360)$  \\  \hline
  {Exa. 2} (2-phase)  & $(0.4081, 0.6045)$  \\  \hline
  {Exa. 3} (5-phase)  & $(0.0311, 0.3372, 0.5360, 0.7175, 0.9324)$  \\  \hline
  {Exa. 4} (4-phase) & $(0.0017, 0.3164, 0.6399, 0.8773)$  \\  \hline
 \multirow{1}{*}{{Exa. 5} (5-phase)} & (0.0756, 0.2878, 0.5033, 0.7153, 0.9238) \\ \hline  
  \multirow{2}{*}{{Exa. 5} (10-phase)}  & (0.0292, 0.1353, 0.2405, 0.3481, 0.4546, 0.5588, 0.6647, 0.7683,   \\ 
 	& \ \ 0.8708, 0.9715) \\  \hline
 \multirow{2}{*}{{Exa. 5} (15-phase)}  & (0.0143, 0.0842, 0.1560, 0.2254, 0.2959, 0.3661, 0.4360, 0.5053,  \\
 	& \ \ 0.5738, 0.6437, 0.7126, 0.7827, 0.8503, 0.9180, 0.9856) \\  \hline
  {Exa. 6} (3-phase) & $(0.2460, 0.3610, 0.5777)$  \\  \hline
  {Exa. 7} (4-phase) & $(0.0449, 0.2789, 0.6649, 0.9366)$  \\  \hline
\end{tabular}
\label{table-cb}
\end{table}

%%%%%%%%%%%%%%%%%%%%%

%\bibliographystyle{siam}
%\bibliography{myrefs_cai}

\begin{thebibliography}{10}

\bibitem{ACC05}
{\sc F.~Alter, V.~Caselles, and A.~Chambolle}, {\em A characterization of
 convex calibrable sets in {$R^N$}}, Mathematische Annalen, 332 (2005),
 pp.~329--366.

\bibitem{AFP00}
{\sc L.~Ambrosio, N.~Fusco, and D.~Pallara}, {\em Functions of Bounded
 Variation and Free Discontinuity Problems}, Oxford University Press, Oxford,
 2000.

\bibitem{AT90}
{\sc L.~Ambrosio and V.~Tortorelli}, {\em Approximation of functions depending
 on jumps by elliptic functionals via $t$-convergence}, Communications in Pure
 and Applied Mathematics, 43 (1990), pp.~999--1036.

\bibitem{ABM06}
{\sc H.~Attouch, G.~Buttazzo, and G.~Michaille}, {\em Variational analysis in
 Sobolev and BV Spaces}, SIAM, Philadelphia, 2006.

\bibitem{BCCJKM11}
{\sc L.~Bar, T.~Chan, G.~Chung, M.~Jung, N.~Kiryati, R.~Mohieddine, N.~Sochen,
 and L.~Vese}, {\em Mumford and shah model and its applications to image
 segmentation and image restoration}, in Handbook of Mathematical Imaging,
 Springer, 2011, pp.~1095--1157.
 
 
 \bibitem{BCPSS17}
{\sc B.~Bauer, X.~Cai, S.~Peth, K.~Schladitz, and G.~Steidl.} 
{\em Variational-based segmentation of biopores in tomographic images},
Computers \& Geosciences, 98 (2017), pp.~1--8.

\bibitem{BPV91}
{\sc G.~Bellettini, M.~Paolini, and C.~Verdi}, {\em Convex approximations of
 functionals with curvature}, 
 Natur. Rend. Lincei (9) Mat. Appl., 2 (1991), pp.~297--306.

\bibitem{BEF84}
{\sc J.~C. Bezdek, R.~Ehrlich, and W.~Full}, {\em {FCM: The fuzzy c-means
 clustering algorithm}}, Computers \& Geosciences, 10 (1984), pp.~191--203.

\bibitem{BZ87}
{\sc A.~Blake and A.~Zisserman}, {\em Visual Reconstruction}, {MIT} Press,
 Cambridge, {MA}, 1987.

\bibitem{BPCPE10}
{\sc S.~Boyd, N.~Parikh, E.~Chu, B.~Peleato, and J.~Eckstein}, {\em Distributed
 optimization and statistical learning via the alternating direction method of
 multipliers}, Foundations and Trends in Machine Learning, 3 (2010),
 pp.~1--122.

\bibitem{BEVTO07}
{\sc X.~Bresson, S.~Esedoglu, P.~Vandergheynst, J.~Thiran, and S.~Osher}, {\em
 Fast global minimization of the active contour/snake model}, Journal of
 Mathematical Imaging and Vision, 28 (2007), pp.~151--167.

\bibitem{BCB12}
{\sc E.~Brown, T.~Chan, and X.~Bresson}, {\em Completely convex formulation of
 the Chan-Vese image segmentation model}, International Journal of Computer
 Vision, 98 (2012), pp.~103--121.

\bibitem{C15}
{\sc X.~Cai}, {\em Variational image segmentation model coupled with image
 restoration achievements}, Pattern Recognition, 48 (2015), pp.~2029--2042.

\bibitem{CCNZ15}
{\sc X.~Cai, R.~Chan, M.~Nikolova, and T.~Zeng}, {\em A three-stage approach
 for segmenting degraded color images: smoothing, lifting and thresholding
 (slat)}, Journal of Scientific Computing, 72 (2017), pp.~1313--1332.

\bibitem{CCZ13}
{\sc X.~Cai, R.~Chan, and T.~Zeng}, {\em A two-stage image segmentation method
 using a convex variant of the mumford-shah model and thresholding}, SIAM
 Journal on Imaging Sciences, 6 (2013), pp.~368--390.

\bibitem{CFNSS15}
{\sc X. Cai, J. Fitschen, M. Nikolova, G. Steidl, and M. Storath}, 
{\em Disparity and optical flow partitioning using extended Potts priors}, 
Information and Inference: A Journal of the IMA, 4 (2015), pp.~43--62.

\bibitem{CS13}
{\sc X.~Cai and G.~Steidl}, {\em Multiclass segmentation by iterated rof
 thresholding}, in EMMCVPR, LNCS, Springer, 2013, pp.~237--250.

\bibitem{Ch04}
{\sc A.~Chambolle}, {\em An algorithm for total variation minimization and
 applications}, Journal of Mathematical Imaging and Vision, 2004,
 pp.~89--97.
 
 \bibitem{Ch05}
{\sc A.~Chambolle}, {\em Total variation minimization and a class of binary MRF models}, 
in EMMCVPR, LNCS, Springer, 2005, pp.~136--152.

 \bibitem{CNCP10}
{\sc A. Chambolle, V.~Caselles, M. Novaga, D. Cremers, and T. Pock},
{\em An introduction to total variation for image analysis},
in Theoretical Foundations and Numerical Methods for Sparse Recovery, 
Radon Ser. Comput. Appl. Math., 9 (2010), pp.~263--340.

\bibitem{CP11}
{\sc A.~Chambolle and T.~Pock}, {\em A first-order primal-dual algorithm for
 convex problems with applications to imaging}, Journal of Mathematical
 Imaging and Vision, 40 (2011), pp.~120--145.

\bibitem{CYZ13}
{\sc R.~Chan, H.~Yang, and T.~Zeng}, {\em A two-stage image segmentation method
 for blurry images with poisson or multiplicative gamma noise}, SIAM Journal
 on Imaging Sciences, 7 (2014), pp.~98--127.

\bibitem{CEN06}
{\sc T.~F. Chan, S.~Esedoglu, and M.~Nikolova}, {\em Algorithms for finding
 global minimizers of image segmentation and denoising models}, SIAM Journal
 on Applied Mathematics, 66 (2006), pp.~1632--1648.

\bibitem{CEPY06}
{\sc T. F. Chan, S. Esedoglu, F. Park, and A. Yip}, {\em Total variation image restoration:
Overview and recent developments}, in Handbook of Mathematical Models in Computer Vision, 
edited by N. Paragios, Y. Chen, and O. Faugeras, Springer-Verlag, New York, 2006, 
pp. 17--31.

\bibitem{CV01-2}
{\sc T.~F. Chan and L.~A. Vese}, {\em Active contours without edges}, IEEE
 Transactions on Image Processing, 10 (2001), pp.~266--277.

\bibitem{DCS10}
{\sc B.~Dong, A.~Chien, and Z.~Shen}, {\em Frame based segmentation for medical
 images}, Commun. Math. Sci., 32 (2010), pp.~1724--1739.

\bibitem{GG84}
{\sc S.~Geman and D.~Geman}, {\em Stochastic relaxation, {G}ibbs distributions,
 and the {B}ayesian restoration of images}, IEEE Transactions on Pattern
 Analysis and Machine Intelligence, 6 (1984), pp.~721--741.

\bibitem{GO09}
{\sc T.~Goldstein and S.~Osher}, {\em The split {B}regman method for
 l1-regularized problems}, SIAM Journal on Imaging Sciences, 2 (2009),
 pp.~323--343.

\bibitem{GFK07}
{\sc J.~Gorski, F.~Pfeiffer, and K.~Klamroth}, {\em Biconvex sets and
 optimization with biconvex functions - a survey and extensions}, Mathematical
 Methods of Optimization Research, 66 (2007), pp.~373--407.

\bibitem{HHMSS12}
{\sc Y.~He, M.~Y. Hussaini, J.~Ma, B.~Shafei, and G.~Steidl}, {\em A new fuzzy
 c-means method with total variation regularization for image segmentation of
 images with noisy and incomplete data}, Pattern Recognition, 45 (2012),
 pp.~3463--3471.

\bibitem{LS12}
{\sc J.~Lellmann and C.~Schn\"orr}, {\em Continuous multiclass labeling
 approaches and algorithms}, SIAM Journal on Imaging Sciences, 44 (2011),
 pp.~1049--1096.

\bibitem{LNZS10}
{\sc F.~Li, M.~Ng, T.~Zeng, and C.~Shen}, {\em A multiphase image segmentation
 method based on fuzzy region competition}, SIAM Journal on Imaging Sciences,
 3 (2010), pp.~277--299.

\bibitem{MS89}
{\sc D.~Mumford and J.~Shah}, {\em Optimal approximation by piecewise smooth
 functions and associated variational problems}, Communications on Pure and
 Applied Mathematics, XLII (1989), pp.~577--685.

\bibitem{NEC06}
{\sc M.~Nikolova, S.~Esedoglu, and T.~F. Chan}, {\em Algorithms for finding
 global minimizers of image segmentation and denoising models}, SIAM Journal
 on Applied Mathematics, 66 (2006), pp.~1632--1648.

\bibitem{PCCB09}
{\sc T.~Pock, A.~Chambolle, D.~Cremers, and H.~Bischof}, {\em A convex
 relaxation approach for computing minimal partitions}, IEEE Conference on
 Computer Vision and Pattern Recognition, 2009, pp.~810--817.

\bibitem{PCBC09a}
{\sc T.~Pock, D.~Cremers, H.~Bischof, and A.~Chambolle}, {\em An algorithm for
 minimizing the piecewise smooth mumford-shah functional}, in ICCV, 2009.

\bibitem{ROF92}
{\sc L.~I. Rudin, S.~Osher, and E.~Fatemi}, {\em Nonlinear total variation
 based noise removal algorithms}, Physica D, 60 (1992), pp.~259--268.

\bibitem{VC02}
{\sc L.~Vese and T.~Chan}, {\em A multiphase level set framework for image
 segmentation using the mumford and shah model}, International Journal of
 Computer Vision, 50 (2002), pp.~271--293.

\bibitem{YBTB10}
{\sc J.~Yuan, E.~Bae, X.-C. Tai, and Y.~Boykov}, {\em A continuous max-flow
 approach to potts model}, in European Conference on Computer Vision, 2010,
 pp.~379--392.

\bibitem{ZGFN08}
{\sc C.~Zach, D.~Gallup, J.-M.Frahm, and M.~Niethammer}, {\em Fast global
 labeling for real-time stereo using multiple plane sweeps}, Vision, Modeling,
 and Visualization Workshop, 2008.

\bibitem{ZMSM08}
{\sc Y.~Zhang, B.~Matuszewski, L.~Shark, and C.~Moore}, {\em Medical image
 segmentation using new hybrid level-set method}, in 2008 Fifth International
 Conference BioMedical Visualization: Information Visualization in Medical and
 Biomedical Informatics, 2008, pp.~71--76.

\end{thebibliography}

\end{document}